\documentclass[10pt,twoside]{amsart}

\usepackage{amssymb}
\usepackage{amscd}
\usepackage{amsmath}
\usepackage[usenames]{color}

\newtheorem{thm}{Theorem}[section]
\newtheorem{prop}[thm]{Proposition}
\newtheorem{lem}[thm]{Lemma}

\newtheorem{conjecture}[thm]{Conjecture}
\newtheorem{defn}[thm]{Definition}

\newtheorem{remark}[thm]{Remark}

\font\russ=wncyr10  1
\def\sha{\hbox{\russ\char88}}

\newcommand{\T}{\mathbb{ T}}
\newcommand{\cA}{\ifmmode {C_1}\else${C_1}$\ \fi}
\newcommand{\cB}{\ifmmode { C_2}\else${C_2}$\ \fi}
\newcommand{\cC}{\ifmmode { C_3}\else${C_3}$\ \fi}
\newcommand{\cD}{\ifmmode { D^\bullet}\else${D^\cdot}$\ \fi}
\newcommand{\cP}{\ifmmode {P^\bullet}\else${P^\cdot}$\ \fi}
\renewcommand{\L}{\ifmmode {\mathcal{L}}\else$\mathcal{L}$\ \fi}
\newcommand{\ca}{\ifmmode { A^\bullet_0}\else${A^\bullet_0}$\ \fi}
\newcommand{\be}{\begin{equation}}
\newcommand{\ee}{\end{equation}}
\renewcommand{\d}{\mathbf{d}}
\renewcommand{\u}{\mathbf{1}}
\newcommand{\id}{\mathrm{id}}
\newcommand{\bbC}{\ifmmode {\mathbb{C}}\else$\mathbb{C}$\ \fi}
\newcommand{\bbR}{\ifmmode {\mathbb{R}}\else$\mathbb{R}$\ \fi}

\renewcommand{\r}{\mathrm{R}\Gamma}

\newcommand{\bq}{\mathbb Q}
\newcommand{\Q}{\mathbb Q}

\newcommand{\bz}{\mathbb Z}
\newcommand{\br}{\mathbb R}
\newcommand{\bc}{\mathbb C}

\renewcommand{\H}{\mathrm{H}}

\newcommand{\M}{\mathfrak M}
\renewcommand{\u}{\mathbf{1}}

\newcommand{\tw}{\mathrm tw}

\newcommand{\zp}{{\bz_p}}
\newcommand{\qp}{{\bq_p}}
\newcommand{\cp}{\mathcal P}

\renewcommand{\O}{\mathcal{ O}}

\newcommand{\eins}{\boldsymbol{1}}
\newcommand{\C}{\mathcal{C}}
\newcommand{\La}{\ifmmode\Lambda\else$\Lambda$\fi}
\newcommand{\G}{\mathcal{G}}

\DeclareMathOperator{\Spec}{Spec}

\DeclareMathOperator{\Gal}{Gal} 

 \DeclareMathOperator{\Irr}{Irr}

\DeclareMathOperator{\Hom}{Hom}

 \DeclareMathOperator{\cok}{coker}
\DeclareMathOperator{\im}{im}

\renewcommand{\det}{\text{det}}

\def\YEAR{\year}\newcount\VOL\VOL=\YEAR\advance\VOL by-1995
\def\firstpage{1}\def\lastpage{1000}
\def\received{}\def\revised{}
\def\communicated{}

\makeatletter
\def\magnification{\afterassignment\m@g\count@}
\def\m@g{\mag=\count@\hsize6.5truein\vsize8.9truein\dimen\footins8truein}
\makeatother

\font\eightrm=cmr8
\font\caps=cmcsc10
\font\Caps=cmcsc10 scaled \magstep1   

\makeatletter
\@specialpagefalse\headheight=8.5pt
\def\DocMath{}
\renewcommand{\@evenhead}{%
    \ifnum\thepage>\lastpage\rlap{\thepage}\hfill%
    \else\rlap{\thepage}\slshape\leftmark\hfill{\caps\SAuthor}\hfill\fi}%
\renewcommand{\@oddhead}{%
    \ifnum\thepage=\firstpage{\DocMath\hfill\llap{\thepage}}%
    \else{\slshape\rightmark}\hfill{\caps\STitle}\hfill\llap{\thepage}\fi}%
\makeatother

\def\TSkip{\bigskip}
\newbox\TheTitle{\obeylines\gdef\GetTitle #1
\ShortTitle  #2
\SubTitle    #3
\Author      #4
\ShortAuthor #5
\EndTitle
{\setbox\TheTitle=\vbox{\baselineskip=20pt\let\par=\cr\obeylines%
\halign{\centerline{\Caps##}\cr\noalign{\medskip}\cr#1\cr}}%
        \copy\TheTitle\TSkip\TSkip%
\def\next{#2}\ifx\next\empty\gdef\STitle{#1}\else\gdef\STitle{#2}\fi%
\def\next{#3}\ifx\next\empty%
    \else\setbox\TheTitle=\vbox{\baselineskip=20pt\let\par=\cr\obeylines%
    \halign{\centerline{\caps##} #3\cr}}\copy\TheTitle\TSkip\TSkip\fi%
\centerline{\caps #4}\TSkip\TSkip%
\def\next{#5}\ifx\next\empty\gdef\SAuthor{#4}\else\gdef\SAuthor{#5}\fi%
\ifx\received\empty\relax
    \else\centerline{\eightrm Received: \received}\fi%
\ifx\revised\empty\TSkip%
    \else\centerline{\eightrm Revised: \revised}\TSkip\fi%
\ifx\communicated\empty\relax
    \else\centerline{\eightrm Communicated by \communicated}\fi\TSkip\TSkip%
\catcode'015=5}}\def\Title{\obeylines\GetTitle}
\def\Abstract{\begingroup\narrower
    \parskip=\medskipamount\parindent=0pt{\caps Abstract. }}
\def\EndAbstract{\par\endgroup\TSkip}

\long\def\MSC#1\EndMSC{\def\arg{#1}\ifx\arg\empty\relax\else
     {\par\narrower\noindent%
     1991 Mathematics Subject Classification: #1\par}\fi}

\long\def\KEY#1\EndKEY{\def\arg{#1}\ifx\arg\empty\relax\else
        {\par\narrower\noindent Keywords and Phrases: #1\par}\fi\TSkip}

\newbox\TheAdd\def\Addresses{\vfill\copy\TheAdd\vfill
    \ifodd\number\lastpage\vfill\eject\phantom{.}\vfill\eject\fi}
{\obeylines\gdef\GetAddress #1
\Address #2
\Address #3
\Address #4
\EndAddress
{\def\xs{4.3truecm}\parindent=0pt
\setbox0=\vtop{{\obeylines\hsize=\xs#1\par}}\def\next{#2}
\ifx\next\empty 
     \setbox\TheAdd=\hbox to\hsize{\hfill\copy0\hfill}
\else\setbox1=\vtop{{\obeylines\hsize=\xs#2\par}}\def\next{#3}
\ifx\next\empty 
     \setbox\TheAdd=\hbox to\hsize{\hfill\copy0\hfill\copy1\hfill}
\else\setbox2=\vtop{{\obeylines\hsize=\xs#3\par}}\def\next{#4}
\ifx\next\empty\ 
     \setbox\TheAdd=\vtop{\hbox to\hsize{\hfill\copy0\hfill\copy1\hfill}
                \vskip20pt\hbox to\hsize{\hfill\copy2\hfill}}
\else\setbox3=\vtop{{\obeylines\hsize=\xs#4\par}}
     \setbox\TheAdd=\vtop{\hbox to\hsize{\hfill\copy0\hfill\copy1\hfill}
                \vskip20pt\hbox to\hsize{\hfill\copy2\hfill\copy3\hfill}}
\fi\fi\fi\catcode'015=5}}\gdef\Address{\obeylines\GetAddress}

\begin{document}
\Title

On descent theory and main conjectures
in non-commutative Iwasawa theory

\ShortTitle
non-commutative main conjectures

\SubTitle
\Author David Burns and Otmar Venjakob

\ShortAuthor
\EndTitle

\centerline{\it Version of 24th October 2007}

\bigskip
\bigskip

\Abstract We prove a `Weierstrass Preparation Theorem' and develop
an explicit descent formalism in the context of Whitehead groups of
non-commutative Iwasawa algebras. We use these results to
 describe the precise connection between the main conjecture of
non-commutative Iwasawa theory (in the sense of Coates, Fukaya,
Kato, Sujatha and Venjakob) and the equivariant Tamagawa number
conjecture. The latter result is both a converse to a theorem of
Fukaya and Kato and also provides an important means of deriving
explicit consequences of the main conjecture. \EndAbstract \MSC
Primary
11G40; 
Secondary
11R65 
19A31 
19B28 
\EndMSC
\KEY
\EndKEY

\Address
King's College London,
Dept. of Mathematics,
London
WC2R2LS,
United Kingdom.

\Address
University of Heidelberg,
Mathematisches Institut,
69120 Heidelberg,
Germany.

\Address
\Address
\EndAddress

\section*{Introduction}

There has been much interest in the study of non-commutative
 Iwasawa theory over the last few years. Nevertheless, there is
 still no satisfactory understanding of the explicit consequences for Hasse-Weil $L$-functions that
 are implied by a `main conjecture' of the kind formulated by Coates, Fukaya, Kato, Sujatha and the second
named author in \cite{cfksv}. Indeed, whilst explicit consequences
of such a conjecture for the values (at $s =1$) of twisted
Hasse-Weil $L$-functions have been studied by
 Coates et al in \cite{cfksv}, by Kato in \cite{kato-k1} and by Dokchister and Dokchister
 in \cite{Doks}, all of these consequences become trivial whenever the $L$-functions vanish at
 $s=1$. Further, the conjecture of Birch and Swinnerton-Dyer implies that these $L$-functions should vanish whenever the
 relevant component of the Mordell-Weil group has strictly positive rank and recent work of
 Mazur and Rubin \cite{mazurrubin} shows that this is often the case. It
 is therefore important to understand what a main conjecture of the kind formulated in
 \cite{cfksv} predicts concerning the values of {\em derivatives} of Hasse-Weil $L$-functions at $s=1$.
   In this article we take the first step towards developing such a theory. Indeed, in a subsequent article it
    will be shown that one of the main results that we prove here (Theorem \ref{MCtoETNC}) can be combined with techniques developed by the
   first named author in \cite{burnsov} to derive from the main conjecture of non-commutative
   Iwasawa theory a variety of explicit (and highly non-trivial)
   congruence relations between values of derivatives of twisted Hasse-Weil $L$-functions at $s=1$.

However, in order to prove Theorem \ref{MCtoETNC} we must first
 develop several aspects of the theory that appear themselves to be of some independent
interest. These include proving a Weierstrass Preparation Theorem
for Whitehead groups of Iwasawa algebras (this result is related
to, but different from, the Weierstrass Preparation Theorem
discussed by the second named author in \cite{ven-wp}), defining a
canonical `characteristic series' for torsion modules over
(localised) Iwasawa algebras, satisfactorily resolving the descent
problem in non-commutative Iwasawa theory and formulating a main
conjecture in the spirit of Coates et al that deals with
interpolation properties of the `leading terms at Artin
representations' (in the sense introduced in \cite{BV}) of
analytic $p$-adic $L$-functions.

In a little more detail, the main contents of this article is as
follows. In \S\ref{prel} we recall some useful preliminaries
concerning localisation of Iwasawa algebras, $K$-theory, virtual
objects and derived categories. In \S\ref{mainres-PartI} we state
the main $K$-theoretical results that are proved in this article. In
\S\ref{sec2} we define a suitable notion of $\mu$-invariant and in
\S\ref{aplf} we combine this notion with a result of Schneider and
the second named author from \cite{sch-ven} and the
 formalism developed by Fukaya and Kato in \cite{fukaya-kato} to define canonical
 `characteristic series' in non-commutative Iwasawa theory
 (this construction extends the notion of `algebraic $p$-adic $L$-functions'
 introduced by the first named author in \cite{burns} and hence also refines the notion of `Akashi series'
 introduced by Coates, Schneider and
 Sujatha in \cite{css}). As a first application of these characteristic series
 we use them in \S\ref{elt}
  to prove an explicit formula for the `leading terms at Artin representations ' of elements of Whitehead groups of non-commutative
 Iwasawa algebras: this result provides a suitable `descent formalism' in non-commutative Iwasawa theory and
  in particular plays a crucial role in proving the arithmetic results discussed in the remainder
 of the article. In \S\ref{f-tp} we present a result of Kato that allows
 reduction to a convenient special case when formulating main conjectures. In \S\ref{ncmcs} we
 formulate explicit main conjectures of non-commutative Iwasawa theory for
 both Tate motives and (certain) critical motives. The approach here is
 finer than that of \cite{cfksv} since we consider interpolation properties for
 leading terms of analytic $p$-adic $L$-functions. In \S\ref{etncs} we combine the descent formalism described in \S\ref{elt} with
 the main results of \cite{BV} to prove that, under suitable
 hypotheses, the main conjectures formulated in
 \S\ref{ncmcs} imply the relevant special cases of the equivariant Tamagawa number
 conjecture formulated by Flach and the first named author in
 \cite[Conj.\ 4(iv)]{bufl01}. These results are a converse to the result of Fukaya and Kato in \cite{fukaya-kato} which asserts
that, under suitable hypotheses, the `non-commutative Tamagawa
number conjecture' of loc.\ cit.\ implies the main conjecture of
Coates et al \cite{cfksv} and can also be used to derive explicit
consequences of the main conjecture. Finally, in several appendices,
we review relevant aspects of the algebraic formalism of localized
$K_1$-groups and Bockstein homomorphisms and clarify certain
normalizations used in \cite{BV}.

The approach adopted here is a natural continuation of that used in
our earlier article \cite{BV}. In particular, in several places we
find it convenient to assume that the reader is familiar with
 the notation used in loc.\ cit.

We are very grateful indeed to Kazuya Kato for his generous help and
encouragement, to Manuel Breuning for advice regarding determinant
functors and to Jan Nekov\'a\v r for many stimulating discussions.
Much of this article was written when the first named author was a
 Visiting Professor at the Institute of Mathematics of the University
of Paris 6. He is extremely grateful to the Institute for this
opportunity.
\bigskip

\centerline{\bf Part I: $K$-theory}

\section{Preliminaries}\label{prel}

\subsection{Iwasawa algebras}\label{sec: iwasawa-alg} We fix a prime $p$. For any compact $p$-adic Lie group $G$
we write $\Lambda (G)$ and $\Omega(G)$ for the `Iwasawa algebras'
$\varprojlim_U \bz_p[G/U]$ and $\varprojlim_U \mathbb{F}_p[G/U]$
where $U$ runs over all open normal subgroups of $G$ and the limits
are taken with respect to the natural projection maps $\bz_p[G/U]
\to \bz_p[G/U']$ and $\mathbb{F}_p[G/U]\to \mathbb{F}_p[G/U']$ for
$U \subseteq U'$. The rings $\La (G)$ and $\Omega(G)$ are both
noetherian and, if $G$ has no element of order $p$, they are also
regular in the sense that their (left and right) global dimensions
are finite. We write $Q(G)$ for the total quotient ring of $\La
(G)$. If $\mathcal{O}$ is the valuation ring of a finite extension
of $\bq_p$, then we write $\Lambda_\mathcal{O}(G)$ for the
$\mathcal{O}$-Iwasawa algebra of $G$ and $Q_\mathcal{O}(G)$ for its
total quotient ring.

We assume throughout that the following condition is satisfied
\begin{itemize}
\item[$\bullet$] $G$ has a closed normal subgroup $H$ for which
the quotient group $\Gamma := G/H$ is isomorphic (topologically) to
the additive group of $\bz_p$.
\end{itemize}

We write $\pi_\Gamma: G \to \Gamma$ for the natural projection and
fix a topological generator $\gamma$ of $\Gamma$. We use $\gamma$ to
identify $\La (\Gamma) $ with the power series ring $\bz_p[[T]]$ in
an indeterminate $T$ (via the identification $T = \gamma - 1$).

We recall from \cite[\S2-\S3]{cfksv} that there are canonical left
and right denominator sets $S_{G,H}$ and $S_{G,H}^*$ of $\La (G)$
where
\[S_{G,H} := \{\lambda\in
\La (G) : \La(G)/\La(G)\cdot\lambda \mbox{ is a finitely generated
} \La (H)\mbox{-module}\}\]
and $S_{G,H}^* := \bigcup_{i\geq 0} p^iS_{G,H}.$
When $G$ and $H$ are clear from context we often abbreviate
$S_{G,H}$ to $S$. We also write $\mathfrak{M}_S(G)$ and
$\mathfrak{M}_{S^*}(G)$ for the categories of finitely generated
$\La (G)$-modules $M$ with $\La(G)_S\otimes_{\La (G)}M = 0$ and
$\La(G)_{S^*}\otimes_{\La (G)}M = 0$ respectively.

For any $\bz_p$-module $M$ we write $M_{\rm tor}$ for its
$\bz_p$-torsion submodule and set $M_{\rm tf} := M/M_{\rm tor}$. We
recall from \cite[Prop. 2.3]{cfksv} that a finitely generated
$\La(G)$-module $M$ belongs to $\mathfrak{M}_S(G)$, resp.
$\mathfrak{M}_{S^*}(G)$, if and only if it is a finitely generated
$\La (H)$-module (by restriction), resp. when $M_{\rm tf}$ belongs
to $\mathfrak{M}_S(G)$. This means in particular that
$\mathfrak{M}_{S^*}(G)$ coincides with the category
$\mathfrak{M}_H(G)$ introduced in loc. cit.

\subsection{$K$-groups}\label{K-group-defs} If $\Sigma$ denotes either $S$ or $S^*$, then we write
$K_0(\La (G),\La(G)_\Sigma)$ for the relative algebraic $K_0$-group
of the homomorphism $\La (G)\to \La(G)_\Sigma$. We recall that this
group is generated by symbols of the form $(P,\lambda, Q)$ where $P$
and $Q$ are finitely generated projective (left) $\Lambda
(G)$-modules and $\lambda$ is an isomorphism of
$\La(G)_{\Sigma}$-modules $\La(G)_{\Sigma}\otimes_{\La (G)}P\to
\La(G)_{\Sigma}\otimes_{\La (G)}Q$ (for more details see \cite[p.
215]{swan}). In addition, if $G$ has no element of order $p$, then
$K_0(\La (G),\La(G)_\Sigma)$ can be identified with the Grothendieck
group $K_0(\mathfrak{M}_{\Sigma}(G))$ of the category
$\mathfrak{M}_\Sigma(G)$. To be precise we normalise this
isomorphism in the following way: if $g = s^{-1}h$ with both $s\in
\Sigma$ and $h \in \Lambda(G) \cap \La (G)_\Sigma^\times,$ then the
element $(\La (G),{\rm r}_g, \La (G))$ of $K_0(\La
(G),\La(G)_\Sigma)$, where ${\rm r}_g$ denotes right multiplication
by $g$, corresponds to $[\mathrm{coker}({\rm
r}_h)]-[\mathrm{coker}({\rm r}_s)]$ in
$K_0(\mathfrak{M}_\Sigma(G))$. There is a natural commutative
diagram of long exact sequences
\begin{equation}\label{leskt}\begin{CD} K_1(\La (G)) @> >> K_1(\La (G)_{S}) @>
\partial_G >> K_0(\La (G),\La(G)_S) @> >> K_0(\La
(G))\\
@\vert @V VV @V VV @\vert\\
K_1(\La (G)) @> >> K_1(\La (G)_{S^*}) @>
\partial_G >> K_0(\La (G),\La(G)_{S^*}) @> >> K_0(\La
(G)).
\end{CD}\end{equation}
If $G$ has no element of order $p$, then the lower row of
 this diagram identifies with the exact sequence \cite[(24)]{cfksv}.

\subsection{Derived categories} For any ring $R$ we write $D(R)$ for the derived
category of $R$-modules. We also write $D^{{\rm fg},-}(R)$, resp.
 $D^{\rm fg}(R)$ for the full triangulated subcategory of $D(R)$
comprising complexes that are isomorphic to a bounded above, resp.
bounded, complex of finitely generated $R$-modules and we let
$D^{\rm p}(R)$ denote the full subcategory of $D^{\rm fg}(R)$
comprising complexes that are isomorphic to an object of the
category $C^{\rm p}(R)$ of bounded complexes of finitely generated
projective $R$-modules.

If $\Sigma \in \{S,S^*\}$, then we write $D^{\rm p}_\Sigma(\La(G))$
for the full triangulated subcategory of $D^{\rm p}(\La(G))$
comprising those complexes $C$ such that
$\La(G)_\Sigma\otimes_{\La(G)}C$ is acyclic. We recall that to each
 $C$ in $D_{S^*}^{\rm p}(\La (G))$ one can associate a
canonical `Euler characteristic' element $[C]$ in $K_0(\La
(G),\La(G)_{S^*})$ (see, for example, \cite[beginning of
\S4]{burns}).

\subsection{Virtual objects}\label{vokt} For any associative unital ring $R$ we write $V(R)$ for the Picard category of virtual
objects associated to the category of finitely generated
projective $R$-modules and we fix a unit object
 ${\bf 1}_R$ in $V(R)$.

Let $\cp_0$ be the Picard category with unique object
$\eins_{\cp_0}$ and $\text{Aut}_{\cp_0}(\eins_{\cp_0}) = 0$. For
 any finite group $\G$ and extension field $F$ of
$\bq_p$ we define $V(\bz_p[\G],F[\G])$ to be the fibre product
category in the diagram
\begin{equation*}\begin{CD}
 V(\bz_p[\G],F[\G]) @>>> \cp_0\\ @VVV @VV{F_2}V\\
 V(\bz_p[\G]) @>F_1>> V(F[\G])
\end{CD}\end{equation*}
where $F_2$ is the (monoidal) functor sending $\eins_{\cp_0}$ to
$\eins_{F[\G]}$ and $F_1(L)= F[\G]\otimes_{\bz_p[\G]}L$ for each
object $L$ of $V(\bz_p[\G])$. We write $K_0(\bz_p[\G],F[\G])$
 for the relative algebraic $K_0$-group of the inclusion $\bz_p[\G]\subset F[\G]$ and
 regard the canonical isomorphism
\be\label{caniso} \pi_0(V(\bz_p[\G],F[\G])) \cong
K_0(\bz_p[\G],F[\G])\ee
of \cite[Prop. 2.5]{bufl01} as an identification. In particular, for
each object $L$ of $V(\bz_p[\G])$ and each
 morphism $\mu: F_1(L) \to \eins_{F[\G]}$ in $V(F[\G])$ we write
  $[L,\mu]$ for the associated element of
  $K_0(\bz_p[\G],F[\G]).$

\subsection{Wedderburn decompositions} We fix an algebraic closure $\bq_p^c$ of $\bq_p$. For any finite group $\G$
we write ${\rm Irr}(\G)$ for the set of irreducible
finite-dimensional $\bq^c_p$-valued characters of $\G$. Then the
Wedderburn decomposition of the (finite dimensional semisimple)
$\bq^c_p$-algebra $\bq^c_p[\G]$ induces a decomposition of its
centre
\begin{equation}\label{centre-decomp} \zeta(\bq^c_p[\G]) \cong \prod_{{\rm Irr}(\G)}\bq^c_p.\end{equation}
%
The natural reduced norm map ${\rm Nrd}_{\bq^c_p[\G]}:
K_1(\bq^c_p[\G]) \to \zeta(\bq^c_p[\G])^\times$ is bijective and we
often (and without explicit comment) combine this map with
(\ref{centre-decomp}) to regard elements of $\prod_{{\rm
Irr}(\G)}(\bq^c_p)^\times$ as elements of the Whitehead group
$K_1(\bq^c_p[\G])$. In particular, we write $\partial_\G:
\prod_{{\rm Irr}(\G)}(\bq_p^c)^\times \to
K_0(\bz_p[\G],\bq^c_p[\G])$ for the connecting homomorphism of
relative $K$-theory (normalized as in (\ref{leskt})).

For each $\rho\in\mathrm{Irr}(\G)$ we fix a minimal idempotent
$e_\rho$ in $\bq_p^c[\G]$ such that the
 left action of $\G$ on $\bq_p^c[\G]$
given by $x\mapsto xg^{-1}$ for $g\in \G$ induces an isomorphism of
(left) $\bq_p^c[\G]$-modules
 $e_\rho\bq_p^c[\G]\cong V_{\rho^*}$ where
$V_{\rho^*}\cong (\bq_p^c)^{n_\rho}$ is the representation space of
the contragredient $\rho^*$ of $\rho$ over $\bq_p^c.$ Then for each
complex $C$ in $D^{\rm p}(\bz_p[ \G])$ the theory of Morita
equivalence induces an identification of morphism groups
\begin{multline}\label{m-e-decomp}
{\rm Mor}_{V(\bq_p^c[\G ])}(\d_{\bq_p^c[\G ]}(\bq^c_p[\G
]\otimes_{\bz_p[\G ]}^{\mathbb{L}} C),\eins_{\bq^c_p[\G ]})\\\cong
\prod_{{\rm Irr}(\G )}{\rm
Mor}_{V(\bq^c)}(\d_{\bq^c_p}(e_\rho\bq^c_p[\G ]\otimes_{\bz_p[\G
]}^{\mathbb{L}}C ),\eins_{\bq^c_p}).
\end{multline}
Details of the `non-commutative determinants' $\d_{\bq^c_p[\G]}(-)$
and $\d_{\bq^c_p}(-)$ that are used here are recalled in Appendix A.

\section{Statement of the main results in Part I}\label{mainres-PartI}

The first main result we prove in Part I is the following
 decomposition theorem for Whitehead groups. 

\begin{thm}\label{nc-weier} If $G$ has no element of order $p$, then there is a natural isomorphism of abelian groups
\[ K_0(\Omega (G)) \oplus K_0(\mathfrak{M}_S(G))\oplus {\rm
im}(K_1({\La}(G))\to K_1({\La}(G)_{S^*})) \cong
K_1({\La}(G)_{S^*}).\]
\end{thm}

Our proof of Theorem \ref{nc-weier} will show that if $G =\Gamma$,
then the above isomorphism reduces to the assertion that every
element of $Q(\Gamma)^\times$ can be written uniquely in the from
$p^mdu$ where $m$ is an integer, $d$ is a quotient of
distinguished polynomials and $u$ a unit in $\La (\Gamma)$ (see
Remark \ref{G=Gamma} and \S\ref{pfof2.1}). Theorem \ref{nc-weier}
is therefore a natural generalisation of the classical Weierstrass
Preparation Theorem. (For an alternative approach to generalising
the latter result see \cite{ven-wp}).

In the remainder of Part I we use Theorem \ref{nc-weier} to resolve
the `descent problem' in non-commutative Iwasawa theory. Before
stating our main result in this regard we recall that for each Artin
representation $\rho: G \to {\rm GL}_n(\mathcal{O})$ the ring
homomorphism $\La(G)_{S^*}\to M_n(Q(\Gamma))$ that sends each
element $g$ of $G$ to $\rho(g)\pi_\Gamma(g)$ induces a homomorphism
of groups
\begin{equation}\label{ltdef} \Phi_\rho: K_1(\La (G)_{S^*})\to K_1(M_n(Q_\O(\Gamma)))\cong
K_1(Q_\O(\Gamma)) \cong Q_\O(\Gamma)^\times \cong Q(\O [[T]])^\times
\end{equation}
where the first isomorphism is induced by the theory of Morita
equivalence, the second by taking determinants (over
$Q_\O(\Gamma)$) and the third by the identification $\gamma -1 =
T$. The `leading term' $\xi^*(\rho)$ at $\rho$ of an element $\xi$
of $K_1(\La (G)_{S^*})$ is then defined to be the leading term
 at $T= 0$ of the power series $\Phi_\rho(\xi)$ (this definition can also be interpreted as a leading
 term at zero of a $p$-adic meromorphic function - see
  \cite[Lem. 3.17]{BV}).

The problem of descent in (non-commutative) Iwasawa theory is then
the following: given an element $\xi$ of $K_1(\La (G)_{S^*})$ and a
finite quotient $\overline{G}$ of $G$, can one use knowledge of the
image of $\xi$ under the connecting homomorphism $\partial_G$ to
give an explicit formula for the image of $(\xi^*(\rho))_{\rho \in
{\rm Irr}(\overline{G})}$ under the connecting homomorphism
$\partial_{\overline{G}}?$ This is known to be an important and
delicate problem. The decomposition of Theorem \ref{nc-weier} plays
a key role in our proof of the following (partial) resolution. We
set
\[ \tilde S := \begin{cases} S, &\text{if $G$ has an element of
order $p$},\\
 S^*, &\text{otherwise.}\end{cases}\]

\begin{thm}\label{lt-result} Let $\overline{G}$ be a finite quotient of $G$. Let $\xi$ be an element
of $K_1(\La (G)_{S^*})$ with $\partial_G(\xi) = [C]$ where $C$ is a complex that belongs to
$D_{\tilde S}^{\rm p}(\La (G))$ and is `semisimple at $\rho$' for each $\rho\in {\rm
Irr}(\overline{G})$ in the sense of \cite[Def. 3.11]{BV}. Then in
$K_0(\bz_p[\overline{G}],\bq^c_p[\overline{G}])$ one has
\[ \partial_{\overline{G}}((\xi^*(\rho))_{\rho\in {\rm Irr}(\overline{G})}) =
 -[\d_{\bz_p[\overline{G}]}(\bz_p[\overline{G}]\otimes_{\La
(G)}^{\mathbb{L}}C),t(C)_{\overline{G}}]\]
with $t (C)_{\overline{G}}$ the morphism
$\d_{\bq^c_p[\overline{G}]}(\bq^c_p[\overline{G}]\otimes_{\La
(G)}^{\mathbb{L}}C) \to \eins_{\bq^c_p[\overline{G}]}$ that
corresponds via (\ref{m-e-decomp}) to
$((-1)^{r_G(C)(\rho)}t(C_{\rho}))_{\rho\in {\rm Irr}(\overline{G})}$
where $r_G(C)(\rho)$ is the integer defined in \cite[Def. 3.11]{BV}
and $t(C_{\rho})$ the morphism
$\d_{\bq^c_p}(e_\rho\bq^c_p[\overline{G}]\otimes_{\La
(G)}^\mathbb{L} C)\to \u_{\bq^c_p}$ defined in \cite[Lem.
3.13(iv)]{BV}.
\end{thm}

\begin{remark}{\em The hypothesis of `semisimplicity at $\rho$' and the definitions of $r_G(C)(\rho)$ and
$t(A_{\rho})$ are recalled explicitly in \S\ref{twisting-semi}. However, in the important case that
$e_\rho\bq^c_p[\overline{G}]\otimes_{\La (G)}^\mathbb{L} C$ is acyclic one knows that $C$ is
automatically semisimple at $\rho$, $r_G(C)(\rho) =0$ and $t(C_{\rho})$ is simply the canonical
morphism induced by property A.e) of the determinant functor described in Appendix A. In
particular, if $G = \Gamma$, $C = M[0]$ for a finitely generated torsion $\La (\Gamma)$-module $M$
for which both $M^\Gamma$ and $M_\Gamma$ are finite and $\rho$ is the trivial character, then the
equality of Theorem \ref{lt-result} is equivalent to the classical descent formula discussed in
\cite[p.\ 318, Ex.\ 13.12]{wash}. Upon appropriate specialisation, Theorem \ref{lt-result} also
recovers
 the descent formalism proved (in certain commutative cases) by
Greither and the first named author in
 \cite[\S8]{burns-greither2003} and is therefore related to the earlier (commutative) work of
 Nekov\'a\v r in \cite[\S11]{nek}. }
\end{remark}

In \S\ref{f-tp} we will prove that it suffices to consider main
conjectures of non-commutative Iwasawa theory in the case that $G$
has no element of order $p$. Theorem \ref{lt-result} therefore
represents a satisfactory resolution of the descent problem in this
context. Indeed, in Part II (\S\ref{f-tp} - \S\ref{etncs}) of this
article we shall combine Theorem \ref{lt-result} with the main
results of \cite{BV} to describe the precise connection between main
conjectures of non-commutative Iwasawa theory (in the sense of
Coates et al \cite{cfksv}) and the appropriate case of the
equivariant Tamagawa number conjecture. Other important applications
of Theorem \ref{lt-result} are described in \cite{mcrc}.

\section{Generalized $\mu$-invariants}\label{sec2}

The key ingredient in our proof of Theorem \ref{nc-weier} is the construction of canonical
`characteristic series' in non-commutative Iwasawa theory. In this section we prepare for this
construction by generalising the classical notion of $\mu$-invariant.

\subsection{The definition}\label{gen-mu} In the
sequel we write $\mu_\Gamma(M)$ for the `$\mu$-invariant' of a
finitely generated $\La(\Gamma)$-module $M.$ For each complex $C$ in
$D^{\rm p}(\Lambda(\Gamma))$ we also set
\[\mu_\Gamma(C):=\sum_{i\in\bz} (-1)^i\mu_\Gamma(H^i(C)).\]

Let $\rho: G\to GL_n(\O)$ be a continuous representation of $G$ and
write $E_\rho\cong\O^n$ for the associated representation module,
where $\O=\O_L$ denotes the ring of integers of a finite extension
$L$ of $\qp.$ We denote the corresponding $L$-linear
 representation $L\otimes_\O E_\rho$ by $V_\rho$. By
 $\bar{\rho}$ we denote the reduction of $\rho$ modulo $\pi,$ the
 uniformising element of $\O,$ with representation space
 $\overline{E_\rho}.$ The residue class field of $\O$ is denoted by
 $\kappa.$

For each $C$ in $D^{\rm p}(\La(G))$ we set
$\,C(\rho^*):=\O^n\otimes_{\zp}C$, regarded an an object in
 $D^{\rm p}(\La_\O(G))$ via the action $g(x\otimes_{\bz_p}c^i) =
 \rho^*(g)(x)\otimes_{\bz_p}g(c^i)$ for each $g$ in $G$, $x$ in
 $\O^n$ and $c^i$ in $C^i$. We then set
\[{C_\rho}:=(\La_\O(\Gamma)\otimes_\O\O^n)\otimes^\mathbb{L}_{\La(G)}C\cong
\La_\O(\Gamma)\otimes^\mathbb{L}_{\La_\O(G)}C(\rho^*)\]
and also
\begin{equation}\label{mu-def}\mu(C,\rho):=  \mu_\Gamma(C_\rho)\in
\mathbb{Z}.\end{equation}

\subsection{Basic properties}

\begin{lem}\label{additiv} Fix a continuous representation $\rho:
G\to GL_n(\O)$.
\begin{itemize}


\item[(i)] If $\cA\to\cB\to\cC\to\cA[1]$ is an exact triangle in
$D^{\rm p}_{S^*}(\La(G))$, then
\[\mu(\cB,\rho)=\mu(\cA,\rho)+\mu(\cC,\rho).\]

\item[(ii)] If $C\in D_{S^*}^{\rm p}(\La(G))$ is
cohomologically perfect, then
\[ \mu(C,\rho) = \mu(\H(C),\rho) \]
where $\H(C)$ denotes the complex with zero differentials and
$\H(C)^i=H^i(C)$ for all $i.$

\item[(iii)] If $C \in D^{\rm p}_S(\La (G))$ is cohomologically
perfect, then $\mu(C,\rho) = 0$.

\item[(iv)] If $U$ is any closed normal subgroup of $G$ such that
$U \subseteq H\cap \ker(\rho)$, then for any $C$ in $D^{\rm p}(\La
(G))$ we have
\[\mu(C,\rho) = \mu(\La(G/U)\otimes^{\mathbb{L}}_{\La (G)}C,
\rho). \]
Here the first $\mu$-invariant is formed with respect to the group $G$ and the second with respect
to $G/U.$
\item[(v)] If $U$ is any open subgroup of $G,$ then for any continuous
representation $\psi:U\to GL_n(\O)$ and any $C\in D^{\rm p}(\La(G))$   one has
\[\mu(C,\mathrm{Ind}_U^G\psi)=\mu(\mathrm{Res}^{G}_{U} C, \psi)\]
where the first $\mu$-invariant is formed with respect to $G$ and the second with respect to $U.$
Here $\mathrm{Res}^{G}_{U}$ denotes the restriction functor from $\La(G)$- to $\La(U)$-modules.
\end{itemize}
\end{lem}

\begin{proof} For each $D$ in $D^{\rm p}_{S^*}(\La(G))$
 all of the $\Lambda (\Gamma)$-modules $H^i(D_\rho)$ are both finitely generated and torsion.
 Since $\mu_\Gamma(-)$ is additive on exact sequences of finitely generated torsion $\La(\Gamma)$-modules,
 claim (i) therefore follows from the long exact sequence of
cohomology of the exact triangle $\,C_{1,\rho}\to C_{2,\rho}\to
C_{3,\rho}\to C_{1,\rho}[1]$ in $D^{\rm p}(\La (\Gamma))$ that is
induced by the given triangle.

If $C\cong H^i(C)[i]$ for some $i,$ then claim (ii) is clear.
 The general case can then be proved by induction with respect to
the cohomological length: indeed, one need only combine claim (i)
together with the exact triangles given by (good) truncation.

In order to prove claim (iii) it is sufficient by claim (ii) to
consider the case $C =M[0]$ with $M$ a $\La(G)$-module that is
finitely generated over $\La(H).$ But then $H^i(C_\rho)$ is a
finitely generated $\zp$-module for all $i\in\mathbb{Z}$ and so
 it is clear that $\mu(C,\rho)=0$.

In the situation of claim (iv) there is a canonical isomorphism of
$\La_\O(G/U)$-modules
\[\La_\O(G/U)\otimes_{\La_\O(G)}^\mathbb{L}
C(\rho^*)\cong C(\rho^*)_U\cong C_U(\rho^*)\cong
\big(\La(G/U)\otimes^\mathbb{L}_{\La(G)}C\big)(\rho^*),\]
from which the claim follows immediately. Similarly, in the
situation of claim (v) we have a canonical isomorphism
 $\mathrm{Ind}_U^G \big((\mathrm{Res}^{G}_{U} C)(\psi^*)\big)\cong
C(\mathrm{Ind}_U^G\psi^*)$ which corresponds to
\begin{eqnarray*}
{\La_\O(G)}\otimes_{\La_\O(U)}\big(\O^n\otimes_\zp\mathrm{Res}^{G}_{U} C\big)&\cong&
(\La(G)\otimes_{\La(U)}\O^n)\otimes_\zp C, \\ g\otimes (a\otimes c) &\mapsto & (g\otimes a)\otimes
gc.
\end{eqnarray*}
Now we write $\Gamma_U$ for the image of $U$ in $\Gamma$ under the natural projection and obtain
\begin{eqnarray*}
\mu(C,\mathrm{Ind}_U^G\psi)&=& \mu_\Gamma( \La_\O(\Gamma)\otimes_{\La_\O(G)}^\mathbb{L}
C(\mathrm{Ind}_U^G\psi^*))\\ &=&
\mu_\Gamma(\La_\O(\Gamma)\otimes_{\La_\O(G)}^\mathbb{L}\mathrm{Ind}_U^G \big((\mathrm{Res}^{G}_{U}
C)(\psi^*)\big))\\
&=&\mu_\Gamma(\La_\O(\Gamma)\otimes_{\La_\O(\Gamma_U)}\big(
\La_\O(\Gamma_U)\otimes_{\La_\O(U)}^\mathbb{L}(\mathrm{Res}^{G}_{U}C)(\psi^*)\big))\\
&=&\mu_{\Gamma_U}(\La_\O(\Gamma_U)\otimes_{\La_\O(U)}^\mathbb{L}(\mathrm{Res}^{G}_{U}C)(\psi^*))\\
&=&\mu(\mathrm{Res}^{G}_{U} C, \psi)
\end{eqnarray*}
as had to be shown.
\end{proof}

\subsection{Module theory}\label{module theory} In order to make a closer examination of the $\mu$-invariant defined in (\ref{mu-def})
we recall some standard module theory.


We write $\mathrm{Jac}(\La(G))$ for the Jacobson radical of
$\La(G)$ and $\prod_{i\in I} A_i$ for the Wedderburn decomposition
of the finite dimensional semisimple $\mathbb{F}_p$-algebra
$A:=A(G):=\La(G)/\mathrm{Jac}(\La(G))$ (so $I$ is finite). Let
$R_i=Aa_i$ be a representative for the unique isomorphism class of
simple $A_i$-modules with $a_i$ some orthogonal primitive
idempotent of $A(G), $  always assuming that
$A_1=\mathbb{F}_p=R_1;$ the corresponding representations we
denote by $\psi_i:G\to GL(R_i)$ for $i$ in $I.$ For each index $i$
we fix an idempotent $e_i$ of $\La(G)$ which is a pre-image of
$a_i$ under the projection
$\La(G)\twoheadrightarrow\La(G)/\mathrm{Jac}(\La(G)).$

We consider the projective $\La(G)$-modules $X_i:=\La(G)e_i$ and
projective $\Omega(G)$-modules $Y_i:=X_i/pX_i$. They are projective
hulls of $R_i$ since $A(G)\otimes_{\La(G)} X_i=
A(G)\otimes_{\Omega(G)} Y_i$ $ =R_i.$ Every finitely generated
projective $\La(G)$-module $X$, resp. $\Omega(G)$-module $Y$,
decomposes in a unique way as a direct sum
\[X=\bigoplus_{i\in I} X_i^{\langle X,X_i\rangle}, \,\,\,\text{resp.  }\,\, Y=\bigoplus_{i\in I} Y_i^{\langle Y,Y_i\rangle},\]
for suitable natural numbers $\langle X,X_i\rangle$, resp. $\langle
Y,Y_i\rangle$. We write $l_{R_i}(\psi)$ for the multiplicity of the
occurrence of $R_i$ in a $\mathbb{F}_p$-linear representation $\psi$
and
\[\chi(G,M):=\prod_i |\mathrm{H}_i(G,M)|^{(-1)^i},\] if this is finite, for the Euler-Poincar\'{e}-characteristic of a
$\La(G)$-module $M.$

\begin{lem}\label{mu-proj}
Let $Y$ be a finitely generated projective $\Omega(G)$-module.
\begin{enumerate}
\item[(i)] $\La(\Gamma)\otimes_{\La(G)}
Y$ is naturally isomorphic to $\Omega(\Gamma)\otimes_{\Omega(G)}
Y=\Omega(\Gamma)^{<Y,Y_1>}$ and thus $\chi(G,Y)=p^{\langle
Y,Y_1\rangle}.$ \item[(ii)] $\langle Y(\psi^*),Y_1\rangle=\sum_{i\in
I}l_{R_i}(\psi)\dim_{\mathbb{F}_p}(\mathrm{End}_{\Omega(G)}(R_i))\langle
Y,Y_i\rangle.$
\end{enumerate}
\end{lem}

\begin{proof} For each index $i$ the module
$\Omega(\Gamma)\otimes_{\Omega(G)}Y_i$ is isomorphic to
$\Omega(\Gamma)^{n_i}$ for some natural number $n_i.$ Since then
$\mathbb{F}_p^{n_i}\cong\mathbb{F}_p\otimes_{\Omega(G)}Y_i\cong
A_1\otimes_{\Omega(G)}Y_i=a_1 R_i $ is isomorphic to $R_1$ if
$i=1$ and is zero if $i \not= 1$, the first claim follows.

Claim (ii) is true because
$\dim_{\mathbb{F}_p}(\mathrm{End}_{\Omega(G)}(R_i))\langle
Y,Y_i\rangle =\dim_{\mathbb{F}_p}(\Hom_{\Omega(G)}(Y,R_i))$ and
$\Hom_{\Omega(G)}(Y(\psi^*_i),R_1)$ is isomorphic to
$\Hom_{\Omega(G)}(Y
 ,R_i)$ (cf.\ \cite[Prop.\ 4.1, Lem.\ 4.4]{arwa}).
\end{proof}

\subsection{The regular case} In this section we consider the constructions of
\S\ref{gen-mu} in the case that $G$ has no element of order $p$.

\subsubsection{Pairings} For a field $K$ we write $R_K(G)$ for the Grothendieck group of the
category of finite-dimensional continuous $K$-linear representations
of $G$ which have finite image. The tensor product induces a
structure of rings on both $R_L(G)$ and $R_\kappa(G)$ and there
exists a canonical surjective homomorphisms of rings
 $\,R_L(G)\twoheadrightarrow R_\kappa(G)$ that is induced by reducing modulo $\pi$ any $G$-stable
$\O$-lattice $E$ of an representation $V$ of the above type (cf.\
\cite{ser-rep}).

\begin{prop}\label{pairing} Assume that $G$ has no element of order $p$.
\begin{itemize}
\item[(i)] If $C\in D^{\rm p}_{S^*}(\La(G))$, then for each continuous representation $\rho:
G\to GL_n(\O)$ one has
\[ \mu(C,\rho)= \sum_{i\in\bz}
(-1)^i\mu_\Gamma\big(\La(\Gamma)\otimes_{\La(G)}^\mathbb{L}(\overline{E_{\rho}}^*\otimes_{\mathbb{F}_p}
{\rm gr}(H^i(C)_\mathrm{tor}) [0])\big)\]
where $\overline{E_{\rho}}^*$ denotes the contragredient module
$\Hom_\kappa(\overline{E_{\rho}},\kappa)$ while for a
$\mathbb{Z}_p$-module $M$ endowed with the $p$-adic filtration we
denote by ${\rm gr}(M)$ the associated graded $\mathbb{F}_p$-module.

\item[(ii)] The $\mu$-invariant induces a $\bz$-bilinear pairing
\[\mu(-,-): K_0(D_{S^*}^{\rm p}(\La(G)))\times R_L(G) \to \bz .\]
This pairing induces a $\bz$-bilinear pairing of the form
\[\mu(-,-):K_0(D_{S^*}^\mathrm{p}(\La(G)))\times R_\kappa(G) \to \bz.\]
\item[(iii)] If $C\in D^{\rm p}_{S^*}(\La(G))$ and $i \in I$, then the integer
$\mu(C,\psi_i)$ defined by claim (ii) is divisible by
 $\dim_{\mathbb{F}_p}(\mathrm{End}_{\Omega(G)}(R_i)).$
\item[(iv)] If $C\in D^{\rm p}_{S}(\La(G))$, then $\mu(C,\psi_i) =0$
for all $i\in I$.

\end{itemize}
\end{prop}

\begin{proof} To prove claim (i), we write $\mu'(C,\bar{\rho})$ for the term on the right hand side of the claimed equality. Since
$\mu'(C,\bar{\rho})=\mu'(\H(C),\bar{\rho})$ by definition and
$\mu(C,{\rho})=\mu(\H(C),{\rho})$ by Lemma \ref{additiv} (ii) we
need only consider the case where $C\cong M[0]$ with $M$ in
$\mathfrak{M}_{S^*}(G)$. Further, since both $\mu$-invariants are
additive on exact triangles (cf. Lemma \ref{additiv}(i)), it is
actually sufficient to prove the following two special cases
(note that $M/M_\mathrm{tor}$ belongs to $\mathfrak{M}_S(G)$ for all $M$ in $\mathfrak{M}_{S*}(G)$):\\
1.) If $M$ is in $\mathfrak{M}_S(G),$ then both
$H^i(\La(\Gamma)\otimes_{\La(G)}^\mathbb{L}(\overline{E_\rho}\otimes_{\mathbb{F}_p}
{\rm gr}(M_\mathrm{tor}) [0]))$ and $H^i(A_\rho)$  are finitely
generated $\bz_p$-modules and thus
$\mu(C,\rho)=0=\mu'(C,\bar{\rho}).$\\
2.) If $p^nM=0$ for some $n,$ we argue by induction on $n.$ For
$n=1$ the isomorphism $\overline{E_\rho}^*\otimes_{\mathbb{F}_p}
{\rm gr}
(M_\mathrm{tor})\cong\overline{E_\rho}^*\otimes_{\mathbb{F}_p}
M\cong \overline{E_\rho}^*\otimes_{\bz_p} M$ implies the equality of
the $\mu$-invariants. For $n>1$ one uses d\'evissage and again the
additivity of both $\mu$-invariants.

To prove the existence of the first pairing in claim (ii) it
suffices to show that $\mu(C,\rho)$ depends only on the space
$V_\rho.$ To this end we assume that $E_{\rho'}$ is another
$G$-stable lattice in $V_\rho$ and we have to show that
$\mu(C,\rho)=\mu(C,\rho').$ By Proposition \ref{additiv}(ii) and
d\'evissage we may assume that $C\cong M[0]$ with $pM=0$ and
similarly that $E_{\rho'}^*\subseteq E_\rho^*$ with $\pi T=0$ for
$T:=E_\rho^*/E_{\rho'}^*.$ In this situation there is an exact
sequence
\[ 0 \to M\otimes_{\mathbb{F}_p}T \to M\otimes_{\mathbb{F}_p}\overline{E_{\rho'}^*} \to
M\otimes_{\mathbb{F}_p}\overline{E_\rho^*}\to
M\otimes_{\mathbb{F}_p}T \to 0\]
of $\La(G)$-modules. The required claim now follows from the known
additivity of $\mu$-invariants and the fact that the $\La
(G)$-modules $M\otimes_{\mathbb{F}_p}\overline{E_{\rho}^*}$ and
$M\otimes_{\mathbb{F}_p}\overline{E_{\rho'}^*}$ are isomorphic to $
M(\rho^*)$ and $M((\rho')^*)$ respectively. The second assertion of
 claim (ii) then follows from claim (i).

To prove claim (iii) we may assume that $C\cong M[0]$ with $M$ a
finitely generated $\Omega(G)$-module. After choosing a finite
resolution $P$ of $M$ by finitely generated projective
$\Omega(G)$-modules and  using the additivity of $\mu(-,\psi_i)$ on
short exact sequences the proof is immediately reduced to the case
of a projective $\Omega(G)$-module because $\mu(P,\psi_i)=\sum_{j\in
\bz} (-1)^j\mu(P^j,\psi_i).$ But for every projective
$\Omega(G)$-module $Y,$ considered also as a $\La(G)$-module, and
for each $i\in I$ we have $\mu(Y,\psi_i)=\langle
Y(\psi^*_i),Y_1\rangle =
\dim_{\mathbb{F}_p}(\mathrm{End}_{\Omega(G)}(R_i))\langle
Y,Y_i\rangle$ by Lemma \ref{mu-proj}(ii).

Claim (iv) follows from Lemma \ref{additiv}(iii). \end{proof}


If $G$ has no element of order $p$, then Proposition \ref{pairing}
(iii) allows us to define an integer $\mu_{\La (G)}^i(C )$ for
each complex $C$ in $D^{\rm p}_{S^*}(\La(G))$ and each index $i$
in $I$ by setting
\[ \mu_{\La (G)}^i(C ) := \mu(C,\psi_i)\cdot
\dim_{\mathbb{F}_p}(\mathrm{End}_{\Omega(G)}(R_i))^{-1}.\]

\subsubsection{$K$-groups} We continue to assume that $G$ has no element of order $p$ and write
$\mathfrak{D}(G)$ for the category  of finitely generated
$\La(G)$-modules that are annihilated by a power of $p.$ Then, by
d\'evissage and lifting of idempotents, one obtains the following
isomorphisms
\begin{equation}\label{p-tor}
K_0(\mathfrak{D}(G))\cong K_0(\Omega(G))\cong K_0(A(G))\cong \bz^I
\end{equation}
where the $i$-th basis vector of the free $\bz$-module on the right
corresponds to the classes of $Y_i$ in $K_0(\mathfrak{D}(G))$ and
$K_0(\Omega(G)).$ Lemma \ref{mu-proj} (ii) implies that if $M$
belongs to $\mathfrak{D}(G)$, then the map in \eqref{p-tor} sends
the class of $M$ to the vector
\be\label{mapsto} \mu(M):=(\mu_{\La (G)}^i(M[0]))_{i\in I}.\ee
%
%

The proof of the following result is a natural generalization of
that given by Kato in \cite[Prop.\ 8.6]{kato-k1}.

\begin{prop}\label{general-kato} If $G$ has no element of order $p$, then there are natural isomorphisms
\begin{eqnarray*}\label{decK0}K_0(\mathfrak{M}_{S^*}(G))&\cong&
K_0(\mathfrak{M}_S(G))\oplus K_0(\Omega (G)),\\
\label{decK1}K_1(\La(G)_{S^*})&\cong& K_1(\La(G)_S)\oplus K_0(\Omega
(G)).
\end{eqnarray*}
The first of these isomorphisms is induced by the embeddings of
categories $\mathfrak{M}_S(G) \subseteq \mathfrak{M}_{S^*}(G)$ and
 $\mathfrak{D}(G) \subseteq \mathfrak{M}_{S^*}(G)$ combined with
the first isomorphism in \eqref{p-tor}. The second isomorphism
depends on the choice of a splitting of
$K_1(\La(G)_{S^*})\twoheadrightarrow K_0(\mathfrak{D}(G))\cong
\bz^I;$ once we have fixed an idempotent $e_i$ for each $i \in I$
a natural choice is induced by sending the $i$-th basis vector of
$\bz^I$ to the class of $f_i:=1+(p-1)e_i$ in $K_1(\La(G)_{S^*})$.
\end{prop}

\begin{proof} We first prove the surjectivity of the homomorphism $\partial_2$
in the long exact localisation sequence of $K$-theory
\[ K_2(\La(G)_{S^*})  \xrightarrow{\partial_2} K_1(\Omega(G)_S) \to K_1(\La(G)_S)  \to
  K_1(\La(G)_{S^*}) \xrightarrow{\partial_1} K_0(\Omega(G)_S).\]
But, since $\Omega(G)_S$ is semi-local, the natural homomorphism
$\Omega(G)_S^*\twoheadrightarrow K_1(\Omega(G)_S)$ is surjective
and hence $K_1(\Omega(G)_S)$ is generated by the image of $S$.
 The surjectivity of $\partial_2$ thus follows from the fact that for each $f\in S$
  one has $\partial_2(\{f,p\})=[f]\in K_1(\Omega(G)_S),$
where $\{f,p\}$ denotes the symbol of $f$ and $p$ in
$K_2(\La(G)_{S^*})$ (indeed, the latter equality is proved by the
argument of \cite[Prop. 5]{kato-cft}). From the above exact
sequence we therefore obtain an exact sequence
\be\label{firstses} 0 \to K_1(\La(G)_S) \to K_1(\La(G)_{S^*})
\xrightarrow{\partial_1} K_0(\Omega(G)_S).  \ee
We next consider the map
\be
\label{map} \bz^I \to K_1(\La(G)_{S^*}) \xrightarrow{\partial_1}
K_0(\Omega(G)_S) \xrightarrow{\alpha} K_0(B(G))\cong\bz^J,
\ee
where the first map is given by sending the $i$-th basis vector of
$\bz^I$ to the class of $f_i=1+(p-1)e_i$ (note that
$\La(G)/\La(G)f_i\cong Y_i$),
$B(G):=\Omega(G)_S/\mathrm{Jac}(\Omega(G)_S) $
   and where the canonical map $\alpha$ is injective by \cite[Chap. IX, Prop. 1.3]{bass} while
   the index set $J$ parametrizes the
   isomorphism classes of simple modules over the semisimple Artinian ring
 $B(G)$. Let $N$ be any closed normal subgroup of $G$ which is both pro-$p$ and open in $H.$
Then it is straight forward to check that \eqref{map} factorizes
through the composite
\be \label{map_S} {\bz^I} \cong K_0(\Omega(G/N)) \xrightarrow{\beta} K_0(\Omega(G/N)_S)
\xrightarrow{\gamma}K_0(B(G)). \ee
Here the surjective map $\beta$ comes from the exact localisation
 sequence and the isomorphism $\gamma$ is induced from the fact that
 $\Omega(G)_S\twoheadrightarrow B(G)$ factors through the composite
 $\,\Omega(G)_S\twoheadrightarrow \Omega(G/N)_S\twoheadrightarrow
B(G)\,$ by the proof of \cite[Lem.\ 4.3]{cfksv} and the fact that $\mathrm{Jac}(\Omega(G/N)_S)$ is
a nilpotent ideal. It follows that the map in \eqref{map} is surjective and hence that $\partial_1$
is surjective and $\alpha$ is bijective.

If $\mathfrak{D}_S(G)$ denotes the category of finitely generated
$\La(G)_S$-modules which are $\bz_p$-torsion, then we have shown
that the composite map
\begin{equation}\label{p-tor_S} K_0(\mathfrak{D}_S(G))\cong
K_0(\Omega(G)_S)\cong K_0(B(G))=\bz^J
\end{equation}
is bijective and that $|J|\leq |I|.$

By combining (\ref{firstses}) with the surjectivity of $\partial_1$,
the bijectivity of (\ref{p-tor_S}) and the assertion of Lemma
\ref{D-D_S}(i) below we obtain an exact sequence
\be\label{secondses}  0 \to K_1(\La(G)_S) \to K_1(\La(G)_{S^*})
\xrightarrow{\partial_1'} K_0(\mathfrak{D}(G))\to 0.\ee
Further, it is straightforward to show that, with respect to the
isomorphism $K_0(\mathfrak{D}(G)) \cong \bz^I$ of (\ref{p-tor}),
this sequence is split by the map which sends the $i$-th basis
vector of $\bz^I$ to the class of $f_i$ in $K_1(\La(G)_{S^*})$. This
proves the final assertion of Proposition \ref{general-kato}.

We next consider the following diagram with exact rows
\begin{equation}\label{loc-seq}
\begin{CD}
  0 @> >> \mathrm{im}(\iota_{S^*}) @> >> K_1(\La(G)_S) @> >> K_0(\mathfrak{M}_S(G)) @> >> 0\\
@. @\vert @V VV @VV \delta V \\
 0 @> >> {\mathrm{im}(\iota_{S^*})} @> >> K_1(\La(G)_{S^*}) @> >> K_0(\mathfrak{M}_{S^*}(G)) @> >> 0,
\end{CD}\end{equation}
where $\iota_{S^*}$ is the natural map $K_1(\La(G)) \to K_1(\La(G)_{S^*})$, $\delta$ is induced by
the embedding $\mathfrak{M}_S(G)\subseteq \mathfrak{M}_{S^*}(G)$ of categories and \cite[Prop.
3.4]{cfksv} implies that each row is indeed exact. By applying the snake lemma to this diagram and
comparing with the sequence (\ref{secondses}) we obtain an exact sequence of the form
\[ 0 \to K_0(\mathfrak{M}_S(G))\xrightarrow{\delta} K_0(\mathfrak{M}_{S^*}(G))\to
K_0(\mathfrak{D}(G))\to 0.\]
The first assertion of Proposition \ref{general-kato} now follows
because this sequence is split by the homomorphism
$K_0(\mathfrak{D}(G))\to K_0(\mathfrak{M}_{S^*}(G))$ induced by the
embedding of categories $\mathfrak{D}(G) \subseteq
\mathfrak{M}_{S^*}(G)$. \end{proof}

\begin{lem}\label{D-D_S}\hfill
\begin{itemize}
\item[(i)] The exact scalar extension functor from
$\La(G)\mbox{-mod}$ to $\La(G)_S\mbox{-mod}$ identifies
$\mathfrak{D}(G)$ with a full subcategory of $\mathfrak{D}_S(G)$
and induces an isomorphism $\,K_0(\mathfrak{D}(G))\cong
K_0(\mathfrak{D}_S(G)).$
\item[(ii)] The natural map $\iota:K_0(\mathfrak{D}(G))\to
K_0(\mathfrak{M}_{S^*}(G))$ is injective.
\item[(iii)] The natural map $\,K_0(\Omega(G/N))\to K_0(\Omega(G/N)_S)\,$ is bijective.
\end{itemize}
\end{lem}

\begin{proof} The assignment $M \mapsto (\mu_{\La (G)}^i(M[0]))_{i\in I}$ induces a homomorphism
\[\mu: K_0(\mathfrak{M}_{S^*}(G))\to \bz^I.\]
Now from \eqref{p-tor} and \eqref{mapsto} we know that $\mu\circ
\iota$ is bijective whilst from Lemma \ref{additiv}(iii)
 we know $\delta(K_0(\mathfrak{M}_S(G)))\subseteq \ker(\mu)$ where $\delta$ is the
 homomorphism in diagram (\ref{loc-seq}). This implies that
 $\iota$ is injective (so proving claim (ii)), that $\mu$ is surjective and that $|I|\leq
 |J|$. But $|J|\leq |I|$ (see just after \eqref{p-tor_S}) and so $|I|=|J|$.

Since $|I|=|J|$ the isomorphisms of (\ref{p-tor}) and (\ref{map_S})
 combine to imply that the natural map $K_0(\mathfrak{D}(G)) \to
 K_0(\mathfrak{D}_S(G))$ is bijective (proving claim (i)).

In a similar way, claim (iii) follows by combining the equality
$|I|=|J|$ together with the surjectivity of the map $\beta$ in
(\ref{map_S}) and the definition of the index set $J$ (in
(\ref{map})).
\end{proof}

\section{Characteristic series}\label{aplf}

In this section we associate a canonical `characteristic series' to
each complex
 in $D^{\rm p}_{\tilde S}(\La (G))$. This construction extends the notion of `algebraic $p$-adic $L$-functions' introduced in
  \cite{burns} and hence refines the `Akashi series' introduced by Coates, Schneider and
 Sujatha in \cite{css}. It will also play a key role in our proof of Theorem
 \ref{nc-weier} (see in particular the proof of Lemma \ref{tfdc2}).
%

\subsection{The definition} If $M$ is any (left) $\Lambda(G)$-module, then
 the completed tensor product
\[ {\rm I}_H^G(M) := \Lambda(G)\hat\otimes_{\Lambda(H)}M\]
has a natural structure as (left) $\Lambda(G)$-module. With
respect to this action, one obtains a (well-defined) endomorphism
$\Delta_\gamma$ of ${\rm I}_H^G(M)$ by setting
\[\Delta_\gamma(x\otimes_{\La (H)} y) :=
 x\tilde\gamma^{-1}\otimes_{\La (H)}  \tilde\gamma y\]
for each $x \in \Lambda(G)$ and $y \in M$, where $\tilde \gamma$ is
any lift of $\gamma$ through the natural projection $G \to \Gamma$.
It is easily checked that $\Delta_\gamma$ is independent of the
precise choice of $\tilde \gamma$. Further, if $M$ belongs to
$\mathfrak{M}_{S^*}(G)$, then \cite[Lem. 2.1]{burns} implies that
\[ \delta_\gamma:= \id - \Delta_\gamma\]
induces an automorphism of the (finitely generated)
$\Lambda(G)_{S^*}$-module ${\rm I}_H^G(M)_{S^*}$.

The ring $\Lambda(G)_{S^*}$ is regular \cite[Prop.
4.3.4]{fukaya-kato}. Hence there exists a natural isomorphism
between $K_1(\Lambda(G)_{S^*})$ and the group
 $G_1(\Lambda(G)_{S^*})$ which is generated (multiplicatively) by symbols $\langle \alpha\mid M\rangle$
 where $\alpha$ is an automorphism of a finitely generated $\Lambda(G)_{S^*}$-module $M$
 (cf.\ \cite[Th. 16.11]{swan}). For each complex $C$ in $D^{\rm
 p}_{S^*}(\La (G))$ we may therefore define an element of
 $K_1(\Lambda(G)_{S^*})$ by setting
\[ {\rm char}^*_{G,\gamma}(C) := \prod_{i \in \bz}\langle
\delta_\gamma \mid {\rm I}_H^G(H^i(C))_{S^*}\rangle^{(-1)^i}.\]

For each $C$ in $D^{\rm
 p}_{\tilde S}(\La (G))$ we also define an `equivariant multiplicative $\mu$-invariant' in
 $\im(K_1(\Lambda(G)[\frac{1}{p}]) \to K_1(\Lambda(G)_{S^*}))$ by setting
\[ \chi^\mu_G(C) := \begin{cases} \langle\sum_{i\in I} p^{\mu_{\Lambda(G)}^i(C)}e_i|
\Lambda(G)_{S^*}\rangle, &\text{if $C$ belongs to $D^{\rm
p}_{\tilde S}(\La (G))\setminus D^{\rm p}_{S}(\La (G))$,}\\
1, &\text{if $C$ belongs to $D^{\rm p}_{S}(\La (G))$.}\end{cases}\]

This definition makes sense since if $C$ belongs to $D^{\rm
 p}_{\tilde S}(\La (G))\setminus D^{\rm
 p}_S(\La (G))$, then $G$ has no element of order $p$. 

\begin{defn}\label{aplf-def} {\em For each $C$ in $D^{\rm p}_{\tilde S}(\La
(G))$ the {\em characteristic series of} $C$ is the element
\[ {\rm char}_{G,\gamma}(C) := \chi^\mu_G(C)\cdot{\rm char}^*_{G,\gamma}(C)\]
of $K_1(\Lambda(G)_{S^*})$.}
\end{defn}

\begin{remark}\label{G=Gamma}{\em If $G = \Gamma$, then $\La (G)_{S^*} = Q(\Gamma)$
and so there is a natural isomorphism $\iota: K_1(\La (G)_{S^*})
\cong Q(\Gamma)^\times$. Further, if $M$ is any finitely generated
torsion $\La(\Gamma)$-module, then $\iota ({\rm
char}_{G,\gamma}(M[0])) = (1+T)^{-\lambda(M)}{\rm char}_T(M)$ where
$\lambda(M)$ is the Iwasawa $\lambda$-invariant of $M$ and ${\rm
char}_T(M)$ is the characteristic polynomial of $M$ with respect to
the variable $T = \gamma - 1$. (For a proof of this fact see
\cite[Lem. 2.3]{burns}).}\end{remark}

\begin{remark}{\em In Proposition \ref{char-el}(i) we will prove that if $G$ has no element of order $p$, then ${\rm char}_{G,\gamma}(C)$ is a
 `characteristic element for $C$' in the sense of
\cite[(33)]{cfksv} (and see also Remark \ref{extension} in this
regard). In \cite[Th. 4.1]{burns} it is proved that this is also
true if $G$ has rank one as a $p$-adic Lie group. In these cases at
least, it therefore seems reasonable to regard ${\rm
char}_{G,\gamma}(C)$ as the canonical `algebraic $p$-adic
$L$-function' associated to $C$.
 }\end{remark}
%

\subsection{Basic properties}

\begin{lem}\label{additive} If $C_1 \to C_2 \to C_3 \to
C_1[1]$ is an exact triangle in $D^{\rm p}_{\tilde S}(\La (G))$,
then $\,{\rm char}_{G,\gamma}(C_2) = {\rm char}_{G,\gamma}(C_1){\rm
char}_{G,\gamma}(C_3)$.
\end{lem}

\begin{proof} If $G$ has an element of order $p$, then each complex
$C_j$ belongs to $D^{\rm p}_{S}(\La (G))$ and so $\chi^\mu_G(C_j) =
1$. If $G$ has no element of order $p$, then the equality
$\chi^\mu_G(C_2) = \chi^\mu_G(C_1)\chi^\mu_G(C_3)$ follows from
Lemma \ref{additiv}(i) and Proposition \ref{pairing}(iv). The
equality ${\rm char}_{G,\gamma}^*(C_2) = {\rm
char}_{G,\gamma}^*(C_1){\rm char}^*_{G,\gamma}(C_3)$
 is equivalent to that of \cite[Prop. 3.1]{burns}.
\end{proof}

%
Let $U$ be a closed subgroup of $H$ that is normal in $G$ and
 set $G_1 := G/U, H_1 := H/U$ and $S_1 := S_{G_1,H_1}$. Then there exists a natural ring
homomorphism $\pi_{G_1}: \La(G)_{S^*} \to \La(G_1)_{S_1^*}$ and
hence an induced homomorphism of groups
\[ \pi_{G_1,*}: K_1(\La(G)_{S^*}) \to
K_1(\La(G_1)_{S_1^*}).\]
%

\begin{lem}\label{bclemma} Let $G_1, H_1$ and $S_1$ be as above. Fix $C$ in $D^{\rm p}_{\tilde S}(\La
 (G))$ and assume that either $C$ belongs to $D^{\rm p}_S(\La
 (G))$ or that $G_1$ has no element of order $p$.
 Then $C_1 := \La(G_1)\hat\otimes_{\La (G)}^{\mathbb{L}}C$ belongs to $D^{\rm p}_{\tilde S_1}(\La
 (G_1))$ and $\pi_{G_1,*}({\rm char}_{G,\gamma}(C)) =
{\rm char}_{G_1,\gamma}(C_1).$
\end{lem}

\begin{proof} If $C$ belongs to $D^{\rm p}_S(\La
 (G))$, then $C_1$ belongs to $D^{\rm p}_{S_1}(\La
 (G_1))$ and so $\pi_{G_1,*}(\chi^\mu_G(C))$ $ = \pi_{G_1,*}(1) = 1 = \chi^\mu_{G_1}
(C_1)$. If $C$ does not belong to $D^{\rm p}_S(\La
 (G))$, then neither $G$ or $G_1$ has an element of order $p$ and the equality
 $\pi_{G_1,*}(\chi^\mu_G(C)) = \chi^\mu_{G_1}(C_1)$ follows from
Lemma \ref{additiv}(iv) and Proposition \ref{pairing}(iv). The
equality $\pi_{G_1,*}({\rm char}_{G,\gamma}^*(C)) = {\rm
char}_{G_1,\gamma}^*(C_1)$ is equivalent to that of \cite[Prop.
3.2]{burns}. \end{proof}

In the next result we fix an open subgroup $U$ of $G$ and set $H_U := H \cap U$ and $\Gamma_U :=
U/H_U$. We use the natural isomorphism $\Gamma_U \cong HU/H$ to regard $\Gamma_U$ as an open
subgroup of $\Gamma$, we set $d_U := [\Gamma: \Gamma_U]$ and write $\gamma_U$ for the topological
generator $\gamma^{d_U}$ of $\Gamma_U$. We set $S_U := S_{U,H_U}$ and note that $\Lambda (G)$,
resp.\ $\La(G)_{S}$, resp. $\La(G)_{S^*}$ is a free $\La (U)$-, resp.\ $\Lambda (U)_{S_U}$-, resp.\
$\Lambda (U)_{S_U^*}$-module, of rank $[G:U]$. In particular, restriction of scalars gives natural
functors $D^{\rm p}_S(\La (G)) \to D^{\rm p}_{S_U}(\La
 (U))$ and $D^{\rm p}_{S^*}(\La (G)) \to D^{\rm p}_{S_U^*}(\La
 (U))$ and a natural homomorphism
\[ {\rm res}_{U,*}: K_1(\La(G)_{S^*}) \to
K_1(\La(U)_{S_U^*}).\]

\begin{lem} Let $G$ and $U$ be as above and fix $C$ in $D^{\rm p}_{\tilde S}(\La
 (G))$. Then $C_1 := {\rm res}^G_UC$ belongs to $D^{\rm p}_{\tilde S}(\La
 (G))$ and ${\rm res}_{U,*}({\rm char}_{G,\gamma}(C)) =
{\rm char}_{U,\gamma_U}(C_1).$\end{lem}

\begin{proof} We prove first that ${\rm res}_{U,*}(\chi_G^{\mu}(C)) =
\chi_U^{\mu}(C_1).$ The complex $C$ belongs to $D^{\rm p}_S(\La
 (G))$ if and only if $C_1$ belongs to $D^{\rm p}_{S_U}(\La
 (U))$ and in this case ${\rm res}_{U,*}(\chi^\mu_G(C)) = 1 = \chi^\mu_{U}
(C_1)$. We may therefore assume that $C$ belongs to $D^{\rm
p}_{S^*}(\La
 (G))\setminus D^{\rm p}_S(\La
 (G))$ and hence that $G$ (and also $U)$ has no element of order $p$. By the same
argument as used in the proof of Proposition \ref{pairing}(iii), we
can also assume that $C = Y_a[0]$ for some index $a$ in $I$.
 Then for each $i\in I$ one has $\mu_{\La (G)}^i(C) = \langle Y_a,Y_i\rangle = \delta_{ai}$ and so
\[ {\rm res}_{U,*}(\chi_G^{\mu}(C)) = {\rm res}_{U,*}(\langle
pe_a| \Lambda(G)_{S^*}\rangle) = {\rm res}_{U,*}(\langle p|
\Lambda(G)_{S^*}e_a\rangle). \]
We write $\{f_j: j \in J\},$ $\{ X(U)_j: j \in J\}$ and $\{ Y(U)_j: j \in J\}$ for the idempotents
of $\La (U)$ as well as the submodules of $\Omega (U)$ and $\La(G)$ that are
 analogous to the idempotents $e_i$
 of $\La (G)$ as well as the  submodules $X_i$ of $\La(G)$ and $Y_i$ of $\Omega(G)$ defined in \S\ref{module theory}. For each
 $j$ in $J$ we set $m_j := \langle \La (G)e_a,X(U)_j\rangle=\langle Y_a,Y(U)_j\rangle.$ Then the
 $\La (U)_{S_U^*}$-module $\La (G)_{S^*}e_a$ is isomorphic to
  $\bigoplus_{j \in J}(\La (U)_{S_U^*}f_j)^{m_j}$ and hence the
  last displayed expression is equal to $\langle\sum_{j \in J}p^{m_j}f_j|\La
  (U)_{S_U^*}\rangle = \chi_U^{\mu}(\mathrm{res}^G_UY_a[0])$, as required.

It remains to prove that ${\rm res}_{U,*}({\rm
char}_{G,\gamma}^*(C)) = {\rm char}_{U,\gamma_U}^*(C_1)$. To do
this we may assume that $C = M[0]$ for a module $M$ in
$\mathfrak{M}_{S^*}(G)$ so that ${\rm char}_{G,\gamma}^*(C)$ is
equal to $\langle\delta_\gamma\mid {\rm I}^G_H(M)_{S^*}\rangle$.
 Now the $\La (U)_{S_U^*}$-module ${\rm I}^G_H(M)_{S^*}$ is equal
to the direct sum $\bigoplus_{i=0}^{i= d_U-1} \Delta_{\gamma^i}({\rm I}^{U}_{H_U}(M)_{S_U^*})\cong
\bigoplus_{i=0}^{i= d_U-1} {\rm I}^{U}_{H_U}(M)_{S_U^*}$ where the isomorphism identifies each
translate $\Delta_{\gamma^i}({\rm I}^{U}_{H_U}(M)_{S_U^*})$ with ${\rm I}^{U}_{H_U}(M)_{S_U^*}$ in
the natural way. With respect to this decomposition $\delta_\gamma$ is the automorphism given by
the $d_U\times d_U$ matrix
\[ \begin{pmatrix}
{\rm id}& -{\rm id} & 0 &\hdotsfor[2]{3} &0\\
0& {\rm id} & -{\rm id} & 0 &\hdotsfor[2]{2} &0\\
0& 0 & {\rm id} & -{\rm id} & 0 &\dots &0\\
\hdotsfor[2]{7}\\
0& \hdotsfor[2]{3} &0 & {\rm id} & -{\rm id}\\
-\Delta_{\gamma^{d_U}}& 0& \hdotsfor[2]{3}  & 0 & {\rm id}
\end{pmatrix}.\]
Elementary row and column operations show that this automorphism
represents the same element of $K_1(\La(U)_{S_U^*})$ as does the
automorphism $\alpha$ of $\bigoplus_{i=0}^{i= d_U-1} {\rm
I}^{U}_{H_U}(M)_{S_U^*}$ that acts as ${\rm id} -
\Delta_{\gamma^{d_U}}$ on the last direct summand and as the
identity on all other summands. The required result thus follows
because,
 since $\delta_{\gamma_U} := {\rm id} - \Delta_{\gamma^{d_U}}$,
 the class of $\alpha$ in $K_1(\La(U)_{S_U^*})$  is equal to
 $\langle \delta_{\gamma_U}\mid {\rm
I}^U_{H_U}(M)_{S_U^*}\rangle =: {\rm char}_{U,\gamma_U}^*(C_1)$.
\end{proof}

%
%
%
%
%
%
%


\subsection{The proof of Theorem \ref{nc-weier}}\label{pfof2.1}

We deduce Theorem \ref{nc-weier} as a consequence of the following
result.

\begin{prop}\label{char-el} Assume that $G$ has no element of order $p$.
\begin{itemize}

\item[(i)] For each $C$ in $D^{\rm p}_{S^*}(\La (G))$ one has
$\,\partial_G({\rm char}_{G,\gamma}(C)) = \sum_{i \in
\bz}(-1)^{i+1}[H^i(C)].$
\item[(ii)] There exists a (unique) homomorphism
$\chi_{G,\gamma}$ from $K_0(\mathfrak{M}_{S^*}(G))$ to $K_1(\La
(G)_{S^*})$ that simultaneously satisfies the following conditions:-

\begin{itemize}
\item[(a)] for each $M$ in $\mathfrak{M}_{S^*}(G)$ one has $\chi_{G,\gamma}([M]) =
{\rm char}_{G,\gamma}(M[1])$;

\item[(b)] $\chi_{G,\gamma}$ is right inverse to $\partial_G$;

\item[(c)] $\chi_{G,\gamma}$ respects the isomorphisms of Proposition
\ref{general-kato};

\item[(d)] Let $U$ be a closed subgroup of $H$ that is normal
in $G$ and such that $\overline{G} := G/U$ has no element of order
$p$. Set $\overline{H} := H/U$ and $\overline{S} :=
S_{\overline{G},\overline{H}}$. Then there is a commutative diagram
\[ \begin{CD} K_0(\mathfrak{M}_{S^*}(G)) @> \chi_{G,\gamma} >> K_1(\La
(G)_{S^*})\\
@V VV @V VV \\
K_0(\mathfrak{M}_{\overline{S}^*}(\overline{G})) @>
\chi_{\overline{G},\gamma}>>
K_1(\La(\overline{G})_{\overline{S}^*})\end{CD}\]
where the vertical arrows are the natural homomorphisms.
\end{itemize}
\end{itemize}
\end{prop}

\begin{remark} {\em Proposition \ref{char-el}(i) shows that ${\rm char}_{G,\gamma}(C)$
is a `characteristic element for $C$' in the sense of
\cite[(33)]{cfksv}.
The surjectivity of $\partial_G$ (which follows directly from
Proposition \ref{char-el}(ii)(b)) was first proved in \cite[Prop.\
3.4]{cfksv}. }
\end{remark}

The proof of Proposition \ref{char-el} will be the subject of
\S\ref{pfofce}. However, we now show that it implies Theorem
\ref{nc-weier}. To do this we consider the map
\[ \iota_G: K_0(\Omega (G)) \oplus K_0(\mathfrak{M}_S(G))\oplus {\rm
im}(K_1({\La}(G))\to K_1({\La}(G)_{S^*})) \to
K_1({\La}(G)_{S^*})\]
which for each $M$ in $\Omega (G)$, $N$ in $\mathfrak{M}_S(G)$ and
$u$ in ${\rm im}(K_1({\La}(G))\to K_1({\La}(G)_{S^*}))$ satisfies
$\iota_G(([M], [N], u)) = {\rm char}_{G,\gamma}(M[1]){\rm
char}_{G,\gamma}(N[1])u$. Then Proposition \ref{char-el}(ii) implies
that $\iota_G$ is a well-defined homomorphism which, upon
restriction to the summand
 $K_0(\Omega (G)) \oplus K_0(\mathfrak{M}_S(G))$, gives a right
 inverse to the composite $K_1(\La(G)_{S^*}) \to K_0(\mathfrak{M}_{S^*}(G)) \to
 K_0(\Omega (G)) \oplus K_0(\mathfrak{M}_S(G))$ where the first arrow
 is $\partial_G$ and the second is the isomorphism of Proposition
\ref{general-kato}. The exactness of the lower row of
(\ref{loc-seq}) thus implies that $\iota_G$ is bijective. This
completes the proof of Theorem \ref{nc-weier}.
%
%

\subsection{The proof of Proposition \ref{char-el}}\label{pfofce} In addition to proving Proposition
\ref{char-el} we also translate Definition \ref{aplf-def} into the
language of localized $K_1$-groups introduced by Fukaya and Kato in
\cite{fukaya-kato}. We therefore use the notation of Appendix A.

\subsubsection{$S$-acyclic complexes}\label{sac} In \cite[Prop. 2.2, Rem. 2.3]{sch-ven} Schneider and the second named author have proved that for each
bounded complex $P$ of projective $\La (G)$-modules
 in $D^{\rm p}_S(\La (G))$ there exists an exact sequence of complexes in $D^{\rm p}(\La (G))$
\begin{equation}\label{canseq} 0 \to {\rm I}^G_H(P) \xrightarrow{\delta}
 {\rm I}^G_H(P)\xrightarrow{\pi} P \to 0\end{equation}
where in each degree $i$ the morphism $\delta$, respectively $\pi$,
is equal to the homomorphism $\delta_\gamma: {\rm I}^G_H(P^i)\to
{\rm I}^G_H(P^i)$, respectively to the natural projection
  ${\rm I}^G_HP^i \to P^i$ . We may therefore define a trivialization
\[ t_S(P): \u_{\Lambda(G)} \rightarrow
\d_{\Lambda (G)}({\rm I}^G_H(P))\d_{\Lambda (G)}({\rm
I}^G_H(P))^{-1} \to \d_{\Lambda (G)}(P)\]
where the first arrow is induced by the identity map on ${\rm
I}^G_H(P)$ and the second by applying property A.d) to the exact
sequence (\ref{canseq}). By using property A.g) of the functor
$\d_{\Lambda (G)}$ we then extend this definition to obtain for any
object $C$ of $D^{\rm p}_S(\La (G))$ a canonical morphism
\[ t_S(C) : \u_{\Lambda(G)} \to \d_{\Lambda (G)}(C).\]
This morphism is analogous to those that arise naturally in the
context of varieties over finite fields (cf.\ \cite[Lem.
3.5.8]{kato}, \cite[\S3.2]{fgt}).

\begin{lem}\label{alg-l} For each $C$ in $D^{\rm p}_S(\La (G))$ one has ${\rm ch}_{\Lambda(G),\Sigma_{C}}(
[C, t_{S}(C)]) = {\rm char}_{G,\gamma}^*(C)$.
\end{lem}

\begin{proof} One has ${\rm ch}_{\Lambda(G),\Sigma_{C}}([C, t_{S}(C)]) =
\theta_{{C},t_{S}(C)}$ where the latter element is defined in
\cite[(1)]{BV}. To compute $\theta_{{C},t_S(C)}$ we set $Q :=
\Lambda (G)_{S^*}$, $C_Q := Q\otimes_{\Lambda (G)}C$, $\H(C)_Q :=
Q\otimes_{\Lambda (G)}\H(C)$ and $\H(C)_{H,Q} := Q\otimes_{\Lambda
(H)} \H(C)$ and consider the diagram
\begin{equation*}\minCDarrowwidth1.2em\begin{CD} \u_{Q} @> \alpha_1>>
\d_Q(\H(C)_{H,Q})\d_Q(\H(C)_{H,Q})^{-1} @> \alpha_2>>
\d_Q(\H(C)_Q) @> \alpha_3>> \d_Q(C_Q)\\
@V \alpha_5VV @\vert @. @\vert \\
\u_{Q} @> \alpha_4>> \d_Q(\H(C)_{H,Q})\d_Q(\H(C)_{H,Q})^{-1} @>
\alpha_2>> \d_Q(\H(C)_Q) @> \alpha_3>> \d_Q(C_Q).
\end{CD}\end{equation*}
In this diagram $\alpha_1$ is induced by the identity map on
$\H(C)_{H,Q}$, $\alpha_2$ results from applying property A.d) to
(\ref{canseq}) with $C = {\H}(C)$, $\alpha_3$ is property A.h),
$\alpha_4$ is induced by $\d_Q(\H(\delta_\gamma))$ and $\alpha_5$ is
defined so that the first square commutes. The upper row of the
diagram is equal to the morphism $\u_Q \to \d_Q(C_Q)$
 induced by
 $t_S(C)$ whilst the
lower row agrees with the morphism $\u_Q \to \d_Q(C_Q)$ induced by
the acyclicity of $C_Q$. From the commutativity of the diagram we
thus deduce that
 the element $\theta_{{C},t_S(C)}$ of $K_1(Q)$ is represented by $\alpha_5$. On the other hand, a comparison of the maps $\alpha_1$ and $\alpha_4$ shows that $\alpha_5$ represents
the same element of $K_1(Q)$ as does the morphism
\[\d_Q(\H(C)_{H,Q})\cong \prod_{i\in \bz}
\d_Q(H^i(C)_{H,Q})^{(-1)^i} \to \prod_{i\in \bz}
\d_Q(H^i(C)_{H,Q})^{(-1)^i}\cong \d_Q(\H(C)_{H,Q})\]
where the first and third maps use property A.h) and the second map
is $\prod_{i\in \bz} \d_Q(H^i(\delta_\gamma))^{(-1)^i}$. From here
we deduce the required equality
\[ \theta_{{\H}({C}),t_{S}({\H}({C}))} = \prod_{i \in \bz}\langle \delta_\gamma \mid
H^i(C)_{H,Q}\rangle^{(-1)^i} \in K_1(Q).\]
\end{proof}



\subsubsection{$p$-torsion complexes}\label{tc} In this subsection we
assume that $G$ has no element of order $p$ (so that $D^{\rm p}(\La
(G))$ identifies with $D^{\rm fg}(\Lambda(G))$).

If $T$ is a bounded complex of finitely generated
$\Omega(G)$-modules, then there is a bounded complex of finitely
generated projective $\Omega(G)$-modules $\bar{P}$ that is
isomorphic in $D(\Omega(G))$ (and hence also in $D^{\rm
p}(\Lambda(G))$) to $T$. Also, 
 following the discussion of \S\ref{module theory},  in each degree $i$ there is a finitely generated projective $\La
(G)$-module $P^i$ and an exact sequence of $\La(G)$-modules
\be\label{res} 0 \to P^i \xrightarrow{p} P^i \to \bar{P}^i \to
0.\ee
We may therefore define a morphism
\begin{multline*} t(T): \,\,\,\u_{\La (G)} \rightarrow
\prod_{i \in\bz}(\d_{{\La}(G)}(P^i)\d_{{\La}(G)}(P^i)^{-1})^{(-1)^i} \to\\
\prod_{i \in \bz}\d_{{\La}(G)}(\bar{P}^i)^{(-1)^i}  =
\d_{{\La}(G)}(\bar{P}) \to \d_{{\La}(G)}(T)\end{multline*}
where the first arrow is induced by the identity map on each module
$P^i$, the second by applying property
 A.d) to each of the sequences (\ref{res}) and the last by the given quasi-isomorphism
$\bar{P} \cong T$. 
%
If now $C$ is any bounded complex of modules in $\mathfrak{D}(G)$,
then there exists a finite length filtration of $C$ by complexes
\begin{equation}\label{filt} 0 = C_d \subset C_{d-1} \subset \cdots
C_1 \subset C_0 = C\end{equation}
so that each quotient complex $T_i := C_i/C_{i+1}$ belongs to
$D^{\rm p}(\Omega(G))$ (and hence to $D^{\rm p}(\Lambda (G))$). This
gives an identification $\, \d_{\Lambda (G)}(C) = \prod_{0\le i < d}
\d_{\Lambda (G)}(T_i)$ and, with respect to this identification, we
set
\[ t(C) := \prod_{0\le i < d} t(T_i).\]
This definition is easily checked to be independent of the choice of
filtration (\ref{filt}) and, for each $i$, of isomorphism $\bar{P}
\cong T_i$ and resolution (\ref{res}) used to define $t(T_i)$.

\begin{lem}\label{mu} If $G$ has no element of order $p$ and $C$ is any
bounded complex of modules in $\mathfrak{D}(G)$, then in
$K_1(\La(G)_{S^*})$ one has ${\rm
ch}_{\Lambda(G),\Sigma_{S^*}}([C,t(C )] ) = \chi_G^\mu(C).$
\end{lem}

\begin{proof} This follows from the definition of ${\rm ch}_{\Lambda(G),\Sigma_{S^*}}$ and the
fact that there is a resolution \eqref{res} of the form $\,0 \to
\Lambda(G)^n \xrightarrow{d^j} \Lambda(G)^n \to \bar{P}^j \to 0 \,$
where $n=\sum_{i=1}^{i=c} \langle\bar{P}^j,\bar{X}_i\rangle$ and
$d^j$ is given with respect to the canonical basis by the diagonal
matrix with entries $ f_i^{\langle\bar{P}^j,\bar{X}_i\rangle},$
$1\leq i\leq n,$ where the natural numbers
$\langle\bar{P}^j,\bar{X}_i\rangle$ are defined via the
decomposition
 $\bar{P}^j\cong\bigoplus_{i=1}^{i=c}\bar{X}_i^{\langle\bar{P}^j,\bar{X}_i\rangle}.$
The fact that the $\mu$-invariants give the correct multiplicities
is the same as for \cite[Prop. 4.8]{arwa}.\end{proof}

\subsubsection{The proof of Proposition {\ref{char-el}}}
For each complex $C$ in $\Sigma_{S^*}$ we write $\H(C)_{\rm tor}$ and ${\H}(C)_{\rm tf}$ for the
complexes with $\H(C)_{\rm tor}^i = H^i(C)_{\rm tor}$ and ${\H}(C)_{\rm tf}^i = H^i(C)_{\rm tf}$ in
each degree $i$ and in which all differentials are zero. There is a tautological exact sequence of
complexes $0 \to \H(C)_{\rm tor}\to \H(C) \to {\H}(C)_{\rm tf} \to 0$ and hence an equality $\,[C]
= [\H(C)] = [\H(C)_{\rm tor}] + [{\H}(C)_{\rm tf}]\,$ in $K_0(\mathfrak{M}_{S^*}(G))$. From
Definition \ref{aplf-def} it is also clear that $\chi_G^{\mu}(C) = \chi_G^{\mu}(\H(C)_{\rm tor})$
and ${\rm char}_{G,\gamma}^*(C) = {\rm char}_{G,\gamma}^*(\H(C)_{\rm tf})$. Claim (i) therefore
follows upon combining Lemmas \ref{alg-l} and \ref{mu} (with $C$ replaced by $\H(C)_{\rm tf}$ and
$\H(C)_{\rm tor}$ respectively) with the following fact: there is a commutative diagram of
homomorphisms of abelian groups
\[\begin{CD}
K_1(\Lambda(G),\Sigma_{S^*}) @> \partial ' >> K_0(\Sigma_{S^*})\\
 @V{\rm ch}_{{\La}(G),\Sigma_{S^*}} VV @V\iota VV\\
 K_1({\La}(G)_{S^*}) @> \partial_G >> K_0(\mathfrak{M}_{S^*}(G))
\end{CD}
\]
where $\partial '$ sends each class $[C,a]$ to $-[[C]]$ (cf.
\cite[Th. 1.3.15]{fukaya-kato}) and $\iota$ sends each class $[[C]]$
to $\sum_{i \in \bz}(-1)^i[H^i(C)]$ (cf.
\cite[\S4.3.3]{fukaya-kato}).

Regarding claim (ii) we note first that if a homomorphism
$\chi_{G,\gamma}$ exists satisfying property (a), then it is
automatically unique. Next we note that Lemma \ref{additive} implies
that the assignment $M\mapsto {\rm char}_{G,\gamma}(M[1])$ for each
$M$ in $\mathfrak{M}_{S^*}(G)$ induces a well-defined homomorphism
$\chi_{G,\gamma}: K_0(\mathfrak{M}_{S^*}(G)) \to K_1(\La (G)_{S^*})$
and claim (i) implies that this homomorphism is a right inverse to
$\partial_G$. Further, for each $M$ in $\mathfrak{D}(G)$ and $N$ in
$\mathfrak{M}_S(G)$ one has ${\rm char}_{G,\gamma}(M[1]) =
\chi^\mu_{G,\gamma}(M[1]) \in \iota(K_0(\Omega (G)))$ and ${\rm
char}_{G,\gamma}(N[1]) = {\rm char}_{G,\gamma}^*(N[1]) \in
\iota(K_0(\mathfrak{M}_S(G)))$ (where the latter equality follows
from Lemma \ref{additiv}(iii)) and so property (c) is satisfied.
Finally, the commutativity of the diagram in (d) is a direct
consequence of Lemma
 \ref{bclemma}.


\section{Descent theory}\label{elt}

In this section we consider leading terms of elements of
 $K_1(\La(G)_{S^*})$ and in particular prove Theorem
\ref{lt-result}. The approach of this section was initially
developed by the first named author in an unpublished early
version of the article
 \cite{burns}.

We deal first with the case that $C$ is acyclic. In this case the
complex $\bz_p[\overline{G}]\otimes_{\La (G)}^{\mathbb{L}}C$ is
acyclic so
$[\d_{\bz_p[\overline{G}]}(\bz_p[\overline{G}]\otimes_{\La
(G)}^{\mathbb{L}}C),t(C)_{\overline{G}}]$ is the zero element of
$K_0(\bz_p[\overline{G}],\bq^c_p[\overline{G}])$ and also $\xi$
belongs to the image of the natural map $K_1(\La(G))\to
K_1(\La(G)_{S^*})$. The equality of Theorem \ref{lt-result} is
therefore a consequence of the following result.

\begin{lem}\label{moddescent} If $u$ belongs to the image of the natural map
$\lambda: K_1(\La(G))\to K_1(\La(G)_{S^*})$, then for each finite
quotient $\overline{G}$ of $G$ the element $(u^*(\rho))_{\rho\in
{\rm Irr}(\overline{G})}$ belongs to
$\ker(\partial_{\overline{G}})$.
\end{lem}

\begin{proof} Let $\O$ be the valuation ring of a finite extension $L$ of $\qp$
such that all representations can be realised over $\O.$
 If $v\in K_1(\La(G))$ with $u = \lambda(v)$, then $ u^*(\rho)= u(\rho) = \lambda(v)(\rho)\in
\O^\times$ for all $\rho\in\mathrm{Irr}(\overline{G}).$ Thus by
functoriality of $K$-theory and the fact that the canonical map
$K_1(\La_\O(\Gamma))\to K_1(\O)$ is equal to the `evaluation at $0$'
map $\La_\O(\Gamma)^\times\to \O^\times,$ the image of $v$ in
$K_1(\zp[\overline{G}])$ under the natural projection is mapped to
$(u^*(\rho))_{\rho\in\mathrm{Irr}(\overline{G})}\in
K_1(L[\overline{G}]).$ \end{proof}

\subsection{Reduction to $S$-acyclicity} We now reduce
the general case of Theorem \ref{lt-result} to the case that $C$
belongs to $D^{\rm p}_S(\La (G))$. In this subsection we therefore
assume that $G$ has no element of order $p$.

\begin{lem}\label{moddescent2} If $G$ has no element of order $p$, then it is enough to prove Theorem \ref{lt-result} in the case that $C$
 is acyclic outside at most one degree.\end{lem}

\begin{proof} We assume that the result of Theorem \ref{lt-result}
is true for all complexes that are acyclic outside at most one
degree. To deduce Theorem \ref{lt-result} in the general case we use
induction on the number of non-zero cohomology groups of $C$ (which
we assume to be at least two). We let $n$ denote the largest integer
$m$ for which $H^m(C)$ is non-zero. We set $C_1 := H^n(C)[-n]$ and
write $C_2$ for the truncation of $C$ in degrees less than $n$
(which has fewer non-zero cohomology group than does $C$). Then
there is an exact triangle in $D_{S^*}^{\rm p}(\La (G))$ of the form
\begin{equation}\label{iet} C_1 \to C \to C_2 \to
C_1[1]\end{equation}
and the assumption that $C$ is semisimple at $\rho$ implies that
$C_1$ and $C_2$ are also semisimple at $\rho$ (for any $\rho$ in
${\rm Irr}(\overline{G})$). Let $\xi$ be an element such that
$\partial_G(\xi) = [C]$. If $\xi_1$ is such that
 $\partial_G(\xi_1) =
[C_1]$, then $\xi_2 := \xi\xi_1^{-1}$ satisfies $\partial_G(\xi_2) =
\partial_G(\xi) - \partial_G(\xi_1) = [C] - [C_1] =
[C_2]$, where the last equality follows from (\ref{iet}). Hence, by
the inductive hypothesis, one has
\begin{multline}\label{iet2} \partial_{\overline{G}}((\xi^*(\rho))_{\rho} =
\partial_{\overline{G}}((\xi_1^*(\rho))_{\rho}) +
\partial_{\overline{G}}((\xi_2^*(\rho))_{\rho}) \\
=
-[\d_{\bz_p[\overline{G}]}(C_{1,\overline{G}}),t(C_1)_{\overline{G}}]
-[\d_{\bz_p[\overline{G}]}(C_{2,\overline{G}}),t(C_2)_{\overline{G}}].
\end{multline}
where $C_{i,\overline{G}} := \bz_p[\overline{G}]\otimes_{\La
(G)}^{\mathbb{L}}C_i$ for $i = 1,2$. Now if we set $C_{\overline{G}}
:= \bz_p[\overline{G}]\otimes_{\La (G)}^{\mathbb{L}}C$, then
(\ref{iet}) induces an exact triangle in $D^{\rm
p}(\bz_p[\overline{G}])$
\begin{equation*} C_{1,\overline{G}} \to C_{\overline{G}} \to C_{2,\overline{G}} \to
C_{1,\overline{G}}[1]\end{equation*}
%
%
and, with respect to this triangle, the trivialisations
$t(C_1)_{\overline{G}}, t(C)_{\overline{G}}$ and
$t(C_2)_{\overline{G}}$ satisfy the `additivity criterion' of
\cite[Cor. 6.6]{br-bu}. Indeed, for each $\rho$ in ${\rm
Irr}(\overline{G})$ the cohomology sequence of the exact triangle
\[ \bq_p^c\otimes^{\mathbb{L}}_{\La (\Gamma)}C_{1,\rho} \to \bq_p^c\otimes^{\mathbb{L}}_{\La
(\Gamma)}C_\rho \to  \bq_p^c\otimes^{\mathbb{L}}_{\La
(\Gamma)}C_{2,\rho} \to \bq_p^c\otimes^{\mathbb{L}}_{\La
(\Gamma)}C_{1,\rho}[1]\]
(that is induced by (\ref{iet})) gives an equality $r_G(C)(\rho) =
r_G(C_1)(\rho) + r_G(C_2)(\rho)$ and a short exact sequence of
complexes
\[ 0 \to \H_{\rm bock} (\triangle(C_{1,\rho},\gamma)) \to
 \H_{\rm bock}(\triangle(C_\rho,\gamma)) \to \H_{\rm
 bock}(\triangle(C_{2,\rho},\gamma))\to 0.\]
Here we use the notation of Appendix \ref{bockhom} and write
$\triangle(C_\rho,\gamma)$ for the triangle
\begin{equation}\label{cantriangle}
\bq^c_p\otimes_\O^\mathbb{L} C_\rho
\xrightarrow{\theta_{\gamma,\rho}} \bq^c_p\otimes_\O^\mathbb{L}
C_\rho \to \bq^c_p\otimes^\mathbb{L}_{\La_\O(\Gamma)}C_\rho
 \to \bq^c_p\otimes_\O^\mathbb{L}
 C_\rho[1]\end{equation}
where $\theta_{\gamma,\rho}$ is induced by multiplication by
$\gamma-1,$ and we use similar notation for $C_1$ and $C_2$. This
means that the criterion of \cite[Cor.\ 6.6]{br-bu} is satisfied if
 one takes (in the notation of loc.\ cit.) $\Sigma$ to be $\bq^c_p[\overline{G}]$,
 $P \xrightarrow{a} Q \xrightarrow{b} R
\xrightarrow{c} P[1]$ to be the exact triangle (\ref{iet}) (so
$\ker(H^{\rm ev}a_\Sigma) = \ker(H^{\rm od}a_\Sigma) =0$) and the
trivialisations $t_P$, $t_Q$ and $t_R$ to be induced by
$(-1)^{r_G(C_1)(\rho)}t(C_{1,\rho})$,
$(-1)^{r_G(C)(\rho)}t(C_\rho)$ and
$(-1)^{r_G(C_2)(\rho)}t(C_{2,\rho})$ respectively. From \cite[Cor.
6.6]{br-bu} we therefore deduce that the last element in
(\ref{iet2}) is indeed equal to
$-[\d_{\bz_p[\overline{G}]}(\bz_p[\overline{G}]\otimes_{\La
(G)}^{\mathbb{L}}C),t(C)_{\overline{G}}]$, as required.
\end{proof}

Taking account of Lemmas \ref{moddescent} and \ref{moddescent2} we
now assume that $C$ is acyclic outside precisely one degree. To be
specific, we assume that $C = M[0]$ with $M$ in
$\mathfrak{M}_{S^*}(\La (G))$. Then there is an exact triangle of
the form
\begin{equation}\label{t-tf}M_{\rm tor}[0] \to M[0] \to M_{\rm tf}[0]\to M_{\rm tor}[1]\end{equation}
where $M_{\rm tor}$ belongs to $\mathfrak{D}(G)$ and $M_{\rm tor}$
to $\mathfrak{M}_S(\La (G))$. In this case one has
$t(M[0])_{\overline{G}} = t(M_{\rm tf}[0])_{\overline{G}}$ and so
(by another application of \cite[Cor. 6.6]{br-bu})
\begin{multline}\label{add0}[\d_{\bz_p[\overline{G}]}(\bz_p[\overline{G}]\otimes_{\La
(G)}^{\mathbb{L}}M[0]),t(M[0])_{\overline{G}}] =\\
 [\d_{\bz_p[\overline{G}]}(\bz_p[\overline{G}]\otimes_{\La
(G)}^{\mathbb{L}}M_{\rm tor}[0]),can] + [\d_{\bz_p[\overline{G}]}(\bz_p[\overline{G}]\otimes_{\La
(G)}^{\mathbb{L}}M_{\rm tf}[0]),t(M[0])_{\overline{G}}]\end{multline}
with $can $ the canonical morphism $\d_{\bq_p[\overline{G}]}(\bq_p[\overline{G}]\otimes_{\La
(G)}^{\mathbb{L}}M_{\rm tor}[0]) = \d_{\bq_p[\overline{G}]}(0) \to {\bf 1}_{\bq_p[\overline{G}]}$.
%

\begin{lem}\label{moddescent3} Let $N$ be an object of $\mathfrak{D}(G)$. If $\xi$ is any element of $K_1(\La (G)_{S^*})$ with $\partial_G(\xi) =
[N[0]]$, then for any finite quotient $\overline{G}$ of $G$ one
has
\[ \partial_{\overline{G}}((\xi^*(\rho))
_\rho)  = -
 [\d_{\bz_p[\overline{G}]}(\bz_p[\overline{G}]\otimes_{\La
(G)}^{\mathbb{L}}N[0]),can].\]
\end{lem}

\begin{proof} An easy reduction (using d\'evissage and the additivity of Euler characteristics on exact sequences in
$\mathfrak{D}(G)$) allows us to assume that $N$ is an object of
$\mathfrak{D}(G)$ which lies in an exact sequence of $\La
(G)$-modules of the form
\be\label{special case} 0 \to Q \xrightarrow{d} P \to N \to 0\ee
where $Q$ and $P$ are both finitely generated and projective.
(Indeed, it is actually enough to consider the case that $Q = P =
\La (G)e_i$ for an idempotent $e_i$ as in \S\ref{module theory} and
with $d$ equal to multiplication by $p$.)

To proceed we identify the subgroup $K_0(\mathfrak{D}(G))$ of $K_0(\mathfrak{M}_{S^*}(G))$ with the
relative algebraic $K$-group $K_0(\La (G),\La (G)[\frac{1}{p}])$. To be compatible with the
normalisations used in \S\ref{K-group-defs} we fix this isomorphism so that for every exact
sequence (\ref{special case}) the element $[N]$ of $K_0(\mathfrak{D}(G))$ corresponds to the
element $(Q, d',P)$ of $K_0(\La (G),\La (G)[\frac{1}{p}])$ with $d' := \La
(G)[\frac{1}{p}]\otimes_{\La (G)}d$.

Now since $\La(G)[\frac{1}{p}]\otimes_{\La (G)}N = 0$ the
localisation sequence of $K$-theory implies that any element $\xi$
as above belongs to the image of $K_1(\La (G)[\frac{1}{p}])$ in
$K_1(\La (G)_{S^*})$. This implies in particular that $\xi^*(\rho) =
\xi(\rho)$ for all $\rho$ in ${\rm Irr}(\overline{G})$. The natural
commutative diagram of connecting homomorphisms
\[\begin{CD}
K_1(\La (G)[{1\over p}]) @>  >> K_0(\La (G),\La
(G)[\frac{1}{p}])\\
@V VV @V VV\\ K_1(\bq_p[\overline{G}])@> \partial_{\overline{G}}>>
 K_0(\bz_p[\overline{G}],\bq_p[\overline{G}])\end{CD}\]
also then implies that $\partial_{\overline{G}}((\xi^*(\rho)) _\rho)
= (Q_{\overline{G}}, d'_{\overline{G}},P_{\overline{G}})$ with
$Q_{\overline{G}} := \bz_p[\overline{G}]\otimes_{\La (G)}Q,
P_{\overline{G}} := \bz_p[\overline{G}]\otimes_{\La (G)}P$ and
$d'_{\overline{G}} := \bq_p[\overline{G}]\otimes_{\La (G)}d$. Hence,
with respect to the isomorphism (\ref{caniso}) (with $F = \bq_p$),
one has
\be\label{almost-descent} \partial_{\overline{G}}((\xi^*(\rho))
_\rho) =
[\d_{\bz_p[\overline{G}]}(Q_{\overline{G}})\d_{\bz_p[\overline{G}]}(P_{\overline{G}})^{-1},\tau
]\ee
with $\tau$ equal to the composite morphism
\begin{multline*} \d_{\bq_p[\overline{G}]}(\bq_p\otimes_{\bz_p}Q_{\overline{G}})\d_{\bq_p[\overline{G}]}(\bq_p\otimes_{\bz_p}P_{\overline{G}})^{-1}
 \to\\
 \d_{\bq_p[\overline{G}]}(\bq_p\otimes_{\bz_p}P_{\overline{G}})\d_{\bq_p[\overline{G}]}(\bq_p\otimes_{\bz_p} P_{\overline{G}})^{-1}
 = {\bf 1}_{\bq_p[\overline{G}]}\end{multline*}
where the first arrow is induced by
$\d_{\bq_p[\overline{G}]}(d'_{\overline{G}})$. Finally we note that
 the (image under $\bz_p[\overline{G}]\otimes_{\La (G)}-$ of the)
sequence (\ref{special case}) induces an isomorphism in $D^{\rm p}(\bz_p[\overline{G}])$ between
$\bz_p[\overline{G}]\otimes_{\La (G)}^{\mathbb{L}}N[0]$ and the complex $Q_{\overline{G}}
\xrightarrow{d_{\overline{G}}} P_{\overline{G}}$ where the first term is placed in degree $-1$ and
$d_{\overline{G}} := \bz_p[\overline{G}]\otimes_{\La (G)}d$ and this implies that the element on
the right hand side of (\ref{almost-descent}) is the inverse of
$[\d_{\bz_p[\overline{G}]}(\bz_p[\overline{G}]\otimes_{\La (G)}^{\mathbb{L}}N[0]),can]$, as
required. \end{proof}

Lemmas \ref{moddescent}, \ref{moddescent2} and \ref{moddescent3}
 combine with (\ref{t-tf}) and (\ref{add0}) to reduce the proof of
 Theorem \ref{lt-result} to consideration of complexes in
 $D_S^{\rm p}(\La (G))$. In the remainder of \S\ref{elt} we shall therefore assume that $C$ belongs to $D_S^{\rm p}(\La
 (G))$.

%


\subsection{Equivariant twists}\label{thetriv} 
 In this subsection we introduce the algebraic formalism that is key to
 a proper understanding of descent.

\subsubsection{The definition} We fix an open normal subgroup $U$ of $G$ and set
$\overline{G} := G/U$. We write
\[ \Delta_{\overline{G}}: \La (G) \to \La(\overline{G}\times G) \cong \bz_p[\overline{G}]\otimes
\La (G)\]
for the (flat) ring homomorphism which sends each element $\sigma$
of $G$ to $\overline{\sigma}\otimes \sigma$ where
$\overline{\sigma}$ is the image of $\sigma$ in $\overline{G}$. Then
for each $\La(G)$-module $M$ the induced $\La(\overline{G}\times
G)$-module $\La(\overline{G}\times G)\otimes_{\La
(G),\Delta_{\overline{G}}}M$ can be identified with the module
\[ {\rm tw}_{\overline{G}}(M) := \bz_p[\overline{G}]\otimes M\]
upon which $\overline{G}$ acts via left multiplication and each
$\sigma \in G$ acts by sending $x\otimes y$ to
$x\overline{\sigma}^{-1}\otimes \sigma y$. This construction extends
to give an exact functor $C \mapsto {\rm tw}_{\overline{G}}(C)$ from
$D^{\rm p}(\La (G))$ to $D^{\rm p}(\La (\overline{G}\times G))$ and
for each such $C$ we set
\[ {\rm tw}_{\overline{G}}(C)_H := \La (\overline{G}\times \Gamma)\otimes^{\mathbb{L}}_{\La
(\overline{G}\times G)}{\rm tw}_{\overline{G}}(C) \in D^{\rm p}(\La
(\overline{G}\times \Gamma)).\]

\subsubsection{Base change} For each $s \in \La (G)$ we write ${\rm r}_s$ and ${\rm
r}_{\Delta_{\overline{G}}(s)}$ for the endomorphisms of $\La(G)$
and $\La(\overline{G}\times G)$ given by right multiplication by
$s$ and $\Delta_{\overline{G}}(s)$ respectively. Then ${\rm
cok}({\rm r}_{\Delta_{\overline{G}}(s)})$ is isomorphic as a
$\La(\overline{G}\times G)$-module to ${\rm
tw}_{\overline{G}}(\cok({\rm r}_s))$ and so is finitely generated
over $\La(\overline{G}\times H)$ if $\cok({\rm r}_s)$ is finitely
generated over $\La(H)$. This implies that
$\Delta_{\overline{G}}(S^*) \subseteq S_1^*$ where $S := S_{G,H}$
and $S_1 := S_{\overline{G}\times G, \overline{G}\times H}$ and so
$\Delta_{\overline{G}}$ induces a ring homomorphism
\begin{equation}\label{ind-hom}\La (G)_{S^*} \to \La(\overline{G}\times G)_{S_1^*} \to
 \La(\overline{G}\times \Gamma)_{S_2^*} = Q(\overline{G}\times \Gamma)\end{equation}
where $S_2:= S_{\overline{G}\times \Gamma, \overline{G}}$, the
second arrow is the natural projection and the equality is because
$\overline{G}\times \Gamma$ has dimension one (as a $p$-adic Lie
group). These maps induce a group homomorphism
\[\pi_{\overline{G}\times\Gamma}: K_1(\La(G)_{S^*}) \to K_1( Q(\overline{G}\times
\Gamma))
\]
which forms the upper row of a natural commutative diagram of
connecting homomorphisms
\be\label{conndiag} \begin{CD} K_1(\La(G)_{S^*}) @> >> K_1(\La
(\overline{G}\times G)_{S_1^*}) @>
>>
K_1(\La (\overline{G}\times \Gamma)_{S_2^*})\\
@V VV @V VV @V VV \\
K_0(\La (G),\La(G)_{S^*}) @> >> K_0(\La (\overline{G}\times G), {S_1^*}) @> >> K_0(\La
(\overline{G}\times\Gamma), {S_2^*})\end{CD} \ee

where we write $K_0(\La (\overline{G}\times G), {S_1^*})$ and
$K_0(\La (\overline{G}\times\Gamma), {S_2^*})$  for $K_0(\La
(\overline{G}\times G),\La(\overline{G}\times G)_{S_1^*})$ and
$K_0(\La
(\overline{G}\times\Gamma),\La(\overline{G}\times\Gamma)_{S_2^*})$
respectively.

\subsubsection{Reduced norms} We set $R := \La (\overline{G}\times\Gamma)$. Then the algebra $Q(R)$ identifies with the
group ring $Q(\La(\Gamma))[\overline{G}]$ and, with respect to this
identification, one has
\begin{equation}\label{centre-dec} \zeta (Q(R)) \subset \zeta (Q^c(R)) = \prod_{\rho
\in \Irr (\overline{G})}Q^c(\Gamma)\end{equation}
where $Q^c(R) := \bq_p^c\otimes_{\bq_p}Q(R)$ and $Q^c(\Gamma) :=
\bq_p^c\otimes_{\bq_p}Q(\La(\Gamma))$. We write $x = (x_\rho)_\rho$
for the corresponding decomposition of each element $x$ of $\zeta
(Q^c(R))$.

In the next result we write ${\rm Nrd}_{Q(R)}: K_1(Q(R)) \to
\zeta(Q^c(R))^\times$ for the reduced norm map of the semisimple
algebra $Q(R)$ and use the homomorphisms
 $\Phi_\rho$ defined in (\ref{ltdef}).

\begin{lem}\label{tfdc} For each $\xi$ in $K_1(\La (G)_{S^*})$ one has ${\rm
Nrd}_{Q(R)}(\pi_{\overline{G}\times \Gamma}(\xi)) = (\Phi_\rho(\xi))_{\rho\in {\rm Irr}(\overline{G})}$.
\end{lem}

\begin{proof} It suffices to show that each homomorphism $\Phi_\rho$ is equal to the
composite
\[ K_1(\La (G)_{S^*}) \to K_1(\La (\overline{G}\times
G)_{S^*_{\overline{G}\times G}}) \to K_1(Q(R))
 \to \zeta(Q^c(R))^\times \to
 Q^c(\Gamma)^\times\]
where the last arrow is projection to the $\rho$-component (under
 (\ref{centre-dec})).

From \cite[Prop. 4.2, Th. 4.4]{cfksv} we know that the natural map
$\La(G)_{\tilde S}^\times \to K_1(\La (G)_{\tilde S})$ is
surjective. Since $\La (G)_{\tilde S}^\times = \tilde S^{-1}\tilde
S$ it is therefore enough to check the claimed result for elements
of the form $\langle {\rm r}_s\mid \La (G)_{\tilde S}\rangle$. Also,
if $V_\rho$ is a finite dimensional $\bq_p^c$-space of character
$\rho$, then for each $x$ in $Q(R)^\times$ one has
\begin{equation}\label{rndec} {\rm Nrd}_{Q(R)}(\langle {\rm r}_x\mid
Q(R)\rangle)_\rho = {\rm det}_{Q^c(\Gamma)}(x\mid
\Hom_{\bq_p^c}(V_\rho,\bq_p^c)\otimes_{\bq^c_p}Q^c(\Gamma))\end{equation}
where $x \in Q(\La(\Gamma))[\overline{G}]$ acts on
$\Hom_{\bq_p^c}(V_\rho,\bq_p^c)\otimes_{\bq_p}Q(\La(\Gamma))$ via
the natural actions of $\overline{G}$ on
$\Hom_{\bq_p^c}(V_\rho,\bq_p^c)$ and of $Q(\Gamma)$ on $Q(\Gamma)$
(cf. \cite[\S3]{RW2}).

Now in computing $\Phi_\rho$ one must fix a $\bq_p^c$-basis $\{v_i:i
\in I\}$ of $V_\rho$. The claimed result follows because, if one
computes with respect to the $Q^c(\Gamma)$-basis $\{v_i\otimes 1: i
\in I\}$ of $\Hom_{\bq_p^c}(V_\rho,\bq_p^c)\otimes_{\bq_p^c}
Q^c(\Gamma)$, then $\Phi_\rho(s)$ is the matrix of the action of the
image of $\Delta_U(s)$ in $Q(R)$ on
$\Hom_{\bq_p^c}(V_\rho,\bq_p^c)\otimes_{\bq_p^c} Q^c(\Gamma)$.
\end{proof}

\subsubsection{Semisimplicity}\label{twisting-semi} There are natural isomorphisms in
$D^{\rm p}(\bz_p[\overline{G}])$
of the form
\[ \bz_p\otimes^{\mathbb{L}}_{\La (\Gamma)}{\rm tw}_{\overline{G}}(C)_H
\cong \bz_p\otimes^{\mathbb{L}}_{\La (G)}{\rm tw}_{\overline{G}}(C)
\cong \bz_p[\overline{G}]\otimes_{\La (G)}^{\mathbb{L}}C \]
and hence an exact triangle in $D(\La (\overline{G}\times\Gamma))$
of the form
\begin{multline*}
\triangle(\tw_{\overline{G}}(C),\gamma):\,\,
 \tw_{\overline{G}}(C)_H \xrightarrow{\theta_\gamma}
\tw_{\overline{G}}(C)_H\to
\bz_p[\overline{G}]\otimes^\mathbb{L}_{\La(G)} C \to
\tw_{\overline{G}}(C)_H [1]
\end{multline*}
where $\theta_\gamma$ is induced by multiplication by $\gamma-{\rm
id} \in \La (\Gamma)$ on $\La (\overline{G}\times \Gamma)$.

In the next result we use the terminology and notation of Appendix
\ref{bockhom}.

\begin{lem}\label{bock} Assume that $C$ is semisimple at $\rho$ (in
the sense of \cite[Def. 3.11]{BV}) for every $\rho$ in
$\mathrm{Irr}(\overline{G}).$ Then

\begin{itemize}
\item[(i)] $\theta_\gamma$ is semisimple;
\item[(ii)] $r_G(C)(\rho) = \sum_{i\in \bz}(-1)^{i+1}{\rm
dim}_{\bq^c_p}(\Hom_{\bq^c_p[\overline{G}]}(V_\rho,\bq^c_p\otimes
 H^i({\rm tw}_{\overline{G}}(C)_H)^\Gamma));$
\item[(iii)] with respect to the decomposition (\ref{m-e-decomp}) one has
 $\beta_{\bq_p\otimes\triangle(\tw_{\overline{G}}(C),\gamma)} =
 (t(C_\rho))_{\rho \in {\rm Irr}(\overline{G})}$ where
 $\bq_p\otimes\triangle(\tw_{\overline{G}}(C),\gamma)$ denotes the exact triangle in
 $D^{\mathrm{p}}(\La (\overline{G}\times\Gamma)[\frac{1}{p}])$ obtained from
 $\triangle(\tw_{\overline{G}}(C),\gamma)$ by scalar extension.
\end{itemize}
\end{lem}

\begin{proof} For any finitely  generated $\La(G)$-module $P$ there is a natural
  isomorphism of $\La(\overline{G}\times\Gamma)$-modules
\[\La(\overline{G}\times \Gamma)\otimes _{\La(\overline{G}\times G)} (\zp[\overline{G}]\otimes_\zp P)\cong \La(\Gamma)\otimes_{\La(G)} (\zp[\overline{G}]\otimes_\zp
  P)\]
where the action of $\overline{G}$ on the second module is just on
$\zp[\overline{G}]$ (from the left). Thus one has the following
  isomorphisms of $\La_{\bq^c_p}(\Gamma)$-modules
\begin{eqnarray*}
  e_\rho
  \bq^c_p[\overline{G}]\otimes_{\zp[\overline{G}]}
  \tw_{\overline{G}}(P)_H&\cong & \Lambda_{\bq^c_p}(\Gamma)
  \otimes^\mathbb{L}_{\La_{\bq^c_p}(G)}(e_\rho
  \bq^c_p[\overline{G}]\otimes_\zp M)\notag \\
  &\cong& \Lambda_{\bq^c_p}(\Gamma)
  \otimes^\mathbb{L}_{\La_{\bq^c_p}(G)}(V_{\rho*}\otimes_\zp M)\notag \\
  &\cong& {\bq^c_p\otimes_\O C_\rho}\label{iso1}
  \end{eqnarray*}
and
\[ e_\rho
\bq^c_p[\overline{G}]\otimes_{\bq^c_p[\overline{G}]}\big(\bq^c_p[\overline{G}]\otimes^\mathbb{L}_{\La(G)}
  C\big) \cong V_{\rho^*} \otimes ^\mathbb{L}_{\La(G)}
  C \cong \bq^c_p\otimes^\mathbb{L}_{\La_\O(\Gamma)}C_\rho.\]
%
%
Upon applying the (exact) functor
$e_\rho\bq^c_p[\overline{G}]\otimes_{\bz_p[\overline{G}]}-$ to
 $\triangle(\tw_{\overline{G}}(C),\gamma)$ we therefore recover the exact
 triangle $\triangle(C_\rho,\gamma)$ defined in
 (\ref{cantriangle}). 
%
  This implies in particular that $\theta_\gamma$ is semisimple if and only
 if $\theta_{\gamma,\rho}$ is semisimple for every $\rho$ in
 ${\rm Irr}(\overline{G})$. Claim (i) thus follows directly
 from the definition of `semisimplicity at $\rho$' (in terms of $\theta_{\gamma,\rho}$).
 The same reasoning also proves claims (ii) and (iii) since
\[ r_G(C)(\rho) = \sum_{i \in
 \bz}(-1)^{i+1}\dim_{\bq_p^c}(H^i(\bq^c_p\otimes_\O C_\rho)^\Gamma)\]
(by \cite[Lem.
 3.13(ii)]{BV}) whilst $t(C_\rho)$ is, by definition, equal to $\beta_{\triangle(C_\rho,\gamma)}$.
\end{proof}

\subsection{Leading terms}\label{leading terms} In this section we
fix an element $\xi$ and a complex $C$ as in Theorem
\ref{lt-result}.
 Then Lemma \ref{bock}(i) implies that the morphism $\theta_\gamma$
of ${\rm tw}_{\overline{G}}(C)_H$ is semisimple. Hence in each
degree $i$ there is a direct sum decomposition of $\La
(\overline{G}\times\Gamma)[\frac{1}{p}]$-modules
\begin{equation}\label{dec} \bq_p\otimes H^i({\rm
tw}_{\overline{G}}(C)_H) = D^i_0 \oplus D^i_1\end{equation}
with $D^i_0:= \bq_p\otimes \ker(H^i(\theta_\gamma)) = \bq_p\otimes
H^i({\rm tw}_{\overline{G}}(C)_H)^\Gamma$ and $D^i_1 := \bq_p\otimes
\im(H^i(\theta_\gamma))$. By assumption, both $D^i_0$ and $D^i_1$
are finitely generated (projective) $\bq_p[\overline{G}]$-modules
and $H^i(\theta_\gamma)$ induces an automorphism of $D^i_1$.

The proof of the following result will occupy the rest of this
section.

\begin{prop}\label{firstkey} 
$\partial_{\overline{G}}(((-1)^{r_G(C)(\rho)}\xi^*(\rho)) _{\rho\in
{\rm Irr}(\overline{G})}) =
\sum_{i\in\bz}(-1)^{i}\partial_{\overline{G}}(\langle
H^i(\theta_\gamma)\mid D^i_1\rangle)$.\end{prop}

\subsubsection{The descent to $Q(R)$} We write $\Sigma$ for the subset of $R$ consisting of
those elements of $\La(\Gamma)$ with non-zero image under the
projection $\La(\Gamma) \to \bz_p$. This is a multiplicatively
closed Ore set in $R$ which consists of central regular elements.

\begin{lem}\label{tfdc2} For each integer $i$ we set
$M^i := ({\rm I}^{\overline{G}\times\Gamma}_{\overline{G}}(H^i({\rm
tw}_{\overline{G}}(C)_H)^\Gamma))_{S^*}$. Then the element
\[ y_\xi := {\rm Nrd}_{Q(R)}(\pi_{\overline{G}\times
\Gamma}(\xi)\prod_{i\in\bz}\langle \delta_\gamma\mid
M^i\rangle^{(-1)^{i+1}})\]
belongs to $\zeta(R_\Sigma)^\times\subseteq
\zeta(Q(R))^\times$.\end{lem}

\begin{proof} We set $X := {\rm tw}_{\overline{G}}(C)_H$. Then the
commutative diagram (\ref{conndiag}) implies that
$\partial_{\overline{G}\times\Gamma}(\pi_{\overline{G}\times\Gamma}(\xi))
= [X]$. But $X$ belongs to $D^{\rm p}_S(R)$ and so \cite[Th.
4.1(ii)]{burns} implies that also
$\partial_{\overline{G}\times\Gamma}({\rm char}_{\overline{G}\times
\Gamma,\gamma}(X)) = [X]$. Hence the exact sequence (\ref{leskt})
(with $G$ replaced by $\overline{G}\times \Gamma$) implies that
there exists an element $u$ of $K_1(R)$ with
\begin{equation}\label{1s} \pi_{\overline{G}\times \Gamma}(\xi) =
\iota_1(u){\rm char}_{\overline{G}\times
\Gamma,\gamma}(X)\end{equation}
where $\iota_1$ is the natural homomorphism $K_1(R) \to K_1(R_\Sigma)
\to K_1(Q(R))$.

For each integer $i$ we now set $N^i := ({\rm
I}^{\overline{G}\times\Gamma}_{\overline{G}}(H^i({\rm
tw}_{\overline{G}}(C)_H)))_{S^*}$. Then the term $\langle
\delta_\gamma \mid N^i\rangle$ occurs in the definition of ${\rm
char}_{\overline{G}\times \Gamma,\gamma}(X) = {\rm
char}^*_{\overline{G}\times \Gamma,\gamma}(X)$. From Lemma
\ref{rest} below the action of $\delta_\gamma$ on $N^i$ restricts
to give an automorphism of
$R_\Sigma\otimes_{\bq_p[\overline{G}]}D^i_1$. From (\ref{dec}) it
follows that $\langle \delta_\gamma \mid N^i\rangle$ is equal to
\begin{multline*} \langle\delta_\gamma\mid
Q(R)\otimes_{\bq_p[\overline{G}]}D^i_0\rangle\langle
\delta_\gamma\mid Q(R)\otimes_{\bq_p[\overline{G}]} D^i_1 \rangle
\\= \langle \delta_\gamma\mid Q(R) \otimes_{\bq_p[\overline{G}]}
D^i_0 \rangle\iota_\Sigma ( \langle \delta_\gamma\mid
R_\Sigma\otimes_{\bq_p[\overline{G}]}D^i_1\rangle)
\end{multline*}
where $\iota_\Sigma$ is the natural homomorphism $K_1(R_\Sigma)
\to K_1( Q(R)).$ Hence by combining (\ref{1s}) with the definition
of ${\rm char}_{\overline{G}\times \Gamma,\gamma}(X)$ one finds
that
\begin{multline}\label{explicit}
 \pi_{\overline{G}\times
\Gamma}(\xi)\prod_{i\in\bz}\langle \delta_\gamma\mid
M^i\rangle^{(-1)^{i+1}} = \pi_{\overline{G}\times
\Gamma}(\xi)\prod_{i\in\bz}\langle \delta_\gamma\mid Q(R)
\otimes_{\bq_p[\overline{G}]}  D^i_0 \rangle^{(-1)^{i+1}}\\
 = \iota_1(u)\iota_\Sigma(\prod_{i\in\bz}\langle \delta_\gamma \mid
R_\Sigma\otimes_{\bq_p[\overline{G}]}D^i_1\rangle^{(-1)^i}) \in
\im(\iota_\Sigma).\end{multline}
%

Now $R_\Sigma$ is finitely generated as a module over the
commutative local ring $\La(\Gamma)_\Sigma$ and so is itself a
semi-local ring (cf. \cite[Prop. (5.28)(ii)]{curtisr}). The
natural homomorphism $R_\Sigma^\times \to K_1(R_\Sigma)$ is thus
surjective (by \cite[Th. (40.31)]{curtisr}) and so
(\ref{explicit}) implies that the element $\pi_{\overline{G}\times
\Gamma}(\xi)\prod_{i\in\bz}\langle \delta_\gamma\mid
M^i\rangle^{(-1)^{i+1}}$ is represented by a pair of the form
$\langle{\rm r}_y\mid Q(R)\rangle$ with $y \in R_\Sigma^\times$.
Now both $y$ and $y^{-1}$ are of the form $z\sigma^{-1}$ for
suitable elements
 $z \in R\cap Q(R)^\times$ and $\sigma\in \Sigma$. Thus, to
complete the proof of the lemma, it suffices to prove that for all
such $z$ and $\sigma$ both ${\rm Nrd}_{Q(R)}(\langle {\rm r}_z\mid
Q(R)\rangle)$ and ${\rm Nrd}_{Q(R)}(\langle {\rm
r}_{\sigma^{-1}}\mid Q(R)\rangle)$ belong to $\zeta(R_\Sigma)$.
But (\ref{rndec}) implies ${\rm Nrd}_{Q(R)}(\langle {\rm
r}_{\sigma^{-1}}\mid Q(R)\rangle) = (\sigma^{-d_\rho})_\rho \in
\zeta(R_\Sigma)$ with $d_\rho := {\rm dim}_{\bq_p^c}(V_\rho)$.
Also, if $\bz_p^c$ is the integral closure of $\bz_p$ in $\bq_p^c$
and $T_\rho$ is any full $\bz_p^c$-sublattice of $V_\rho$, then
the action of $\overline{G}$ on $T_\rho$ induces a homomorphism
$\varrho : R = \La(\Gamma)[\overline{G}] \to {\rm
M}_{d_\rho}(\bz_p^c\otimes \La(\Gamma))$ and (\ref{rndec}) implies
${\rm Nrd}_{Q(R)}(\langle {\rm r}_x\mid Q(R)\rangle) = ({\rm
det}(\varrho(z)))_\rho \in (\bq_p^c\otimes \zeta(R))\cap
\zeta(Q(R)) = \bq_p\otimes \zeta(R) \subseteq \zeta(R_\Sigma)$, as
required. \end{proof}

\begin{lem}\label{rest} $\delta_\gamma$ induces an automorphism of $R_\Sigma\otimes_{\bq_p[\overline{G}]}D^i_1$.
\end{lem}

\begin{proof} The argument of \cite[Prop. 2.2, Rem. 2.3]{sch-ven} gives a short exact
sequence
\[ 0 \to
R\otimes_{\mathbb{Z}_p[\overline{G}]}D^i_1
\xrightarrow{\delta_\gamma}
R\otimes_{\mathbb{Z}_p[\overline{G}]}D^i_1 \to D^i_1 \to 0\]
and so it suffices to show that $(D^i_1)_\Sigma = 0$. But
 $D_1^i := \bq_p\otimes \im(H^i(\theta_\gamma))$ and, regarding $\im(H^i(\theta_\gamma))$ as a
(finitely generated) module over $\La(\Gamma) \subseteq R$, the
decomposition (\ref{dec}) implies that
$\im(H^i(\theta_\gamma))_\Gamma$ is finite. This implies that
 $\im(H^i(\theta_\gamma))$ is a finitely generated torsion
$\La(\Gamma)$-module whose characteristic polynomial $f(T)$ is
coprime to $T$. It follows that $f(T)$ is invertible in $R_\Sigma$
and so $(D^i_1)_\Sigma = \im(H^i(\theta_\gamma))_\Sigma = 0$, as
required.\end{proof}

\subsubsection{The proof of Proposition {\ref{firstkey}}}\label{pfoffk}

From Lemmas \ref{tfdc} and \ref{tfdc2} we know that
\be\label{almost} (\xi^*(\rho)\prod_{i\in\bz}({\rm
Nrd}_{Q(R)}(\langle \delta_\gamma\mid
M^i\rangle))_\rho^*(0)^{(-1)^{i+1}})_{\rho\in {\rm
Irr}(\overline{G})} = \pi (y_\xi)\ee
where $\pi$ is the natural
 projection $\zeta(R_\Sigma)^\times \to \zeta(\bq
_p[\overline{G}])^\times$. But if $x$ is in $R_\Sigma^\times$, then
${\rm Nrd}_{Q(R)}(\langle {\rm r}_x\mid Q(R)\rangle)$ belongs to
$\zeta(R_\Sigma)^\times$ (see the proof of Lemma \ref{tfdc2}) and
(\ref{rndec}) implies $\pi({\rm Nrd}_{Q(R)}(\langle {\rm r}_x\mid
Q(R)\rangle)) = ({\rm Nrd}_{\bq_p[\overline{G}]}(\langle {\rm
r}_{\overline{x}}\mid \bq_p[\overline{G}]\rangle)_{\rho}$ with
$\overline{x}$ the image of $x$ in $\bq_p[\overline{G}]^\times$.
Hence (\ref{explicit}) implies $\pi(y_\xi) = {\rm
Nrd}_{\bq_p[\overline{G}]}(\overline{u})\prod_{i\in\bz}{\rm
Nrd}_{\bq_p[\overline{G}]}(\langle H^i(\theta_\gamma) \mid
 D^i_1\rangle)^{(-1)^i}$ where $\overline{u}$ is the image of $u$ under the natural composite homomorphism
$K_1(R)\to K_1(\bz_p[\overline{G}]) \to K_1(\bq_p[\overline{G}])$.
But $\partial_{\overline{G}}(\overline{u}) = 0$ and so
\be\label{2almost} \partial_{\overline{G}}(\pi(y_\xi)) =
\sum_{i\in\bz}(-1)^{i}\partial_{\overline{G}}(\langle
H^i(\theta_\gamma)\mid D^i_1\rangle).\ee
The equality of Proposition \ref{firstkey} thus follows upon
substituting (\ref{almost}) into (\ref{2almost}) and then using both
the explicit formula for $r_G(C)(\rho)$ given in Lemma
\ref{bock}(ii) and the following result (with $M = H^i({\rm
tw}_{\overline{G}}(C)_H)$ for each $i$).

\begin{lem}\label{exp-comp} If $M$ is any finitely generated $R$-module, then for each $\rho$ in $\Irr(\overline{G})$  one has
$({\rm Nrd}_{Q(R)}(\delta_\gamma\mid
 {\rm I}^{\overline{G}\times
 \Gamma}_{\overline{G}}(M^\Gamma)_{S^*}))
 ^*_\rho(0) = (-1)^{{\rm
dim}_{\bq^c_p}(\Hom_{\bq^c_p[\overline{G}]}(V_\rho,\bq^c_p\otimes
M^\Gamma))}.$
\end{lem}

\begin{proof} Set $W_\rho := \Hom_{\bq^c_p[\overline{G}]}(V_\rho,\bq_p^c\otimes M^\Gamma)$. Then
there are natural isomorphisms of $Q(\Gamma)$-modules of the form
\begin{multline*}\Hom_{\bq^c_p[\overline{G}]}(V_\rho,{\rm I}^{\overline{G}\times
\Gamma}_{\overline{G}}(M^\Gamma)_{S^*}) \cong
 \Hom_{\bq^c_p[\overline{G}]}(V_\rho,
 Q^c(\Gamma)\otimes_{\bz_p}M^\Gamma) \cong
 Q^c(\Gamma)\otimes_{\bq^c_p}W_\rho\end{multline*}
under which the induced action of $\delta_\gamma$ on the first
module corresponds to the endomorphism $\tilde \delta_\gamma$
 of the third module which sends $x\otimes_{\bz_p}m$ to
 $x(\gamma^{-1}-1)\otimes_{\bz_p}m$. It follows that ${\rm Nrd}_{Q(R)}(\delta_\gamma\mid
 ({\rm I}^{\overline{G}\times
 \Gamma}_{\overline{G}}(M^\Gamma))_\rho$ is equal to
\begin{multline*} \det_{Q^c(\Gamma)}(\tilde\delta_\gamma\mid
Q^c(\Gamma)\otimes_{\bq^c_p}W_\rho) =
\det_{Q^c(\Gamma)}(\gamma-1\mid Q^c(\Gamma)) ^{{\rm
dim}_{\bq^c_p}(W_\rho)}\\=(-T/(1+T))^{{\rm
dim}_{\bq^c_p}(W_\rho)}\end{multline*}
where the last equality follows from the fact that $\gamma^{-1}-1 =
(1-\gamma)/\gamma = -T/(1+T)$. From this explicit formula it is now
clear that the first non-zero coefficient of $T$ in the series ${\rm
Nrd}_{Q(R)}(\delta_\gamma\mid
 ({\rm I}^{\overline{G}\times
 \Gamma}_{\overline{G}}(M^\Gamma))_\rho$ is equal to $(-1)^{{\rm
dim}_{\bq^c_p}(W_\rho)}$.
\end{proof}

\subsection{Completion of the proof of Theorem \ref{lt-result}} We set $N := U \cap H$. Then in each degree $i$ there is an
isomorphism of $\La(\overline{G})$-modules
\[ H^i({\rm tw}_{\overline{G}}(C)_H)\cong\bz_p[\overline{G}]\otimes_{\La(H/N)}H^i(\La(G/N)\otimes^{\mathbb{L}}_{\La (G)}C).\]
But $\La(G/N)\otimes^{\mathbb{L}}_{\La (G)}C$ belongs to $D^{\rm
p}_{S_{G/N.H/N}}(\La(G/N))$ and so each $H^i({\rm
tw}_{\overline{G}}(C)_H)$ is a finitely generated
$\bz_p[\overline{G}]$-module. This implies that $\Delta({\rm
tw}_{\overline{G}}(C),\gamma)$ is an exact triangle in
 $D^{\rm p}(\bz_p[\overline{G}])$. In view of Lemma \ref{bock} and
 Proposition \ref{firstkey} we may therefore deduce Theorem \ref{lt-result} by applying the following
 result with $\mathcal{G} = \overline{G}, R = \bz_p$ and $\Delta = \Delta({\rm
tw}_{\overline{G}}(C),\gamma)$.

\begin{prop}\label{triprop} Let $\mathcal{G}$ be a finite group, $R$ an integral domain and $F$ the field of fractions of
 $R$. Let $\Delta: C \xrightarrow{\theta} C \to D \to C[1]$ be an exact triangle in $D^{\rm p}(R[\mathcal{G}])$. Assume
  that $\theta$ is semisimple and in each degree $i$ fix an
$F[\mathcal{G}][H^i(\theta)]$-equivariant direct complement $W^i$ to
$F\otimes_R\ker(H^i(\theta))$ in
 $F \otimes_RH^i(C)$. Then $H^i(\theta)$ induces an automorphism of the
 (finitely generated projective) $F[\mathcal{G}]$-module $W^i$, the element
 $\langle H^i(\theta) \rangle^* := \langle H^i(\theta)\mid W^i\rangle$ of $K_1(F[\mathcal{G}])$ is independent
 of the choice of $W^i$
 and in $K_0(R[\mathcal{G}],F[\mathcal{G}])$
 one has
\be\label{bockeq} \sum_{i \in
\bz}(-1)^i\partial_{\mathcal{G}}(\langle H^i(\theta)\rangle^*) =
  - [\d_{R[\mathcal{G}]}(D),\beta_{\Delta}]\ee
where $\partial_{\mathcal{G}}$ is the connecting homomorphism
$K_1(F[\mathcal{G}]) \to K_0(R[\mathcal{G}],F[\mathcal{G}])$ and
$\beta_\Delta$ is as defined in (\ref{bockdef}).
\end{prop}

\begin{proof} 
It is clear that $H^i(\theta)$ induces an automorphism of $W^i$ and
 straightforward to verify that $\langle H^i(\theta) \rangle^*$ is independent
 of the choice of $W^i$. However to prove (\ref{bockeq}) we replace $C$ by a bounded complex of finitely
generated projective $R[\mathcal{G}]$-modules $P$ and argue by
induction on $\,|P | := {\rm max}\{i: P^i \not= 0\} - {\rm min}\{j:
P^j \not= 0\}.$ We write $\phi$ for the morphism of complexes $P\to
P$ that corresponds to $\theta$.

If $|P| = 0$, then $P = P^m[-m] = H^m(P)$ (and $\phi = \phi^m =
H^m(\phi)$) for some integer $m$. In this case $D$ identifies with
 the mapping cone $P^m \xrightarrow{\phi^m} P^m$ of $\phi$ (so the first term of this complex is placed in degree $m-1$)
 in such a way that the homomorphism $H^{m-1}(D) \to \ker(H^m(\theta)) \to \cok(H^m(\theta)) \to H^m(D)$
 induced by $\Delta$ corresponds to the tautological map $\tau: \ker(\phi^m) \to
 \cok(\phi^m)$. Further, if $W$ is a direct complement to $F\otimes_R \ker(\phi^m)$ in $F\otimes_R P^m$,
 then $\phi^m(W) = W$ (since $\phi$ is semisimple) and $[\d_{R[\mathcal{G}]}(D),\beta_\Delta] = (P^m,\iota^{(-1)^{m-1}},P^m)$
 with $\iota$ the composite isomorphism
\[ F\otimes _RP^m = (F\otimes_R\ker(\phi^m))\oplus W \xrightarrow{(F\otimes_R\tau,H^m(\phi))} (F\otimes_R\cok(\phi^m))\oplus W \cong
F\otimes_R P^m\]
where the isomorphism is induced by a choice of splitting of the
 tautological exact sequence $0 \to W \to F\otimes_R P^m \to F\otimes_R \cok(\phi^m)\to
 0$ (this description of $[\d_{R[\mathcal{G}]}(D),\beta_\Delta]$ follows, for example, from
 \cite[Th. 6.2]{br-bu}). The equality (\ref{bockeq}) is therefore valid because $-(P^m,\iota^{(-1)^{m-1}},P^m) =
 (-1)^m\partial_{\mathcal{G}}(\langle\iota\mid F\otimes_RP^m\rangle) = (-1)^m\partial_{\mathcal{G}}(\langle H^m(\phi) \mid
 W\rangle)$.

We now assume that $|P| = n >0$ and, to fix notation, that
${\rm min}\{j: P^j \not= 0\} = 0$. 
We set $C_2 := P$ and $\phi_2 = \phi$, write $C_3$ for the naive
truncation in degree $n-1$ of $P$ and set $C_1 := P^n[-n]$. Then one
has a tautological short exact sequence of complexes
\be\label{taut}0 \to C_1 \to C_2 \to C_3\to 0.\ee
From the associated long exact cohomology sequence we deduce that
$H^i(C_3) = H^i(C_2)$ if $i < n-1$ and that there are commutative
diagrams of exact sequences
\be\label{diag1} \begin{CD} 0 @> >> H^{n-1}(C_2) @> >> H^{n-1}(C_3)
@>
>> B^n(C_2) @> >> 0\\ @. @V H^{n-1}(\phi_2)VV @V H^{n-1}(\phi_3)VV @V
\phi^nVV\\ 0 @>
>> H^{n-1}(C_2) @> >> H^{n-1}(C_3) @> >> B^n(C_2) @> >>
0\end{CD}\ee
\be\label{diag2}\begin{CD}0 @> >> B^n(C_2) @> >>
H^{n}(C_1) @>
>>
H^n(C_2) @> >> 0\\ @. @V \phi^n VV @V H^n(\phi_1)VV @V H^n(\phi_2)VV\\
0 @>
>> B^n(C_2) @> >> H^{n}(C_1) @> >> H^n(C_2) @> >> 0,\end{CD}\ee
where $B^n(C_2)$ denotes the coboundaries of $C_2$ in degree $n.$
 By mimicking the argument of \cite[Lem. 4.4]{burns}
 we may change $\phi$ by a homotopy in order to
 assume that, in each degree $i$, the restriction of $\phi^i$ induces
an automorphism of $F\otimes_RB^i(C_2)$. This assumption has two
important consequences. Firstly, the above diagrams imply that the
morphisms $\phi_1$ and $\phi_3$ of $C_1$ and $C_3$ that are induced
by $\phi$ are semisimple. Secondly, if we write $D_i$ for the
mapping cone of $\phi_i$ for $i = 1,2, 3$, then (\ref{taut}) induces
short exact sequences of the form
\begin{align*}
&0 \rightarrow D_1 \xrightarrow{} D_2 \xrightarrow{} D_3 \rightarrow
0\\
 0 \rightarrow\, &{\rm Z}(D_1) \xrightarrow{} {\rm Z}(D_2)
\xrightarrow{}
{\rm Z}(D_3) \rightarrow 0\\
0 \rightarrow\, &{\rm B}(D_1) \xrightarrow{} {\rm B}(D_2)
\xrightarrow{}
{\rm B}(D_3) \rightarrow 0\\
0 \rightarrow\, &\H(D_1) \xrightarrow{} \H(D_2) \xrightarrow{}
\H(D_3) \rightarrow 0
\\
0 \to \H_{\rm bock}&(\Delta_1) \xrightarrow{} \H_{\rm
bock}(\Delta_2) \xrightarrow{} \H_{\rm bock}(\Delta_3) \to
0.\end{align*}
Here we write ${\rm B}(D_i)$ and ${\rm Z}(D_i)$ for the complexes of
coboundaries and cocycles of $D_i$ (each with zero differentials)
and $\Delta_i$ for the tautological exact triangle
 $C_i \xrightarrow{\phi_i} C_i \to D_i \to C_i[1]$. Now from the displayed exact sequences (and the definition of each
term $[D_i,\beta_{\Delta_i}]$ in Appendix B) one has an equality
\[ [\d_{R[\mathcal{G}]}(D_2),\beta_{\Delta_2}] = [\d_{R[\mathcal{G}]}(D_1),\beta_{\Delta_1}] +
[\d_{R[\mathcal{G}]}(D_3),\beta_{\Delta_3}].\]
But the inductive hypothesis implies
\[ -[\d_{R[\mathcal{G}]}(D_3),\beta_{\Delta_3}]=  \sum_{i=0}^{i = n-1}(-1)^i\partial_{\mathcal{G}}(\langle H^i(\phi_3)\rangle^*),\]
whilst, since $|C_1| = 0$, our earlier argument proves
\[ -[\d_{R[\mathcal{G}]}(D_1),\beta_{\Delta_1}] = (-1)^n\partial_{\mathcal{G}}(\langle H^n(\phi_1)\rangle^*).\]
It is also clear that $\langle H^i(\phi_3)\rangle^* = \langle
H^i(\phi_2)\rangle^*$ for $i < n-1$ whilst (\ref{diag1}) and
(\ref{diag2}) imply $\langle H^{n-1}(\phi_2)\rangle^* = \langle
H^{n-1}(\phi_3)\rangle^*\langle \phi^n\mid F\otimes_R
B^n(C_2)\rangle$ and $\langle H^{n}(\phi_1)\rangle^* = \langle
\phi^n\mid F\otimes_R B^n(C_2)\rangle$ $\langle
H^{n}(\phi_2)\rangle^*$ respectively. The claimed description of
$-[D_2,\beta_{\Delta_2}]$ thus follows upon combining the last
three displayed equations.
\end{proof}

\begin{remark}{\em Notwithstanding the difference in Euler characteristic formalisms, Proposition
\ref{triprop} can also be deduced from the result of \cite[Prop.
3.1]{fgt}. If $\mathcal{G}$ is abelian, then Proposition
\ref{triprop} can be reinterpreted in terms of graded
 determinants and in this case has been proved to within a
sign ambiguity by Kato in \cite[Lem. 3.5.8]{kato}. (This sign
ambiguity arises because Kato uses ungraded determinants
  - for more details in this regard see \cite[Rem. 3.2.3(3) and
3.2.6(3),(5)]{kato} and \cite[Rem. 9]{bufl01}).}\end{remark}
\bigskip

\centerline{\bf Part II: Arithmetic}

\section{Field-theoretic preliminaries}\label{f-tp}  We first
introduce the class of fields for which the techniques of
\cite{cfksv} allow one to formulate a main conjecture of
non-commutative Iwasawa theory.

We fix an odd prime $p$ and write $\mathcal{F}$ for the set of
Galois extensions
 $L$ of $\bq$ which satisfy the following conditions
\begin{itemize}
\item[(i)] $L$ contains the cyclotomic $\bz_p$-extension $\bq^{\rm
cyc}$ of $\bq$;
\item[(ii)] $L/\bq$ is unramified outside a finite set of places;
\item[(iii)] $\Gal(L/\bq)$ is a compact $p$-adic Lie group.
\end{itemize}

We let $\mathcal{F}^+$ denote the subset of $\mathcal{F}$ comprising
those fields that are totally real.

The following observation provides an important reduction step and
was explained to us by Kazuya Kato.

\begin{lem}\label{kl} For each $F$ in $\mathcal{F}$ there exists
 $F'$ in $\mathcal{F}$ with $F \subseteq F'$ and such that $\Gal(F'/\bq)$
 has no element of order $p$. If $F$ belongs to $\mathcal{F}^+$, then one can also choose $F'$ in $
 \mathcal{F}^+$.
 \end{lem}

\begin{proof} We set $\tilde F := F(\zeta_{p^\infty})$. There is a $p$-torsion
free open normal subgroup $U$ of
 $V:=\Gal(\tilde F/\bq (\zeta_{p^{\infty}}))$. Let $L$ be the extension of $\bq$ in $\tilde F$ that
corresponds to $U$. For each non-trivial $p$-torsion element
 $\sigma_i$ of $V/U$, let $L_i$ be the fixed subfield of $L$ by $\sigma_i$.
Then $L=L_i(a_i^{1/p^n})$ for some $a_i \in L_i^\times$. Let
 $a_{i,j},\; 1\leq j \leq s(i),$ be all conjugates of $a_i$ over $F$ and
set $L'_i$ denote the field generated over $L_i$ by the set
 $\{a_{i,j}^{1/p^n}: 1\leq j \leq s(i), n\geq 1\}$. Then
 $\tilde F L'_i$ is a Galois extension of $F$ that contains $L$. Furthermore,
 $\Gal(L'_i/L_i)$ is isomorphic to a subgroup of $\bz_p^{s(i)}$ by
$\tau\mapsto (r(j))_j$ with $\tau(a_j^{1/p^n})/a_j^{1/p^n}=
\zeta_{p^n}^{r(j)}$. Let $F'$ be the composite field of
 $\tilde F$ and $L'_i$ for all $i$.

The group $\Gal(F'/\bq)$ is a compact $p$-adic Lie group and
 we now prove that it has no element of order $p$. We note first that $\Gal(\bq
(\zeta_{p^{\infty}})/\bq )$ is isomorphic to a subgroup of
$\bz_p^\times$ and hence is $p$-torsion free by the assumption $p
\neq 2$. Thus if $\sigma\in \Gal(F'/\bq)$ has order $p$, then the
image of $\sigma$ in $\Gal(\tilde F/\bq )$ is contained in $V$ and
so the image of $\sigma$ in $V/U$ coincides with $\sigma_i$ for
some $i$. Thus $\sigma$ fixes all elements of $L_i$. But then the
image of $\sigma$ in
 $\Gal(L_i'/L_i)$ is both $p$-torsion and also non-trivial (for its restriction to $\Gal(L/L_i)$ is non-trivial). This
contradicts the fact that $\Gal(L'_i/L_i)$ has no element of order
$p$. Hence $\Gal(F'/\bq)$ has no element of order $p$, as claimed.

Lastly we assume that $F$ is totally real. Then $\tilde F$ is a CM
field with maximal real subfield $\tilde F^+$ equal to the
compositum of $F$ and the maximal totally real subfield of $\bq
(\zeta_{p^\infty})$. Also, by the above construction, the extension
$F'/\tilde F$ is pro-$p$. Since $p$ is odd, the group
$\Gal(F'/\tilde F^+)$ therefore contains a unique element of order
$2$ and the fixed field $(F')^+$ of $F'$ by this element is totally
real, contains $F$, is Galois over $\bq$ and such that
$\Gal((F')^+/\bq)$ has no element of order $p$.
\end{proof}

In the remainder of this article we set $\Gamma := \Gal(\bq^{\rm cyc}/\bq)$ and $\mathcal{G}_L :=
\Gal(L/\bq)$ and $\mathcal{H}_L := \Gal(L/\bq^{\rm cyc})$ for each $L$ in $\mathcal{F}$. We also
set $\La (L) := \La (\mathcal{G}_L)$ and write $S_L$ and $S_L^*$ for the Ore sets
$S_{\mathcal{G}_L,\mathcal{H}_L}$ and $S_{\mathcal{G}_L,\mathcal{H}_L}^*$ that are defined in
\S\ref{sec: iwasawa-alg}.

\begin{remark}\label{extension}{\em  If $\mathcal{C}$ denotes either $\mathcal{F}$ or
$\mathcal{F}^+$, then it is an ordered set (by inclusion). Lemma
\ref{kl} implies that the subset $\mathcal{C}'$ of $\mathcal{C}$
comprising those fields $F$ for which $\mathcal{G}_F$ has no element
of order $p$ is cofinal. Taking account of the functorial properties
of the isomorphism in Theorem \ref{nc-weier} and of the equality in
Proposition \ref{char-el}(ii)(b) we may therefore deduce the
following extensions of these results.

\begin{itemize}
\item[$\bullet$] There is a natural isomorphism of abelian groups
\[\varprojlim_{F \in \mathcal{C}}K_1({\La}(F)_{S_F^*})\cong
 \varprojlim_{F \in \mathcal{C}}K_0(\Omega (\mathcal{G}_F))\oplus
 \varprojlim_{F \in \mathcal{C}}K_0(\La (\mathcal{G}_F),\La (\mathcal{G}_F)_{S_F})\oplus
 \varprojlim_{F \in \mathcal{C}} {\rm im}(\lambda_F)
\]
where $\lambda_F$ is the natural homomorphism $K_1(\La (F)) \to
K_1(\La (F)_{S_F^*})$ and in each inverse limit the transition maps
are the maps induced by the homomorphism ${\La}(F) \to {\La}(F')$
for each $F'\subseteq F$.

\item[$\bullet$] Let $(x_F)_F$ be an element of $\varprojlim_{F \in
\mathcal{C}}K_0(\La (\mathcal{G}_F),\La (\mathcal{G}_F)_{S_F^*})$.
Then for each $F$ in $\mathcal{C}$ we define an element ${\rm
char}_{\mathcal{G}_F,\gamma}(x_F)$ of $K_1(\La
(\mathcal{G}_F)_{S^*_F})$ in the following way: we choose $F'$ in
 $\mathcal{C}'$ with $F \subseteq F'$ and let ${\rm
char}_{\mathcal{G}_F,\gamma}(x_F)$ denote the image of ${\rm
char}_{\mathcal{G}_{F'},\gamma}(x_{F'})$ under the natural
projection $K_1(\La (\mathcal{G}_{F'})_{S^*_{F'}})\to K_1(\La
(\mathcal{G}_F)_{S^*_F})$. Then Lemma \ref{bclemma} implies ${\rm
char}_{\mathcal{G}_F,\gamma}(x_F)$ is independent of the precise
choice of $F'$ and Proposition \ref{char-el} implies
 that $\partial_{\mathcal{G}_F}({\rm char}_{\mathcal{G}_F,\gamma}(x_F)) =
x_F$.
\end{itemize}}
\end{remark}

\section{Non-commutative main conjectures}\label{ncmcs}

In this section we formulate explicit `main conjectures of
non-commutative Iwasawa theory' for both Tate motives and (certain)
critical motives. In particular, in the setting of elliptic curves,
the conjecture we formulate here is finer than that formulated by
Coates et al in \cite{cfksv} in that we consider interpolation
formulas for the leading terms (rather than merely the values) of
$p$-adic $L$-functions at Artin representations.

Henceforth we will fix an isomorphism of fields $j:\bc \cong
\bc_p$ and often omit it from the notation.

\subsection{Tate motives}

We fix a finite Galois extension $E$ of $\bq$, with $G :=
\Gal(E/\bq)$, and a finite set of places $\Sigma$ of $\bq$ that
contains the archimedean place and all places that ramify in
$E/\bq$. For each character $\chi$ in ${\rm Irr}_\bc(G)$ we write
$L_\Sigma(s,\chi)$ for the Artin $L$-function of $\chi$ that is
truncated by removing the Euler factors attached to primes in
$\Sigma$ (cf.\ \cite[Chap. 0, \S4]{tate}).

We also set $E_\infty := \br \otimes_\bq E \cong
\prod_{\Hom(E,\bc)}\br$ and write $\log_\infty(\O_E^\times)$ for the
inverse image of $\O_E^\times \hookrightarrow E_\infty^\times$ under
the (componentwise) exponential map $\exp_\infty: E_\infty \to
E_\infty^\times.$ We set $E_0 := \{ x \in E: {\rm Tr}_{E/\bq}(x) =
0\}$. The Dirichlet Unit Theorem implies that
 $\log_\infty(\O_E^\times)$ is a lattice in $\br \otimes_\bq E_0$ and
so there is a canonical
 isomorphism of $\bc [G]$-modules $\mu_\infty: \bc \otimes_\bz\log_\infty(\O_E^\times) \cong \bc\otimes_\bq E_0$.
In addition, if we write $S_p(E)$ for the set of $p$-adic places of
$E$, then the composite homomorphism
\begin{multline*}\bz_p\otimes_\bz\log_\infty(\O_E^\times) \xrightarrow{{\rm
exp}_\infty} \bz_p\otimes_\bz\O_E^\times \\\xrightarrow{}
\prod_{w\in S_p(E)} U^1_{E_w} \xrightarrow{(u_w)_w \mapsto
(\log_p(u_w))_w} \prod_{w\in  S_p(E)} E_w \cong \bq_p\otimes _\bq
E\end{multline*}
(where the second arrow is the natural diagonal map) factors through
the inclusion $\bq_p\otimes_{\bq}E_0 \subset \bq_p\otimes_{\bq}E$
and hence induces a homomorphism of $\bc_p[G]$-modules $\mu_p: \bc_p
\otimes_\bz\log_\infty(\O_E^\times) \cong \bc_p\otimes_\bq E_0$. We
set
\[ \Omega_j(\rho) := {\rm det}_{\bc_p}(\mu_p\circ
(\bc_p\otimes_{\bc,j}\mu_\infty)^{-1})^{\rho } \in \bc_p .\]
%

\begin{conjecture}\label{mctm} Fix a field $K$ in $\mathcal{F}^+$ that is
unramified outside a finite set of places $\Sigma$ and is such that
$\G := \mathcal{G}_K$ has no element of order $p$. Then the Galois
group $X_\Sigma(K)$ of the maximal pro-$p$ abelian extension of $K$
that is unramified outside $\Sigma$ belongs to
$\mathfrak{M}_{S^*}(\G)$. Further, there exists an element $\xi$ of
$K_1(\La (\G)_{S^*})$ which satisfies both of the following
conditions.

\begin{itemize} \item[(a)] At each Artin character $\rho$ of $G$ one
has $\xi(\rho) = c_{\rho,\gamma}^{-1}\Omega_j(\rho)L_\Sigma^*(1,\rho
 )^j$ with $c_{\rho,\gamma} := 1$ if either $\rho$ is trivial or
$\mathcal{H}_K\not\subset\ker(\rho)$ and $c_{\rho,\gamma} := 1
-\rho(\gamma^{-1})$ otherwise.

\item[(b)] $\partial_\G(\xi) = [X_\Sigma (K)].$
\end{itemize}
\end{conjecture}

\begin{remark}{\em For each character $\rho$ in ${\rm Irr}_{\bc_p}(G)$ the
`($S$-truncated) $p$-adic Artin $L$-function' of $\rho$ is the
unique $p$-adic meromorphic function $L_{p,S}(\cdot, \rho): \bz_p
\to \bc_p$ with the property that for each strictly negative integer
$n$ and each isomorphism $j:\bc _p\cong \bc$ one has
$\,L_{p,\Sigma}(n,\rho)^j = L_\Sigma(n,(\rho\cdot\omega^{n-1})^j )$
where $\omega: \mathrm{G}(\overline{\bq}/\bq) \to \bz_p^\times$ is
the Teichm\"uller character. The `$p$-adic Stark conjecture at
$s=1$' formulated by Serre in \cite{ser} asserts that the leading
term at $s=1$ of $L_{p,\Sigma}(s,\rho)$ is equal to
$\Omega_j(\rho)L_\Sigma^*(1,\rho
 )^j$.
See \cite[Rem. 5.3]{BV} for more
 details.}\end{remark}


See Conjecture \ref{mctmgc} for a version of Conjecture \ref{mctm}
which does not assume that $\mathcal{G}_F$ contains no element of
order $p$.

\subsection{Critical motives}

\subsubsection{Preliminaries}\label{critical preliminaries} Let $M$ be a critical motive over $\bq$ that has good ordinary reduction at $p$.  Then its $p$-adic
realization $V=M_p$ has a unique $\qp$-subspace $\hat{V}$ which is stable under the action of
$G_\qp$ such that $\,D^0_{dR}(\hat{V})=t_p(V):=D_{dR}(V)/D^0_{dR}(V)$. Now let $\rho$ be an Artin
representation defined over the number field $F$ and $[\rho]$ the corresponding Artin motive. We
fix a $p$-adic place $\lambda$ of $F$, set $L:= F_\lambda$ and write $\O$ for the valuation ring of
$L$. Then the $\lambda$-adic realisation   \be\label{Wdef}
W:=W_\rho:=N_\lambda=V\otimes_{\qp}[\rho]^*_\lambda\ee   of the motive $N:=M(\rho^*):=M\otimes
[\rho]^*$ is an $L$-adic representation and contains the $G_\qp$-subrepresentation
$\hat{W}=\hat{V}\otimes_{\qp}[\rho]^*_\lambda.$   The {\em algebraic rank} of $M(\rho^*)$ is
defined as
\be\label{algrk} r(M)(\rho):=\dim_L (H_f^1(\bq,W_\rho ))-\dim_L
(H_f^3(\bq,W_\rho)).\ee
Let $\Sigma$ be a finite set of places of $\bq$ containing $p,$
$\infty$ and all places  at which $M$ has bad reduction   or which
ramify in $K/\bq.$ Fix a field $K$ in $\mathcal{F}$ that is
unramified outside $\Sigma$ and write $G := G(K/\bq)$ for the
corresponding Galois group. By  $\Upsilon$ we denote the set  of
those primes $\ell\neq p$ such that the ramification index of $\ell$
in $K/\bq$ is infinite.


For a $F$-motive $N$ over $\bq$ we denote by $\Omega_\infty(N)$ and
$\Omega_p(N)$ the complex and $p$-adic periods, by $R_p(N)$ and
$R_\infty(N)$ the complex and $p$-adic regulator, see again
\cite{fukaya-kato} or \cite[Th.\ 6.5]{BV}. Recall that
$\Omega_\infty(N)\neq 0,$ if $N$ is critical, and that
$R_\infty(N)\neq 0,$ if the (complex) height pairing of $N$ is
non-degenerated. Furthermore, for a $\qp$-linear continuous
$G_\qp$-representation $Z$ we write $\Gamma(Z)$ for its
$\Gamma$-factor (loc.\ cit.). Finally, for any $L$-linear continuous
representation $V$ and prime number $\ell$ we define an element of
the polynomial ring $L[u]$ by setting
\[ P_\ell(V,u) :=
P_{L,\ell}(V,u) := \begin{cases} \det_L(1-\varphi_\ell u| V^{I_\ell}), &\text{if $\ell\neq p,$}\\
\det_L(1-\varphi_p u| D_{cris}(V)), &\text{if $\ell = p$,} \end{cases}\]
where $\varphi_\ell$ denotes the geometric Frobenius automorphism of
$\ell.$

As shown by Fukaya and Kato in \cite[Th.\ 4.2.26]{fukaya-kato}, the
behaviour of local $\epsilon$-factors implies that $p$-adic
$L$-functions can exist only after a suitable extension of scalars.
 To describe this we must assume that
\begin{equation}\label{cond-ram} \begin{cases} &\text{the maximal absolutely abelian subfield $K^{ab,p}$ of $K$ in which}\\
 &\text{$p$ is unramified is finite}.\end{cases}\end{equation}
Under this hypothesis we let $A$ denote the ring of integers of the
completion at any of the places of the field $K^{ab,p}$ that have
residue characteristic $p$. We set $\Lambda_A(G):=\La(G)\otimes_\zp
A$ and $\La_A(G)_S:=\La(G)_S\otimes_\zp A$ and write $\partial_{A,G}$ for the corresponding connecting homomorphism in $K$-theory.\\

\subsubsection{Elliptic curves} We first consider the case of the motive $M=h^1(E)(1)$ of an
elliptic curve $E$ over $\bq$ with good ordinary reduction at $p$
with $K=\bq(E(p))$ being the extension of $\bq$ which arises  by
adjoining the $p$-power division points and we assume that $G$ does
not contain any element of order $p.$ In this situation the
formulation of a (refined) main conjecture is very explicit since
 one can work with the dual $X(E/K)$ of the ($p$-primary) Selmer
group; later we will give another formulation for general critical
motives involving Selmer complexes. In the present situation
 one knows that the condition (\ref{cond-ram}) is satisfied (cf. \cite[just before Conj.\ 5.7]{cfksv}) and also that, if $\sha(E_{/K^{\ker(\rho)}})$ is finite, then
\[ r(M)(\rho) = \dim_{\mathbb{C}_p}(e_{\rho^*}(E(K^{\ker(\rho)})\otimes_\mathbb{Z}
\mathbb{C}_p)).\]
Upon combining the leading term computations of \cite[Th. 6.5]{BV}
with \cite[Th. 4.2.22]{fukaya-kato} and the general approach of
\cite{cfksv} we are led to formulate the following conjecture.

\begin{conjecture}\label{mcelliptic} 
Fix a field $K$ in $\mathcal{F}$ that is unramified outside a finite
set of places $\Sigma$ and is such that $\G := \mathcal{G}_K$ has no
element of order $p$. Then, under the above conditions, the module
$X(E/K)$ belongs to $\M_{S^*}(G)$. Further, there exists an element
$\L=\L(E) $ of $K_1(\La_A(G)_{S^*})$  which satisfies both of the
following conditions:

\begin{itemize} \item[(a)] At each Artin character $\rho$ of $G$, the value at $T = 0$
of $\,T^{-r(M)(\rho)}\Phi_\rho(\L)$ is equal to
\[(-1)^{r(M)(\rho)}\frac{L_{F,\Upsilon}^*(M(\rho^*))}{\Omega_\infty(M(\rho^*))R_\infty(M(\rho^*))}\cdot\Omega_p(M(\rho^*))R_p(M(\rho^*))\cdot
 \frac{P_{L,p}(\hat{W}_\rho^*(1),1)}{P_{L,p}(\hat{W}_\rho,1) },\]
where $L_{F,\Upsilon}^*(M(\rho^*))$ is the leading coefficient at $s=0$ of the complex $L$-function
of $M(\rho^*),$ truncated by removing Euler factors for all primes in  $\Upsilon.$

\item[(b)] $\partial_{A,G}(\L) = [\La_A(G)\otimes_{\La(G)} X(E/K)].$
\end{itemize}
\end{conjecture}

\begin{remark}\label{expinter}{\em The interpolation formula in Conjecture \ref{mcelliptic}(a) can of course also be stated in terms of
the classical Hasse-Weil $L$-functions and their twists
$L(E,\rho^*,s)$ in the sense of \cite[(102)]{cfksv} (which is the
same as the $L$-function attached to the $F$-motive $h^1(E)\otimes
[\rho]^*$); due to the shift one now has to consider the leading
term $ L^*(E,\rho^*)$ of $L(E,\rho^*,s)$ at $s=1$. Moreover, one can
simplify the above expression and make it more explicit. To this end
 we let $u$ in $\mathbb{Z}_p$ be the unit root of the polynomial
$1-a_pX+pX^2$ where, as usual, $p+1-a_p=\#\tilde{E}_p(\mathbb{F}_p)$
with $\tilde{E}_p$ denoting the reduction of $E$ modulo $p.$
Furthermore we write $p^{f_\rho}$ for the $p$-part of the conductor
of $\rho$ and $\epsilon_p(\rho)$ for the local $\epsilon$-factor of
$\rho$ at the prime $p.$ Moreover, let $d_+(\rho)$ and $d_-(\rho)$
denote the dimension of the subspace of $[\rho]_\lambda$ on which
complex conjugation acts by $+1$ and $-1,$ respectively. We denote
the periods of $E$ by
\[\Omega_+(E):=\int_{\gamma^+} \omega,\;\;\Omega_-(E):=\int_{\gamma^-}
\omega\]
where $\omega$ is the N\'{e}ron differential and $\gamma^+$ and
$\gamma^-$ denote  a generator for the subspace of
$\H_1(E(\mathbb{C}),\mathbb{Z})$  on which complex conjugation
acts as $+1$ and $-1$ respectively. Finally, we write
$R_\infty(E,\rho^*)$ and  $R_p(E,\rho^*)$ for the complex and
$p$-adic regulators of $E$ twisted by $\rho^*.$ Then the displayed
expression in Conjecture \ref{mcelliptic}(a) is equal to
\begin{multline}\label{exp-version}(-1)^{\dim_{\mathbb{C}_p}(e_{\rho^*}(E(K^{\ker(\rho)})\otimes_\mathbb{Z}
\mathbb{C}_p))}\frac{L_{R}^*(E,\rho^*)}{\Omega_+(E)^{d_+(\rho)}\Omega_-(E)^{d_-(\rho)}R_\infty(E,\rho^*)}\\
\times \epsilon_p(\rho) u^{-f_\rho} R_p(E,\rho^*)
 \frac{P_{L,p}([\rho]_\lambda,u^{-1})}{P_{L,p}([\rho]_\lambda^*,up^{-1} ) }.\end{multline}
Here $L_{R}^*(E,\rho^*)$ is the leading coefficient at $s=1$ of the $L$-function $L_R(E,\rho^*,s)$
obtained from the Hasse-Weil $L$-function of $E$ twisted by $\rho^*$ by removing the Euler factor
at $p$ and the Euler factors at all primes $\ell$ with $\mathrm{ord}_\ell(j_E)<0,$ where $j_E$
denotes the $j$-invariant of $E.$ See \cite[Rem.\ 4.2.27]{fukaya-kato} with $u=\alpha$ for the
calculation of $\Omega_p(M(\rho^*)).$}
\end{remark}

Before stating the next result we recall that the explicit interpolation formula given in the main
conjecture of \cite[Conj.\ 5.8]{cfksv} requires minor modification. To be precise, one must
interchange all occurrences of $\rho$ and $\hat{\rho}$ on the right hand side of the equality of
[loc.\ cit., (107)] except for the term
 `$e_p(\rho)$' (for further details see
 the footnote at the end of \S6.0 in \cite{ven-BSD}). 

%

\begin{prop} Assume the hypotheses of Conjecture \ref{mcelliptic}. Assume also that
for all Artin representations $\rho$ of $G$ the `order of vanishing part' of
 the Birch and Swinnerton-Dyer Conjecture for $E(\rho^*)$ holds.
 Then Conjecture \ref{mcelliptic} implies the `main conjecture of non-commutative Iwasawa theory' of
 \cite[Conj.\ 5.8]{cfksv} (modified as above). \end{prop}

\begin{proof} In this case $ \Upsilon$ is the set comprising
the prime $p$ and all prime numbers $q$ with ${\rm ord}_q(j_E)<0$
(see also \cite[4.5.3]{fukaya-kato} or \cite[Rem.\ 6.5]{ven-BSD}).
 In view of Remark \ref{expinter} the only essential difference between the two
conjectures is therefore that Conjecture \ref{mcelliptic} involves
an interpolation formula for ($(r(M)(\rho)!)^{-1}$ times) the
value at $T = 0$ of the $r(M)(\rho)$-th derivative of
$\Phi_\rho(\L)$ rather than merely for the value at $T =0$ of
$\Phi_\rho(\L)$ itself as in \cite[Conj.\ 5.8]{cfksv}. (Note that
the conjectured value at $T = 0$ of $T^{-r(M)(\rho)}\Phi_\rho(\L)$
should be the leading term at $T = 0$ of $\Phi_\rho(\L)$  only if
$R_p(M(\rho^*))\neq 0$.)

At the outset we note that $r(\Phi_\rho(\L))\geq
\dim_{\mathbb{C}_p}(e_{\rho^*}(E(K^{\ker(\rho)})\otimes_\mathbb{Z}
\mathbb{C}_p))\geq 0$ because the given interpolation formula has
no pole. In particular, $\L$ does not have $\infty$ as its value
at any $\rho.$ We also note that the `order of vanishing part' of
the Birch and Swinnerton-Dyer Conjecture for $E(\rho^*)$ implies
that the order of vanishing of $L_R(E,\rho^*,s)$ at $s=1$ is equal
to
$\dim_{\mathbb{C}_p}(e_{\rho^*}(E(K^{\ker(\rho)})\otimes_\mathbb{Z}
\mathbb{C}_p))$.

We now assume that $e_{\rho^*}(E(K^{\ker(\rho)})\otimes_\mathbb{Z}
\mathbb{C}_p)$ vanishes. Then both $R_p(M(\rho^*))=1 $ and
$R_\infty(M(\rho^*))=1$. Also, the leading term
$L_{R}^*(E,\rho^*)$ is in this case equal to the value at $s=1$ of
$L_R(E,\rho^*,s)$. Hence, the interpolation formula
(\ref{exp-version}) coincides with that given in \cite[Conj.\
5.8]{cfksv}.

On the other hand, if
 $e_{\rho^*}(E(K^{\ker(\rho)})\otimes_\mathbb{Z}\mathbb{C}_p)\neq 0,$
then $r(\Phi_\rho(\L)) > 0$ and so the value of $\mathcal{L}$ at
$\rho$ is equal to $0$. In addition, in this case the function
$L_R(E,\rho^*,s)$ vanishes at $s=1$ and so the interpolation
formula of \cite[Conj. 5.8]{cfksv} also implies that the value of
$\mathcal{L}$ at $\rho$ is equal to $0$, as required.
\end{proof}

\subsubsection{The general case} We return to the more general case discussed in \S\ref{critical preliminaries}.
 We fix a full Galois stable
$\bz_p$-sublattice $T$ of $V$ and define a $G_\qp$-stable
$\zp$-sublattice of $\hat{V}$ by setting $\hat{T}:=T\cap \hat{V}. $
As before we let $\T$ denote the Galois representation
$\La(G)\otimes_\zp T$ and set $\hat{\T}:=\La(G)\otimes_\zp \hat{T}$
similarly. Then $\hat{\T}$ is a $G_\qp$-stable  $\La (G)$-submodule
of $\T.$ For the definition of the Selmer complex ${SC}_U := {
SC}_U(\hat{\T},\T ),$ which is originally due to
 Nekov\'a\v r \cite{nek}, we refer the reader to either \cite[4.1.2]{fukaya-kato} or \cite[(31)]{BV}.\\


\begin{conjecture}[General formulation for critical motives]\label{mccm} 
Fix a field $K$ in $\mathcal{F}$ that is unramified outside a finite
set of places $\Sigma$ and is such that $\G := \mathcal{G}_K$ has no
element of order $p$. Then, under the above conditions, the
 complex ${\rm SC}_U$ belongs to $D^{\rm p}_{S^*}(\La (G))$. Further, there exists an element $\xi=\xi(U,M) $
of $K_1(\La_A(G)_{S^*})$ which satisfies both of the following
conditions:
\begin{itemize} \item[(a)] At each Artin character $\rho$ of $G$ for which
$P_{L,p}(\hat{W}_\rho,1)\neq 0\neq P_{L,p}({W}_\rho,1)$ the value at
$T = 0$ of $\,T^{-r(M)(\rho)}\Phi_\rho(\xi)$ is equal to
\[(-1)^{r(M)(\rho)}\frac{L_{F,\Sigma}^*(M(\rho^*))}{\Omega_\infty(M(\rho^*))R_\infty(M(\rho^*))}\cdot\Omega_p(M(\rho^*))R_p(M(\rho^*))\cdot\Gamma(\hat{V})^{-1}
\cdot \frac{P_{L,p}(\hat{W}_\rho^*(1),1)}{P_{L,p}(\hat{W}_\rho,1)
},\]
where $L_{F,\Sigma}^*(M(\rho^*))$ is the leading coefficient of
the complex $L$-function of $M(\rho^*),$ truncated by removing Euler
factors for all primes in  $ \Sigma\setminus \{ \infty\}.$

\item[(b)] $\partial_{A,G}(\xi) = [\La_A(G)\otimes_{\La(G)}{  SC}_U].$
\end{itemize}
\end{conjecture}
\medskip

\begin{prop} Conjecture \ref{mcelliptic} is compatible with Conjecture
\ref{mccm}.
\end{prop}

\begin{proof} To be precise, we prove that if $M=h^1(E)(1)$ and $K=\bq(E(p))$ are as in Conjecture
\ref{mcelliptic}, then Conjecture \ref{mccm} is equivalent to
Conjecture \ref{mcelliptic}.

First note that \cite[Prop.\ 4.3.7]{fukaya-kato} implies that the
complex ${\rm SC}_U $ belongs to $ D^{\rm p}_{S^*}(\La (G))$
precisely when the module $X(E/K)$ belongs to
$\frak{M}_H(G)=\frak{M}_{S^*}(G)$. Also, ${\rm SC}_U$ differs from
the complex ${\rm SC}(\hat{\T},\T)$ in loc.\ cit.\ only by local
terms which belong to $\frak{M}_{S^*}(G)$ (by \cite[Prop.\
4.3.6]{fukaya-kato}) and have characteristic elements (denoted
$\zeta(l,K/\bq)$ in loc. cit.) that correspond to the Euler-factors
$P_{L,l}(W_\rho,s)$ and whose values $P_{L,l}(W_\rho,1)$ at $\rho$
are neither $0$ or $\infty$ (by \cite[Lem.\ 4.2.23]{fukaya-kato}).
To deduce the
 claimed result from here one need only note that $\Gamma_{\bq_p}(\hat{V})=1$ in this case and recall
 (from \cite[Prop.\ 4.3.15-18]{fukaya-kato}) that the class of ${\rm SC}(\hat{\T},\T)$ in
$K_0( \frak{M}_{S^*}(G))$ is equal to $[X(E/K)]$.\end{proof}

\section{Equivariant Tamagawa numbers}\label{etncs}

Let $K/k$ be a finite Galois extension of number fields with $\G
:= \Gal(K/k)$. Then for any motive $M$ defined over $k$ the
equivariant Tamagawa number conjecture of \cite[Conj.\
4.1(iv)]{bufl01} asserts the vanishing of an element $T\Omega
(M_K,\bz [\G ])$ of $K_0(\bz [\G],\br [\G ])$ that is constructed
from the various realisations and comparison isomorphisms
associated to the motive $M_K:= K\otimes_k M$. Here $M_K$ is
regarded as defined over $k$ and endowed with a natural left
action of $\bq [\G]$ (via the first factor).

Now the product over all primes $p$ and all field isomorphisms
$j:\bc \cong \bc_p$ of the composite homomorphism
\[K_0(\bz [\G],\br [\G])
 \rightarrow K_0(\bz[\G],\bc [\G]) \xrightarrow{j} K_0(\bz [\G],\bc_p[\G]) \rightarrow
 K_0(\bz_p[\G],\bc_p[\G])\]
is injective (cf.\ \cite[Lem. 2.1]{br-bu2}). To prove \cite[Conj.\
4.1(iv)]{bufl01} it therefore suffices to prove that $j_*(T\Omega
(M_K,\bz [\G]))=0$ for each such $j$. This reduction has the
further advantage that the element $j_*(T\Omega (M_K,\bz [\G]))$
can be directly defined without assuming the `Coherence
Hypothesis' of \cite[\S3.3]{bufl01} that is necessary to define
$T\Omega (M_K,\bz [\G])$ (cf.\ \cite[Rem. 8]{bufl01}). However,
even if one assumes the standard compatibility conjectures
concerning the definition of Euler factors (cf.\ \cite[Conj.\
3]{bufl01}), the definition of $j_*(T\Omega (M_K,\bz [\G ]))$ is
in general still conditional, being dependent upon the conjectural
existence of a fundamental exact sequence relating the motivic
cohomology spaces of $M_K$ and its Kummer dual \cite[Conj.\
1]{bufl01} and of canonical $p$-adic Chern class isomorphisms
\cite[Conj.\ 2]{bufl01}. In particular, since we are  assuming
here that the element $j_*(T\Omega (M_K,\bz [\G]))$ is
well-defined, the results that we prove in this section will not
shed any new light on either of \cite[Conj.\ 1, Conj.\ 2]{bufl01}.

\subsection{Tate motives}\label{tmdf}

We fix a finite totally real Galois extension $E$ of $\bq$ and set
$G := \Gal(E/\bq)$. We recall that all of the conjectures
necessary for the definition of $T\Omega(h^0(\Spec E)(1),\bz [G])$
are known to be valid and so $j_*(T\Omega(h^0(\Spec E)(1),\bz
[G]))$ is defined unconditionally as an element of $K_0(\bz_p
[G],\bc_p [G])$. For a discussion of various explicit consequences
of the vanishing of $T\Omega(h^0(\Spec E)(1),\bz [G])$ see
\cite{br-bu2, br-bu3} and \cite{burnsov}.

%
%

\begin{thm}\label{tatedescent} Let $K$ be any field which belongs to $\mathcal{F}^+$, contains $E$ and is such that
 $\Gal(K/\bq)$ has no element of order $p$ (such a field $K$ exists by Lemma \ref{kl}). If $K$ validates Conjecture
 \ref{mctm} and $E$ validates Leopoldt's Conjecture
 (at $p$), then one has $j_*(T\Omega(h^0(\Spec E)(1),\bz
[G])) = 0$.\end{thm}

\begin{proof} We set $\mathcal{G} := \Gal(K/\bq)$ and $\mathcal{H} := \Gal(K/\bq^{\rm cyc})$. We also write $\Lambda(\mathcal{G})^\#(1)$ for the
$\Lambda(\mathcal{G})$-module $\Lambda(\mathcal{G})$ endowed with
the following action of $\Gal(\bar{\bq}/\bq)$: each $\sigma$ in
$\Gal(\bar{\bq}/\bq)$ acts on $\Lambda(\mathcal{G})^\#(1)$ as
multiplication by the element $\chi_{\rm
cyc}(\bar{\sigma})\bar{\sigma}^{-1}$ where $\bar{\sigma}$ denotes
the image of $\sigma$ in $\mathcal{G}$ and $\chi_{\rm cyc}$ is the
cyclotomic character $\mathcal{G} \to \Gamma \to \bz_p^\times$.
Then, following Fukaya and Kato \cite[\S2.1.1]{fukaya-kato} and
Nekov\'a\v r \cite{nek}, we obtain an object of $D^{\rm p}(\La
(\mathcal{G}))$ by setting
\[ C_K := R\Gamma_c(\Spec(\bz[1/\Sigma])_{\rm \acute e
t},\Lambda(\mathcal{G})^\#(1))\]
(where the subscript `$c$' denotes cohomology with compact support).
It is straightforward to show that $C_K$ is acyclic outside degrees
$2$ and $3$ and that its cohomology in degree $2$, respectively $3$,
is canonically isomorphic to $X_\Sigma(K)$, respectively $\bz_p$.
The first claim of Conjecture \ref{mctm} is therefore equivalent to
asserting that $C_K$ belongs to $D^{\rm p}_{S^*}(\La
(\mathcal{G}))$. In addition, Conjecture \ref{mctm}(b) asserts that
$\partial_\mathcal{G}(\xi) = [C_K] +[\bz_p]$, or equivalently by
Proposition \ref{char-el}(ii)(a) and (b) that
\be\label{reintermc} \partial_\mathcal{G}(\xi ') = [C_K]\ee
with $\xi ':= {\rm char}_{\mathcal{G},\gamma}(\bz_p[0])\cdot\xi$.

\begin{lem}\label{keytml} Assume that Leopoldt's conjecture is valid for $E$ (at $p$) and fix an Artin
representation $\rho: \mathcal{G} \to {\rm GL}_n(\mathcal{O})$ such
that $V_\rho$ belongs to ${\rm Irr}_{\bc_p}(G)$.

\begin{itemize}
\item[(i)] Conjecture \ref{mctm}(a) implies that $(\xi')^*(\rho) = \Omega_j(\rho)L_\Sigma^*(1,\rho
 )^j.$

\item[(ii)] If $\rho$ is non-trivial, then the complex $\bq_p\otimes^{\mathbb{L}}_{\La (\Gamma)}C_{K,\rho}$ is
acyclic. Further, the complex $C_K$ is $\rho$-semisimple,
$r_\mathcal{G}(C_K)(\rho) =0$ and $t(C_{K,\rho})$ is the canonical
morphism.

\item[(iii)] If $\rho$ is trivial, then the complex $\bz_p\otimes^{\mathbb{L}}_{\La (\Gamma)}C_{K,\rho}$ is acyclic outside degrees $2$ and $3$ and its
cohomology in degrees $2$ and $3$ identifies with
$\bq_p\otimes_{\bz_p}{\rm cok}(\lambda_E)$ and $\bq_p$, where
$\lambda_E$ is the localisation map
 $\mathcal{O}_E[\frac{1}{p}]^\times \otimes \bz_p \to \prod_{w \in
S_p(E)}E_w^\times\hat\otimes\bz_p$. Further, the complex $C_K$ is
$\rho$-semisimple, $r_\mathcal{G}(C_K)(\rho) =1$ and
$(-1)^{r_\mathcal{G}(C_K)(\rho)}t(C_{K,\rho})$ is induced by the
isomorphism $\beta: \bq_p\otimes_{\bz_p}{\rm cok}(\lambda_E) \to
\bq_p$ which sends each element $(e_w)_w$ of $\prod_{w\in
S_p(E)}E_w^\times$ to $\log_p(\chi_{\rm
cyc}(\gamma))^{-1}\sum_w\log_p({\rm N}_w(e_w))$ with ${\rm N}_w$ the
norm map $E_w^\times \to \bq_p^\times$.
 \end{itemize}
\end{lem}

\begin{proof} Claims (ii) and (iii) are both proved in \cite[Lem. 5.1]{BV}. Indeed, one need only note that
the minus sign in the expression $-c_\gamma^{-1}$ which occurs in
the formula of \cite[Lem. 5.1(iii)]{BV} cancels against the factor
$(-1)^{r_\mathcal{G}(C_K)(\rho)}$ in the expression
$(-1)^{r_\mathcal{G}(C_K)(\rho)}t(C_{K,\rho})$ in claim (iii). (But
see also Remark \ref{correction} below).

To prove claim (i) we note that Leopoldt's Conjecture implies that
the period $\Omega_j(\rho)$, and hence also the product
$c_{\rho,\gamma}^{-1}\Omega_j(\rho)L_\Sigma^*(1,\rho
 )^j$ in Conjecture \ref{mctm}(a), is non-zero. The latter conjecture thus implies that
  $\xi^*(\rho) = \xi(\rho) = c_{\rho,\gamma}^{-1}\Omega_j(\rho)L_\Sigma^*(1,\rho
 )^j$. Next we claim that
\be\label{expform} \Phi_\rho({\rm
char}_{\mathcal{G},\gamma}(\bz_p[0])) =
\begin{cases} 1-\rho(\gamma^{-1})(1+T)^{-1},&\text{if $\mathcal{H}
\subset \ker(\rho)$,}\cr 1,&\text{otherwise.}
\end{cases}\ee
If this is true, then it is clear that ${\rm
char}_{\mathcal{G},\gamma}(\bz_p[0])^*(\rho) = c_{\rho,\gamma}$.
Claim (i) would thus follow from the obvious equalities
$(\xi')^*(\rho) = ({\rm
char}_{\mathcal{G},\gamma}(\bz_p[0])\cdot\xi)^*(\rho) = {\rm
char}_{\mathcal{G},\gamma}(\bz_p[0])^*(\rho)\xi^*(\rho)$.

It thus suffices to prove (\ref{expform}). To do this we regard the
tensor product $M_\rho:= \La (\Gamma) \otimes_{\bz_p} {\rm
M}_n(\mathcal{O})$ as a
$(\La_\mathcal{O}(\Gamma),\La(\mathcal{H}))$-bimodule, where the
(left) action of $\La_\mathcal{O}(\Gamma)$ is clear and the (right)
action of each element $h$ of $\mathcal{H}$ is via $x\otimes y
\mapsto x\otimes y\rho(h)$. Then the definition of ${\rm
char}_{\mathcal{G},\gamma}(\bz_p[0])$ combines with the definition
of $\Phi_\rho$ to imply that
\begin{multline}\label{er} \Phi_\rho({\rm char}_{\mathcal{G},\gamma}(\bz_p[0]))\\ = {\rm
det}_{Q(\mathcal{O}[[T]])}({\rm id}\otimes{\rm id} - {\rm id}\otimes\theta\mid
Q(\mathcal{O}[[T]])\otimes_{\La_\mathcal{O}(\Gamma)}(M_\rho\otimes_{\La
(\mathcal{H})}\bz_p))\end{multline}
where $\theta$ is the endomorphism of $M_\rho\otimes_{\La
(\mathcal{H})}\bz_p = (\La (\Gamma) \otimes_{\bz_p} {\rm
M}_n(\mathcal{O}))\otimes_{\La (\mathcal{H})}\bz_p$ which sends each
element $(x\otimes y)\otimes z$ to $(x\gamma^{-1}\otimes
y\rho(\tilde\gamma^{-1}))\otimes z$ (this recipe is independent of
the choice of lift $\tilde \gamma$ of $\gamma$ through
$\mathcal{G}\to \Gamma$).

Now $V_\rho$ is irreducible and $\mathcal{H}$ is normal in
$\mathcal{G}$ and so the space $M_\rho\otimes_{\La
(\mathcal{H})}\bq_p$ is either zero or canonically isomorphic to
$M_\rho \otimes_{\bz_p}\bq_p$ depending on whether
 $\mathcal{H}\subset\ker(\rho)$ or not. In particular, if $\mathcal{H}\not\subset \ker(\rho)$,
 then (\ref{er}) implies that $\Phi_\rho({\rm
char}_{\mathcal{G},\gamma}(\bz_p[0]))$ is the determinant of an
endomorphism of the zero space and so equal to $1$. On the other
hand, if $\mathcal{H}\subset \ker(\rho)$, then $n=1$
 (since $\Gamma$ is abelian and $V_\rho$ is irreducible) and so (\ref{er}) implies $\Phi_\rho({\rm
char}_{\mathcal{G},\gamma}(\bz_p[0]))$ is the determinant of the
endomorphism of $Q(\mathcal{O}[[T]])$ given by multiplication by
$1-\rho(\gamma^{-1})(1+T)^{-1}$. The required equality
(\ref{expform}) is therefore clear.
\end{proof}

\begin{remark}\label{correction}{\em All of the results of \cite[\S5]{BV} are valid as stated without assuming
 that the condition \cite[(19)]{BV} is satisfied. Indeed, the latter condition is only introduced in loc.\ cit.\ to
 simplify the proofs of \cite[Lem. 5.1(ii),(iii)]{BV}. However, these results can be verified more directly as
 follows (we use the notation of loc.\ cit.). If $\rho$ is non-trivial, then \cite[Lem. 5.1(i)]{BV} implies that
 $\bz_p\otimes^{\mathbb{L}}_{\La(\Gamma)}R\Gamma(U,\mathbb{T})_\rho$ is acyclic and so \cite[Lem. 3.13(ii), (iii)]{BV} makes it
 clear that $R\Gamma(U,\mathbb{T})$ is semisimple at $\rho$, that $r_G(R\Gamma(U,\mathbb{T}))(\rho) =0$ and that
  the Bockstein homomorphism described in \cite[Lem. 5.1(iii)]{BV} is the unique
  endomorphism of the zero
  space. If $\rho$ is the trivial representation, then $R\Gamma(U,\mathbb{T})_\rho$
  identifies with $R\Gamma(U,\mathbb{T}_{\bq_{\rm cyc}}) \in D^{\rm p}(\La (\Gamma))$
  and so \cite[Lem. 3.13(ii), (iii), (iv)]{BV} implies that all of the claims made in
  \cite[Lem. 5.1(ii),(iii)]{BV} for $\rho$ can be verified after replacing $E$ by $\bq$
  (which certainly satisfies the condition \cite[(19)]{BV}).}\end{remark}

Returning to the proof of Theorem \ref{tatedescent} we set
\[C_E :=
 \bz_p[G]\otimes_{\La(\mathcal{G})}^{\mathbb{L}}C_K \cong R\Gamma_c(\Spec(\bz[1/\Sigma])_{\rm \acute e
t},\bz_p[G]^\#(1))\in
 D^{\rm p}(\bz _p[G]).\]
Then upon combining the conjectural equality (\ref{reintermc}) with
the explicit descriptions in Lemma \ref{keytml} and the result of
Theorem \ref{lt-result} one finds that
\be\label{almost2}
\partial_{\overline{G}}((\Omega_j(\rho)L_S^*(1,\rho )^j)_{\rho\in
{\rm Irr}(G)}) = -[\d_{\bz_p[\overline{G}]}(C_E), \beta_*] \ee
where $\beta_*$ is the morphism $\d_{\bc_p[\G ]}(\bc_p[\G
]\otimes_{\bz_p[\G ]}^{\mathbb{L}} C_E) \to \eins_{\bc_p[\G ]}$
induced by the isomorphism
\[ \bq_p\otimes_{\bz_p}H^2(C_E) \cong
\bq_p\otimes_{\bz_p}{\rm cok}(\lambda_E) \xrightarrow{\beta}
\bq_p\cong
 \bq_p\otimes_{\bz_p}H^3(C_E)\]
coming from the descriptions in Lemma \ref{keytml}(ii), (iii). But
the proof of \cite[Th. 5.5]{BV} shows that (\ref{almost2}) is
equivalent to an equality of the form
$\,[\d_{\bz_p[\overline{G}]}(C_E), \beta_* '] = 0\,$ where $\beta_*
'= (\beta '_\rho)_{\rho\in {\rm Irr}(G)}$ (under the identification
(\ref{m-e-decomp})) and each $\beta '_\rho$ is the morphism
described in \cite[(25)]{BV}. The fact that (\ref{almost2}) implies
the vanishing of the element $j_*(T\Omega(h^0(\Spec E)(1),\bz [G]))$
  then follows directly from the definition of $T\Omega(h^0(\Spec E)(1),\bz [G])$ and of each morphism $\beta '_\rho$.
  This therefore completes the proof of Theorem \ref{tatedescent}.
\end{proof}


We end this subsection by noting that the above computations show
that the correct generalisation of Conjecture \ref{mctm} (to groups
with an element of order $p$) is the following.

\begin{conjecture}\label{mctmgc} Fix a field $K$ in $\mathcal{F}^+$ that is
unramified outside a finite set of places $\Sigma$. Then $C_K$
belongs to $D^{\rm p}_{S^*}(\La (\mathcal{G}_K))$.
 Further, there exists an element $\xi'$ of $K_1(\La (\G)_{S^*})$
which satisfies both of the following conditions.
\begin{itemize} \item[(a)] At each Artin character $\rho$ of $G$ one
has $\xi'(\rho) = \Omega_j(\rho)L_\Sigma^*(1,\rho )^j.$
\item[(b)] $\partial_\G(\xi') = [C_K].$
\end{itemize}
\end{conjecture}

This conjecture is compatible with that formulated (in the case
that $\mathcal{G}_K$ has rank one) by Ritter and Weiss in
\cite[\S4]{RW2}. For details see \cite{br-bu3}.

\subsection{Critical Motives}\label{critsec}


We now assume the notation and hypotheses of Conjecture \ref{mccm}
and fix a subfield $E$ of $K$ which is Galois over $\Q.$ We set
$\overline{G}:=G(E/\Q)$ and
$\hat{\mathbb{T}}_E:=\La(\overline{G})\otimes_\zp \hat{T}\cong
\La(\overline{G})\otimes_{\La(G)}\hat{\mathbb{T}}.$ We fix a number
field $F$ over which all representations $\rho$ in
$\mathrm{Irr}(\overline{G})$ can be realised. 
We write $Z=Z_\rho$ and $\tilde{Z}$ for the Kummer dual
$W_\rho^*(1)$ and $\hat{W}^*(1)$ of $W_\rho$  and $\hat{W},$
respectively; finally we set $\tilde{W}:=W/\hat{W}.$ In terms of the
notation of \cite{BV} we consider the following assumption on $W$.
\medskip

{\bf Assumption (W):} For each $\rho$ in
$\mathrm{Irr}(\overline{G})$ the space $W=W_\rho$ satisfies all of
the following conditions:-
\begin{enumerate}
\item[(A$_1$)]  $P_\ell(W,1)P_\ell(Z,1)\neq0$ for all primes
$\ell\not= p,$
\item[(B$_1$)] $P_p(W,1)P_p(Z,1)\neq0,$
\item[(C$_1$)] $P_p(\tilde{W},1)P_p(\tilde{Z},1)\neq0$ and
\item[(D$_2$)] $H^0_f(\Q,W)= H^0_f(\Q,Z)=0$.
\end{enumerate}
\medskip

In the following result we write $\iota_A:
K_0(\zp[\overline{G}],\bc_p[\overline{G}])\to
K_0(A[\overline{G}],\bc_p[\overline{G}])$ for the canonical
homomorphism obtained by regarding $A$ as a subring of $\bc_p$. We
also recall that \cite[Conj. 4(iii)]{bufl01} (which is a natural
 equivariant version of the Deligne-Beilinson Conjecture) for the motive $M_E$,
regarded as defined over $\bq$ and with an action of $\bq
[\overline{G}]$, implies that the element
$j_*(T\Omega(M_E,\mathbb{Z} [\overline{G}]))$ belongs to the
subgroup $K_0(\zp[\overline{G}],\mathbb{Q}_p[\overline{G}])$ of
 $K_0(\zp[\overline{G}],\bc_p[\overline{G}])$.

\begin{thm}\label{MCtoETNC} Assume that
\begin{itemize}
\item[$\bullet$] Assumption {\rm (W)} is valid;
\item[$\bullet$] the complex ${\rm SC}_U$ is semi-simple at all
$\rho$ in $\mathrm{Irr}(\overline{G});$
\item[$\bullet$] an $\epsilon$-isomorphism
$\epsilon_{p,\La(\overline{G})}(\hat{\mathbb{T}}_E) :
\u_{\Lambda(\overline{G})}\to
\d_{\Lambda(\overline{G})}(\r(\qp,\hat{\mathbb{T}}_E))\d_{\Lambda(\overline{G})}(\hat{\mathbb{T}}_E)$
in the sense of \cite[Conj.\ 3.4.3]{fukaya-kato} exists;
\item[$\bullet$] Conjecture \ref{mccm} is valid for the motive $M$
and the extension $K/\Q$.
\end{itemize}
Then $\iota_A(j_*(T\Omega(M_E,\mathbb{Z} [\overline{G}]))) =0$. If
$j_*(T\Omega(M_E,\mathbb{Z} [\overline{G}]))$ belongs to
$K_0(\zp[\overline{G}],\mathbb{Q}_p[\overline{G}])$, then also
$j_*(T\Omega(M_E,\mathbb{Z} [\overline{G}]))=0$.
\end{thm}

\begin{proof} We fix an element $\xi$ as in Conjecture \ref{mccm}. Since ${\rm SC}_U$ is
semisimple at each $\rho$ in $\mathrm{Irr}(\overline{G})$, the
obvious analogue of Theorem \ref{lt-result} with $A$ in place of
$\zp$ combines with Conjecture \ref{mccm}(b) to imply that
\[\partial_{\overline{G}}( (\xi^*(\rho))_{\rho\in \mathrm{Irr}(\overline{G})})= -\iota_A([\d_{\bz_p[\overline{G}]}({\rm SC}_U(\hat{\mathbb{T}}_E,\T_E)),
t({\rm SC}_U(\hat{\mathbb{T}}_E,\T_E))_{\bar{G}}]).\]
%
After unwinding the identification (\ref{caniso}), this means that there exists a morphism in
$V(A[\overline{G}])$
\[\psi:\u_{A[\overline{G}]}\to
\d_{A[\overline{G}]}(A[\overline{G}]\otimes_{\zp[\overline{G}]}{\rm
SC}_U(\hat{\mathbb{T}}_E,\T_E))\]
such that
\[ (\xi^*(\rho)^{-1})_{\rho\in
 \mathrm{Irr}(\overline{G})} = t({\rm SC}_U(\hat{\mathbb{T}}_E,\T_E))_{\bar{G}}\circ
\psi_{\mathbb{C}_p[\overline{G}]} \in
\mathrm{Aut}_{V(\mathbb{C}_p[\overline{G}])}(\u_{\mathbb{C}_p[\overline{G}]})\cong
K_1(\mathbb{C}_p[\overline{G}])\]
under the identification \eqref{m-e-decomp}. After recalling the
explicit definition of
 $t({\rm SC}_U(\hat{\mathbb{T}}_E,\T_E))_{\bar{G}}$ given in Theorem
 \ref{lt-result} and then taking inverses
we obtain a morphism in $V(A[\overline{G}])$
\[\psi^{-1}:\u_{A[\overline{G}]}\to \d_{A[\overline{G}]}(A[\overline{G}]\otimes_{\bz_p[\overline{G}]}
 {\rm SC}_U(\hat{\mathbb{T}}_E,\T_E))^{-1}\]
such that
\[ (-1)^{r_G({\rm SC}_U)(\rho)} \xi^*(\rho) = t({\rm SC}_U(\rho^*))^{-1}\circ
\psi^{-1}(\rho)\in \mathrm{Aut}_{V(\mathbb{C}_p )}(\u_{\mathbb{C}_p
})\cong
 \mathbb{C}_p^\times\]
for all $\rho$ in $\mathrm{Irr}(\overline{G}).$ Here we write $\psi^{-1}(\rho)$ for the
$\rho$-component of the morphism $\u_{\bc_p[\overline{G}]}\to
\d_{\bc_p[\overline{G}]}(\bc_p[\overline{G}]\otimes_{\mathbb{Z}_p[\overline{G}]}{\rm
SC}_U(\hat{\mathbb{T}}_E,\T_E))^{-1}$ induced by $\psi^{-1}$. This is equivalent to asserting the
existence of a morphism in $V(A[\overline{G}])$
\[\psi':\u_{A[\overline{G}]}\to
\d_{A[\overline{G}]}(A[\overline{G}]\otimes_{\bz_p[\overline{G}]}\r_c(U,\T_E))^{-1}\]
such that for all $\rho$ in $\mathrm{Irr}(\overline{G})$ the
composite morphism
\begin{multline}\label{keycomp} {\u_{\mathbb{C}_p}}
\xrightarrow{\psi'(\rho)_{\mathbb{C}_p}}
{\d_L(\r_c(U,W_\rho))^{-1}_{\mathbb{C}_p}}
\xrightarrow{\beta(\rho)\epsilon(\hat{\T})^{-1}(\rho)}
{\d_L({\rm SC}_U(\hat{W}_\rho,W_\rho))^{-1}_{\mathbb{C}_p}}\\
\xrightarrow{t({\rm SC}_U(\rho^*))_{\mathbb{C}_p}^{-1}}
\u_{\mathbb{C}_p}\end{multline}
corresponds to $(-1)^{r_G({\rm SC}_U)(\rho)} \xi^*(\rho)$. In this
displayed expression we write
 $\epsilon(\hat{\T})(\rho)$ for
$V_{\rho^*}\otimes_{\zp[\overline{G}]}
\epsilon_{p,\La(\overline{G})}(\hat{\mathbb{T}}_E) $ and
$\beta(\rho)$ for $V_{\rho^*}\otimes_{\zp[\overline{G}]}
(\zp[\overline{G}]\otimes_{\La(G)}\beta)\cong V_{\rho^*}
\otimes_{\La(G)}\beta$ with $\beta$ the morphism $\d_{\La}
(\mathbb{T}^+)_{\tilde\La}\cong
\d_{\La}(\hat{\mathbb{T}})_{\tilde\La}$ defined in \cite[(35)]{BV},
and all underlying identifications are as explained in \cite[\S
6]{BV}.

Now the hypothesis that ${\rm SC}_U$ is semisimple at $\rho$
combines with the assumption (W), the duality isomorphism
$H^3_f(\bq,W)\cong H^0_f(\bq,Z)$ and the results of \cite[Lem. 6.7
and Lem. 3.13(ii)]{BV} to imply that the algebraic rank
$r(M)(\rho)$ defined in (\ref{algrk}) is equal to $r_G({\rm
SC}_U)(\rho)$, that \cite[Condition (F)]{BV} is satisfied and that
the value at $T = 0$ of $\,T^{-r(M)(\rho)}\Phi_\rho(\xi)$ is equal
to the leading term $\xi^*(\rho)$. Conjecture \ref{mccm}(a)
therefore gives an explicit formula for $\xi^*(\rho)$. Taking this
formula into account, one can compare the composite morphism
(\ref{keycomp}) to the first
 displayed morphism after \cite[Lem.\ 6.8]{BV}. After unwinding the
proof of \cite[Th. 6.5]{BV} (for which we use assumption (W)) this
comparison shows that
\[ \psi'(\rho) = \vartheta_\lambda(M(\rho^*))_{\mathbb{C}_p}\circ
\zeta_K(M(\rho^*))_{\mathbb{C}_p}\]
for all $\rho$ in $\mathrm{Irr}(\overline{G})$, where
 $\vartheta_\lambda(M(\rho^*))_{\mathbb{C}_p}$ and
$\zeta_K(M(\rho^*))_{\mathbb{C}_p}$ are the morphisms that occur in
 \cite[Conj. 4.1]{BV}. Finally we note that the validity of the last displayed equality
 (for all $\rho$ in $\mathrm{Irr}(\overline{G})$) is equivalent to
  asserting that the element $\iota_A(j_*(T\Omega(M_E,\mathbb{Z}[\overline{G}])))$ vanishes (by the very
definition of the latter element). This proves the first claim of
  the theorem.

Given this, the second claim of Theorem \ref{MCtoETNC} will follow
if we can show that the natural composite homomorphism
$K_0(\bz_p[\overline{G}],\Q_p[\overline{G}])\to
K_0(\bz_p[\overline{G}],\bc_p[\overline{G}])\to
K_0(A[\overline{G}],\bc_p[\overline{G}])$ is injective. We write
 $F$ for the field of fractions of $A$ (so $F \subset \bc_p$). Then the natural homomorphism
 $K_0(A[\overline{G}],F[\overline{G}])\to
K_0(A[\overline{G}],\bc_p[\overline{G}])$ is clearly injective and
so it suffices to prove that the natural homomorphism
$K_0(\bz_p[\overline{G}],\bq_p[\overline{G}])\to
K_0(A[\overline{G}],F[\overline{G}])$ is also injective. But,
since $A/\bz_p$ is unramified, this is an immediate consequence of
a result of M. Taylor \cite[Chap. 8, Th.
 1.1]{taylor}. Indeed, one need only note that the groups $K_0(\bz_p[\overline{G}],\bq_p[\overline{G}])$ and
 $K_0(A[\overline{G}],F[\overline{G}])$ are naturally isomorphic
 to the groups $K_0T(\bz_p[\overline{G}])$ and
 $K_0T(A[\overline{G}])$ which occur in loc. cit. \end{proof}

\begin{remark}{\em If $M = h^1(A)$ for an abelian variety $A$ which
has good ordinary reduction at $p$ and is such that the
Tate-Shafarevich group $\sha(A_{/E})$ of $A$ over $E$ is finite,
then the vanishing of $j_*(T\Omega(M_E,\mathbb{Z}[\overline{G}]))$
implies the $p$-part of a Birch and Swinnerton-Dyer type formula
(see, for example, \cite[\S 3.1]{ven-BSD}). However, Conjecture
\ref{mccm} does {\em not} itself imply that $\sha(A_{/E})$ is
finite.}
\end{remark}
%
%

\begin{remark}{\em Explicit consequences of Conjecture \ref{mccm} for the
values (at $s =1$) of twisted Hasse-Weil $L$-functions have been
described by
 Coates et al in \cite{cfksv}, by Kato in \cite{kato-k1} and by Dokchister and Dokchister
 in \cite{Doks}. However, all of the consequences described in \cite{cfksv, Doks, kato-k1} become trivial when the
 $L$-functions vanish. One of the key advantages of Theorem \ref{MCtoETNC} is that in many of these cases it
 can be combined with the
 approach of \cite{burnsov} to show that Conjecture \ref{mccm} implies
  a variety of explicit (and highly non-trivial) congruence relations between
 values of {\em derivatives} of twisted Hasse-Weil $L$-functions.
 Such explicit (conjectural) congruences will be the subject of a
 subsequent article.}\end{remark}

\begin{remark}{\em Following Theorem \ref{MCtoETNC} it is of
 some interest to study elements in $K$-theory of the form
$T\Omega(h^1(E_{/K})(1),\bz[\Gal(K/\bq)])$ with $E$ an elliptic
curve over $\bq$ and $K/\bq$ a finite non-abelian Galois
extension. The study of such elements is however still very much
in its infancy. Indeed, the only explicit computation that we are
currently aware of is the following. Let $E$ be the elliptic curve
$y^2+y = x^3-x^2 -10x -20$ (this is the curve 11A1 in the sense of
Cremona \cite{cremona}). Then, with $K$ equal to the splitting
field of the polynomial $x^3-4x-1$, the group $\Gal(K/\bq)$ is
dihedral of order $6$ and Navilarekallu \cite{tejaswi} has proved
numerically that if $\sha(E_{/K})$ is trivial, then the element
$T\Omega(h^1(E_{/K})(1),\bz[\Gal(K/\bq)])$ vanishes. }\end{remark}

\appendix

\section{Determinant functors} In this appendix we recall the formalism of determinant functors
 introduced by Fukaya and Kato in \cite{fukaya-kato} and used in \cite{BV} (see also \cite{ven-BSD})

For any ring $R$ we write ${ B}(R)$ for the category of bounded
complexes of (left) $R$-modules, ${  C}(R)$ for the  category of
bounded complexes of finitely generated (left) $R$-modules, $P(R)$
for the category of finitely generated projective (left)
$R$-modules, $ {C}^{\rm p}(R)$ for the category of bounded
(cohomological) complexes of finitely generated projective (left)
$R$-modules. By $D^{\rm p}(R)$ we denote the category of perfect
complexes as full triangulated subcategory of the derived category
$D^b(R)$ of the homotopy category of $B(R).$ We write $(C^{\rm
p}(R),{\rm quasi})$ and $(D^{\rm p}(R),{\rm is})$ for the
subcategory of quasi-isomorphisms of $C^{\rm p}(R)$ and isomorphisms
of $D^{\rm p}(R),$ respectively.

For each complex $C = (C^\bullet,d_C^\bullet)$ and each integer $r$
we define the $r$-fold shift $C[r]$ of $C$ by setting $C[r]^i=
C^{i+r}$ and $d^i_{C[r]}=(-1)^rd^{i+r}_C$ for each integer $i$.

We first recall that for any (associative unital) ring $R$ there
exists a Picard category $\C_R$ and a determinant functor $\,\d_R:(
{C}^{\rm p}(R),{\rm quasi})\to \C_R$ with the following properties
(for objects $C,C'$ and $C''$ of $\mathrm{C}^{\rm p}(R)$)

\begin{itemize}
\item[A.d)]\footnote{The listing starts with d) to be compatible
with the notation of  \cite{ven-BSD}.} If $0\to C'\to C\to C''\to 0$
is a short exact sequence of complexes, then there is a canonical
isomorphism $\,\d_R(C)\cong\d_R(C')\d_R(C'').$
which we take as an identification.
\item[A.e)] If $C$ is acyclic, then the quasi-isomorphism $0\to C$
induces a canonical isomorphism $\,\u_R\to\d_R(C).$
\item[A.f)] For any integer $r$ one has
$\d_R(C[r])=\d_R(C)^{(-1)^r}$.
\item[A.g)] the functor $\d_R$ factorizes over the image of
$C^{\rm p}(R)$ in the category of perfect complexes $D^{\rm p}(R),$
and extends (uniquely up to unique isomorphisms) to $(D^{\rm
p}(R),{\rm is}).$
\item[A.h)] For each $C$ in $D(R)$ we write $\H(C)$ for the
 complex which has $\H(C)^i = H^i(C)$ in each degree $i$ and in which all differentials are $0$. If
 $\H(C)$ belongs to $D^{\rm p}(R)$ (in which case one says that $C$ is {\em cohomologically perfect}), then there are canonical
 isomorphisms
\[\d_R(C) \cong \d_R(\H(C)) \cong \prod_{i\in \bz} \d_R(H^i(C))^{(-1)^i}.\]
(For an explicit description of the first isomorphism see
\cite[\S3]{knudsen} or \cite[Rem. 3.2]{br-bu}.)
\item[A.i)] If $R'$ is any further ring and $Y$ an $(R',R)$-bimodule
which is both finitely generated and projective as an $R'$-module,
 then the functor $Y\otimes_R-:\mathrm{P}(R)\to P(R')$ extends to a
commutative diagram
\[\begin{CD}
(D^{\rm p}(R), {\rm is}) @> \d_R >> \C_R\\
@V {Y\otimes_R^\mathbb{L}-}VV @VV Y\otimes_R- V\\
(D^{\rm p}(R'), {\rm is}) @> \d_{R'} >> \C_{R'}.
\end{CD}\]
In particular, if $R\to R'$ is a ring homomorphism and $C$ is in
$\mathrm{D}^{\rm p}(R),$ then we often write $\d_R(C)_{R'}$ in place
of $R'\otimes_R\d_R(C).$
\end{itemize}


%
In \cite{fukaya-kato} a {\em localized} $K_1$-group was defined for
any full subcategory $\Sigma$ of $C^{\rm p}(R)$ which satisfies the
following four conditions:

\begin{itemize}

\item[(i)] $0\in \Sigma,$

\item[(ii)] if $C,C'$ are in $C^{\rm p}(R)$ and $C$ is quasi-isomorphic
to $C',$ then $C\in \Sigma$ $\Leftrightarrow$ $C' \in \Sigma,$

\item[(iii)] if $C\in \Sigma,$ then also $C[n]\in\Sigma$ for all
$n\in \bz,$

\item[(iv$'$)]  if  $C'$ and $C''$ belong to  $\Sigma$, then $
C'\oplus C''$ belongs to $\Sigma$.
\end{itemize}

\begin{defn}{\em (Fukaya-Kato) Assume that $\Sigma$ satisfies (i), (ii), (iii) and (iv$'$). The {\em localized $K_1$-group} $K_1(R,\Sigma)$ is defined to be the (multiplicatively written)
abelian group which has as generators symbols of the form $[C,a]$
for each $C\in \Sigma$ and morphism $a:\u_R\to \d_R(C)$ in $\C_R$
and relations
\begin{itemize}
\item[(0)] $[0,\id_{\u_R}]=1,$
\item[(1)] $[C',\d_R(f)\circ a]=[C,a]$ if $f:C\to C'$ is an
quasi-isomorphism with $C$ (and thus $C'$) in $\Sigma,$
\item[(2)] if $0 \to C' \to C \to C'' \to 0$ is an exact sequence
in $\Sigma,$ then
\[ [C,a]=[C',a']\cdot [C'',a'']\]
where $a$ is the composite of $a'\cdot a''$ with the isomorphism
induced by property d),
\item[(3)] $[C[1],a^{-1}]=[C,a]^{-1}.$
\end{itemize}}
\end{defn}

We now assume given a left denominator set $S$ of $R$ and we let
$R_S:= S^{-1}R$ denote the corresponding localization and $\Sigma_S$
the full subcategory of $C^{\rm p}(R)$ consisting of all complexes
$C$ such that $R_S\otimes_R C$ is acyclic. For any $C$ in $\Sigma_S$
and any morphism $a: \u_{R} \to \d_{R}(C)$ in $\C_R$ we write
$\theta_{C,a}$ for the element of $K_1(R_S)$ which corresponds under
the canonical isomorphism $K_1(R_S) \cong {\rm
Aut}_{\C_{R_S}}(\u_{R_S})$ to the composite
\[ \u_{R_S} \xrightarrow{} \d_{R_S}(R_S\otimes_R C)
\rightarrow \u_{R_S}\]
where the first arrow is induced by $a$ and the second by the
 fact that $R_S\otimes_RC$ is acyclic. Then it can be shown that the assignment $[C,a] \mapsto
\theta_{C,a}$ induces an isomorphism of groups
\[ {\rm ch}_{R,\Sigma_S}: K_1(R,\Sigma_S) \cong K_1(R_S)\]
(cf.\ \cite[Prop.\ 1.3.7]{fukaya-kato}). Hence, if $\Sigma$ is any subcategory of $\Sigma_S$ we
also obtain a composite homomorphism
\[ {\rm ch}_{R,\Sigma}: K_1(R,\Sigma) \to K_1(R,\Sigma_S) \cong K_1(R_S).\]

In particular, we shall often use this construction in the following
case: $C \in \Sigma_S$ and $\Sigma$ is equal to smallest full
subcategory $\Sigma_{C}$ of  $C^{\rm p}(R)$ that contains $C$ and
also satisfies the conditions (i), (ii), (iii) and (iv$'$) that are
described above.

\section{Bockstein homomorphisms}\label{bockhom} Let $A$ be a noetherian regular ring
and assume given an exact triangle in $D^{\rm p}(A)$
\[ \Delta: \,\,\,\,\,\, C \xrightarrow{\theta} C \to D \to C
[1].\]
%
%
%
In each degree $i$ let $\beta^i_{\Delta}$ denote the composite
homomorphism
\[ H^{i}(D) \rightarrow \ker(H^{i+1}(\theta)) \to H^{i+1}(C) \to
\cok(H^{i+1}(\theta))\to H^{i+1}(D) \]
where the first and fourth maps occur in the long exact sequence of
cohomology of $\Delta$ and the second and third are tautological and
write
\[ \H_{\rm bock} (\Delta): \,\, \cdots
\xrightarrow{\beta^{i-1}_\Delta} H^{i}(D)
\xrightarrow{\beta^i_\Delta} H^{i+1}(D)
\xrightarrow{\beta^{i+1}_\Delta} \cdots \]
for the associated complex (with $H^i(D)$ is placed in degree $i$).
%
%
The morphism $\theta$ is said to be `semisimple' if the
tautological map $\ker(H^{i}(\theta)) \to \cok(H^i(\theta))$ is
bijective in each degree $i$. This condition is equivalent to
asserting the acyclicity of $\H_{\rm bock} (\Delta)$.
 Hence, if true, there is a composite morphism in $V(A)$ of the form
\begin{equation}\label{bockdef} \beta_{\Delta}: \d_{A}(D) \to
 \d_{A}(\H (D))\to \d_{A}(\H_{\rm bock}(\Delta))
 \to {\bf 1}_{V(A)}\end{equation}
where the first map is as in A.h), the second is the obvious
 map (induced by the fact that the complexes $\H(D)$ and
$\H_{\rm bock}(\Delta)$ agree termwise) and the third is induced by
the acyclicity of $\H_{\rm bock}(\Delta)$.

\section{Sign conventions in \cite{BV}}

Due to different normalisations, which had not been noticed by the
authors, the following sign conflict has arisen: in
\cite{fukaya-kato} for a discrete valuation ring $\O$ with field
of fractions $L$ an element $c\in \O\setminus{\{0\}}\subseteq
L^\times$ corresponds to the class $[ {
  {\O} \stackrel{c}{\to} {\O}   }, \id]$ in $K_1(L)$, where the complex is concentrated in degree $0$ and $1$ (while it is implicitly concentrated in degrees $-1$ and $0$  in
  \cite[Rem.\ 2.4]{BV}). With this convention, Fukaya and Kato must define the connecting homomorphism as $[C,a]\mapsto -[[C]]$ in order to
  ensure that the connecting homomorphism $L^\times=
  K_1(L)\to K_0(\Sigma_{\O\setminus{\{0\}}})\cong \mathbb{Z}$ coincides with the valuation $\mathrm{ord}_L$ (cf. \cite[Rem.\ 1.3.16]{fukaya-kato}). It follows that in \cite[Rem.\ 2.4]{BV} the correct formula is \[\mathrm{ord}_L(c)=-\mathrm{length_\O(A)},\]
if we identify $A$ with the complex $A[0].$ For the same reason, the
signs in \cite[Prop. 3.19]{BV} are also incorrect, the corrected
versions being
\[ \chi_{\rm add}(G,C(\rho^*))=-\mathrm{ord}_L(\L^*(\rho))\]
and
\[ \chi_{\rm mult}(G,C(\rho^*))=|\L^*(\rho)|_p^{[L:\qp]}.\]

We note also that $\L:=[C,a]\in K_1(\Lambda,\Sigma_{S^*})$ is a
characteristic element of $-[[C]]=[[C[1]]]$ (rather than of
$[[C]]$) in $K_0(\Sigma_{S^*})$ due to the normalisation of the
connecting homomorphism in \cite{fukaya-kato}. For a similar
reason we have to add a sign in the formulae of [loc. cit., (38)
and (39)] to obtain the corrected versions
\be \L_{U,\beta}:=\L_{U,\beta}(M):\u_\La\to \d_\La({\rm
SC}_U(\hat{\T},\T))^{-1}\ee
and
\be\label{Lconj} \L_{\beta}:=\L_{\beta}(M):\u_\La\to \d_\La({\rm
SC}(\hat{\T},\T))^{-1}\ee

Also in the following convention we need a shift by one: we write
 $\L_{U,\beta}$ and $\L_{\beta}$ for the elements $[{\rm SC}_U[1],\L_{U,\beta}]$ and $[{\rm SC}[1],\L_{\beta}]$ of
$K_1(\La(G),\Sigma_{{\rm SC}_U})$ and $K_1(\La(G),\Sigma_{{\rm
SC}})$ respectively. Finally, in the first displayed formula after
 \cite[Lem. 6.8]{BV} one has to replace $t({\rm
SC}_U(\rho^*))_{\tilde{L}}$ by $t({\rm
SC}_U(\rho^*)[1])_{\tilde{L}}=t({\rm
SC}_U(\rho^*))_{\tilde{L}}^{-1}.$

The authors would like to apologise for these oversights.

\Addresses
\end{document}